\documentclass[reqno,11pt]{amsart}

\usepackage{fullpage}

\usepackage{amssymb}
\usepackage{dsfont}
\usepackage[all]{xy}

\usepackage[utf8]{inputenc} 
\usepackage[T1]{fontenc} 

\usepackage{hyperref}

\usepackage[nobysame,alphabetic,initials,msc-links]{amsrefs}

\usepackage{enumerate}

\DefineSimpleKey{bib}{how}

\renewcommand{\eprint}[1]{#1}
\BibSpec{misc}{%
  +{}{\PrintAuthors}  {author}
  +{,}{ \textit}      {title}
  +{,}{ }             {how}
  +{}{ \parenthesize} {date}
  +{,} { available at \eprint}        {eprint}
  +{,}{ available at \url}{url}
  +{,}{ }             {note}
  +{.}{}              {transition}
}

\numberwithin{equation}{section}

\newenvironment{claim}[1]{\par\noindent\itshape\underline{Claim:}\space#1}{}

\theoremstyle{plain}
\newtheorem{theorem}{Theorem}[section]
\newtheorem{proposition}[theorem]{Proposition}
\newtheorem{lemma}[theorem]{Lemma}
\newtheorem{corollary}[theorem]{Corollary}

\newtheorem{rlemma}{Lemma}
\newenvironment{replemma}[1]
	{\renewcommand\therlemma{\ref*{#1}}\rlemma}
	{\endrlemma}

\theoremstyle{definition}
\newtheorem{definition}[theorem]{Definition}

\theoremstyle{remark}
\newtheorem{remark}[theorem]{Remark}

\newtheorem{condition}[theorem]{Condition}

\newcommand\bp{\begin{proof}}
\newcommand\ep{\end{proof}}

\newcommand{\un}{\mathds{1}}

\newcommand\Z{\mathbb{Z}}

\newcommand{\G}{\mathcal{G}}
\newcommand{\F}{\mathcal{F}}
\newcommand{\J}{\mathcal{J}}
\newcommand{\K}{\mathcal{K}}

\newcommand{\RR}{\mathcal{R}}

\newcommand{\Gu}{{\mathcal{G}^{(0)}}}
\newcommand{\Gxx}{\mathcal{G}^x_x}

\newcommand\con{\operatorname{con}}
\newcommand\dom{\operatorname{dom}}
\newcommand\ran{\operatorname{ran}}

\begin{document}

\title{Boundary quotients of C$^*$-algebras of left cancellative monoids and their groupoid models}

\date{July 2025}

\author{Gaute Schwartz}
\email{gautesc@uio.no}

\begin{abstract}
For a left cancellative monoid $S$ we consider a quotient of the reduced semigroup C$^*$-algebra $C_r^*(S)$ known as the boundary quotient. We present two potential groupoid models for this boundary quotient, obtained as reductions of Paterson and Spielberg's groupoids associated to $S$, and formulate conditions on $S$ which guarantees that either is a groupoid model. We outline how these conditions are related to the notions (strong) C$^*$-regularity introduced in a previous paper  \cite{NS}, and construct an example of a left cancellative monoid which is not C$^*$-regular, but satisfies both of the new conditions. 
\end{abstract}

\maketitle

\section*{Introduction}

In recent years, several researchers have studied boundary quotients of C*-algebras attached to semigroups, see for example \cite{Star}, \cite{LS} or \cite{KKLL}. Typically, the focus has been on semigroups that are either right LCM (least common multiple) monoids or group embeddable, where much can be said about the particular nature of the boundary quotients, e.g.  whether they are nuclear, simple or purely infinite. In this paper we take a broader point of view and define a boundary quotient for general left cancellative monoids. Every such monoid $S$ has a regular representation on~$\ell^2(S)$ and hence a well-defined reduced $C^*$-algebra $C_r^*(S)$. Inside this C*-algebra is a semilattice of projections corresponding to a certain family of subsets of $S$, $\J(S)$, called the constructible right ideals \cite{MR2900468}. To obtain the boundary quotient we then quotient $C_r^*(S)$ out by an ideal, $I$, which is defined in terms of foundation sets in $\J(S)$. The notion of foundation sets was introduced in \cite{SY} in the context of quasi-lattice orders and has been used in ~\cite{Star} and \cite{LS} to treat the right LCM and group embeddable case, respectively.\\

The goal of this paper is to find a suitable groupoid model for this boundary quotient $C_r^*(S)/I$. In a previous paper ~\cite{NS} we introduced conditions on $S$, which guaranteed, respectively, that the \'etale groupoids $\G_P(S)$ and $\G(S)$, were groupoid models for $C_r^*(S)$. The groupoid $\G_P(S)$ is a reduction of Paterson's universal groupoid (see \cite{MR1724106}) attached to the inverse hull, $I_\ell(S)$, of $S$ (see \cite{MR3200323}). The groupoid $\G(S)$ was introduced by Spielberg in \cite{MR4151331}, actually in the more general context of left cancellative small categories, and is a quotient of $\G_P(S)$. These groupoids have the same unit space which is identified with a subspace, denoted $\Omega(S)$, of the space filters on $\J(S)$. More precisely $\Omega(S)$ is the closure of the set of principal filters $\chi_s$ ($s\in S$) in the product topology. We always have left regular representations $\rho_{\chi_e}:C_r^*(\G_P(S))\rightarrow C_r^*(S)$ and $\rho_{\chi_e}:C_r^*(\G(S))\rightarrow C_r^*(S)$, where $e$ denotes the identity element in $S$. A sufficient condition for these maps to be isomorphisms, respectively, is that for the dense subset $Y=\{\chi_s \mid s\in S\}$ of $\Omega(S)$, every point $x\in \Omega(S)\setminus Y$ satisfies the following topological property for $\G= \G_P(S),\ \G(S)$, respectively,
\begin{equation}\label{eq:topproperty1}
\text{There is a net in } Y \text{ converging to } x\\
\text{ and having no other accumulation points in } \G_x^x,
\end{equation}

This property is equivalent to one introduced by Khoshkam and Skandalis~\cite{KhSk} where they give a condition which guarantees that all left regular representations are weakly contained in the set of regular representations determined by some given dense subset of a groupoid's unit space. There they show that for an inverse semigroup $I$, the dense subset $\{\rho_{\chi_g} \mid g\in I\}$ of $\G(I)$ (Paterson's universal groupoid associated to $I$), always satisfies this condition, but curiously the same is not true for the subset $\{\rho_{\chi_s} \mid s\in S\}$ in the groupoids $\G_P(S)$ or~$\G(S)$. Property \eqref{eq:topproperty1} for $\G=\G_P(S)$ and $\G(S)$ both have equivalent formulations in terms of $S$, namely strong C$^*$-regularity, see \cite[Definition~2.2]{NS} and C$^*$-regularity, see \cite[Definition~2.11]{NS}.\\

It turns out that the reduction of either $\G_P(S)$ or $\G(S)$ by the closure of the so called ultrafilters, denoted $\partial\Omega(S)$, are good candidates for a groupoid model of~$C_r^*(S)/I$. They appear in \cite{MR4151331} (in the more general context of  left cancellative small categories) where they are called boundary groupoids. For $\G_P(S)$ this reduction coincides with the tight groupoid, introduced by Exel in \cite{MR2419901}, of $I_\ell(S)$. There he showed that, for any inverse semigroup $\mathcal{S}$ the universal C*-algebra of the tight groupoid, $C^*(\G_{\text{tight}}(\mathcal{S}))$, is isomorphic to $C_{\text{tight}}^*(\mathcal{S})$, a C*-algebra generated by~$\mathcal{S}$ which is universal for so called tight representations of~$\mathcal{S}$. Since one can associate an inverse semigroup $I_\ell(S)$ to an arbitrary left cancellative monoid $S$ this indicated that the boundary quotient of $S$ is related to the $C^*$-algebra of an  \'etale groupoid. In \cite{Star}, Starling showed that for right LCM monoids $S$ then, at the level of full C*-algebras, the nicest relation one can hope for holds, namely, $Q(S)\cong C_{\text{tight}}^*(I_\ell(S))\cong C^*(\G_{\text{tight}}(I_\ell(S)))$. Here $Q(S)$ denotes a universal boundary quotient, see ~\cite[Definition~3.3]{Star}.\\

In investigating when $\partial\G_P(S):=(\G_P(S))_{\partial\Omega(S)}$ is a groupoid model for $C_r^*(S)/I$, we derive the following commutative diagram

\begin{equation}\label{eq:precommutativediagram}
\xymatrix{
C^*_r(\G_P(S))\ar[d]^{\phi} \ar[r]^{\rho_{\chi_e}} & C^*_r(S) \ar[d]^q\\
C^*_e(\partial\G_P(S)) \ar[r] & C_r^*(S)/I.
}
\end{equation}
Here $e$ is an exotic norm completion of $C_c(\partial\G_P(S))$ coming from the short exact sequence $$0 \rightarrow C_r^*(\G_P(S)\setminus \partial\G_P(S)) \rightarrow C_r^*(\G_P(S)) \xrightarrow[]{\phi} C_e^*(\partial\G_P(S)) \rightarrow 0.$$
It turns out that when the $e$ norm coincides with the reduced norm and~\eqref{eq:topproperty1} is satisfied for $\G=\G_P(S)$ and all $x\in \partial\Omega(S)$, then the bottom map in~\eqref{eq:precommutativediagram} is an isomorphism. For the groupoid $(\G(S))_{\partial\Omega(S)}$ we have an analogous exotic norm~$e_2$ and commutative diagram where the bottom map is an isomorphism if the $e_2$ norm coincides with the reduced norm and~\eqref{eq:topproperty1} is satisfied for $\G=\G(S)$ and all $x\in \partial\Omega(S)$. We deduce two conditions on $S$, strong C$^*$-regularity on the boundary which is equivalent to~\eqref{eq:topproperty1} being satisfied for $\G=\G_P(S)$ and all $x\in \partial\Omega(S)$, and C$^*$-regularity on the boundary which is equivalent to~\eqref{eq:topproperty1} being satisfied for $\G=\G(S)$  and all $x\in \partial\Omega(S)$. These conditions are reminiscent of their strong C$^*$-regularity and C$^*$-regularity counterparts, in fact they only differ in that a collection of sets that before were supposed to (set theoretically) cover something, are now only required to be a foundation set for that something.\\

This begs the question of when the $e$ norm, or $e_2$ norm, coincide with the reduced norm. It seems difficult to find general conditions on $S$ that imply this, much less are equivalent to it. However, in Section ~\ref{sec:example} we exhibit a monoid,~$R$, such that, (i) $R$ is strongly C$^*$-regular on the boundary, (ii) $\G_P(R)=\G(R)$, (iii) the boundary groupoid $\partial\G_P(R)=\partial\G(R)$ is Borel amenable (\cite[Definition~2.1]{R-2}) and (iv) the map $\rho_{\chi_e}: C_r^*(\G(R)) \rightarrow C_r^*(R)$ is not an isomorphism. Since the groupoid $\G:=\partial\G_P(R)=\partial\G(R)$ is $\sigma$-compact, it follows from a recent paper (\cite[Corollary~6.10]{BGHL}) that the canonical map $C^*(\G) \rightarrow C_r^*(\G)$ is an isomorphism. In particular the $e$ norm on $C_c(\partial\G(R))$ must be the reduced norm, hence the bottom map in~\eqref{eq:precommutativediagram} is an isomorphism in this case.\\

We begin the paper by introducing C$^*$-algebras of 
general étale groupoids, as well as the boundary quotient, 
in Section~\ref{sec:preliminaries}. In Section \ref{sec:boundary models} we define two 
potential groupoid models for the boundary quotient and 
give conditions on $S$ which guarantee that these are bona fide groupoid models. As explained, Section~\ref{sec:example} contains an example of a non-C$^*$-regular (see Definition \ref{def:regularity2}) left cancellative monoid which satisfies these conditions. In Section~\ref{sec:relations} conclude the main text by addressing the relations between the notions (strong) C$^*$-regularity, (strong) C$^*$-regularity on the boundary and the groupoid equalities $\G_P(\cdot)=\G(\cdot)$ and $\partial\G_P(\cdot)=\partial\G(\cdot)$. In particular we
show that the groupoids $\G_P(\cdot)$ and $\G(\cdot)$ need not coincide, thus answering a question posed by Xin Li in \cite{Li-1}, in the negative.\\

Let us finally mention that it might be interesting to look at these questions in the more general context of left cancellative small categories.  It would also be an interesting question for further study to find conditions on general left cancellative monoids which guarantee that the associated groupoids $\G_P(\cdot)$ or $\G(\cdot)$ are amenable. We feel that some of the techniques we used in proving the amenability of $\G_P(R)=\G(R)$ should hold more generally.\\

The author would like to thank Sergey Neshveyev for many helpful discussions as well as for introducing me to the problem.

\section{Preliminaries on Groupoid and Semigroup C$^*$-Algebras}\label{sec:preliminaries}

\subsection{C$^*$-algebras of non-Hausdorff \'Etale Groupoids} Let $\G$ be a locally compact, not necessarily Hausdorff, étale groupoid. By this we mean that $\G$ is a groupoid equipped with a locally compact topology such that

\begin{enumerate}[(I)]

\item the groupoid operations are continuous.
\item\label{ax:axiom2} the unit space $\Gu$ is Hausdorff in the relative topology.
\item\label{ax:axiom3} the source map $s: \G \rightarrow \Gu$ and the range map $r: \G \rightarrow \Gu$ are local homeomorphisms. 
\end{enumerate}
A subset $U\subseteq \G$ is called a bisection if the source and range maps are injective on $U$. Axioms~\eqref{ax:axiom2} and~\eqref{ax:axiom3} imply that bisections are Hausdorff in the relative topology. Axiom~\eqref{ax:axiom3} guarantees that~$\G$ has a basis of topology consisting of open bisections.\\

Let us associate a function space to $\G$. For each open Hausdorff set $V\subseteq \G$, consider the usual space $C_c(V)$ of continuous compactly supported complex-valued functions on $V$. Extending such functions by zero to $\G$, we may view $C_c(V)$ as a subspace of the space of functions on $\G$. Since in general $\G$ can be non-Hausdorff, these extensions need not be continuous on $\G$. We define the space

$$
C_c(\G):=\text{span}\{C_c(V)\mid V\subseteq \G,\ V\text{ open Hausdorff}\}.
$$
\\
Denoting $\G^x=r^{-1}(x)$ and $\G_x=s^{-1}(x)$ for a unit $x\in \Gu$, $C_c(\G)$ is a $*$-algebra with the convolution product
\begin{equation*} \label{eprod}
(f_{1}*f_{2})(g) := \sum_{h \in \G^{r(g)}} f_{1}(h) f_{2}(h^{-1}g)\quad \text{for}\quad g\in \G,
\end{equation*}
and involution $f^*(g)=\overline{f(g^{-1})}$. The full groupoid C$^*$-algebra $C^*(\G)$ is defined as the C$^*$-enveloping algebra of $C_c(\G)$.\\
For each $x\in\Gu$, define a \textit{left regular} representation $\rho_{x}\colon C_c(\G) \to B(\ell^{2}(\G_{x}))$ by
\begin{equation*}\label{eq:rhox}
(\rho_{x}(f)\xi)(g) :=\sum_{h \in \G^{r(g)}} f(h) \xi(h^{-1}g).
\end{equation*}
The reduced C$^*$-algebra $C^*_r(\G)$ is then defined as the completion of $C_c(\G)$ with respect to the norm
\begin{equation}\label{eq: reduced-formula}
\lVert f \rVert_{r} : = \sup_{x\in \Gu } \lVert \rho_{x}(f) \rVert , \qquad f\in C_c(\G).
\end{equation}

A central question in the previous paper \cite{NS} was when a subset $Y\subseteq \Gu$ is big enough to compute the reduced norm. That is, when can we replace $\Gu$ by $Y$ in Equation~\eqref{eq: reduced-formula}. This is equivalent to saying that for all $x\in\Gu$, the representation $\rho_x$ is weakly contained in $\bigoplus_{y\in Y}\rho_y$. We recall the following sufficient condition.

\begin{proposition}[{\cite[Proposition~1.4]{NS}}]\label{prop:weak-containmenttwo}
Let $\G$ be a locally compact étale groupoid, $Y\subseteq \Gu$ and $x\in \Gu\setminus Y$. Assume there is a net $(y_i)_i$ in $Y$ such that $x$ is the only accumulation point of $(y_i)_i$ in $\G_x^x=\G_x \cap \G^x$. Then the representation~$\rho_x$ of $C_r^*(\G)$ is weakly contained in $\bigoplus_{y\in Y}\rho_y$.
\end{proposition}

If $x\in\Gu\cap\bar Y$, the non-existence of a net as in Proposition~\ref{prop:weak-containmenttwo} is equivalent to the following property:

\begin{condition}[{\cite[Condition~1.6]{NS}}]\label{condition2}
For some $n\ge1$ there are elements $g_1,\dots,g_n\in\Gxx\setminus\{x\}$, open bisections $U_1,\dots,U_n$ such that $g_k\in U_k$ and a neighbourhood $U$ of $x$ in $\Gu$ satisfying\\ $Y\cap U\subseteq U_1\cup\dots\cup U_n$.
\end{condition}

Note that Condition \eqref{condition2} can only be satisfied for $x\in (\Gu\cap\bar Y)\setminus Y$. In summary, if $Y\subseteq \Gu$ is a dense subset such that Condition \ref{condition2} is not satisfied for any $x\in \Gu\setminus Y$, then for each $x\in \Gu$ the representation $\rho_x$ is weakly contained in $\bigoplus_{y\in Y}\rho_y$ and we have
$$
\|f\|_r=\sup\limits_{y\in Y}\|\rho_y(f)\|, \qquad f\in C_c(\G).
$$

For an invariant subset $X\subseteq \Gu$, let $\G_X=s^{-1}(X)=r^{-1}(X)$. $\G_X$ is a subgroupoid of $\G$ and with the relative topology it is again an étale groupoid. Similarly $\G\setminus \G_X=\G_{\Gu\setminus X}$ is an étale subgroupoid. When $X$ is closed we have a natural embedding $C_c(\G\setminus \G_X) \hookrightarrow C_c(\G)$ and a restriction map $C_c(\G) \rightarrow C_c(\G_X)$. These maps induce the following short exact sequence of full groupoid C$^*$-algebras.
\begin{equation}\label{eq: universal-short-exact}
0 \rightarrow C^*(\G\setminus \G_X) \rightarrow C^*(\G) \rightarrow C^*(\G_X) \rightarrow 0.
\end{equation}
At the level of reduced algebras there is a short exact sequence, \cite[Proposition~1.2]{CN-2}),
\begin{equation}\label{eq: reduced-short-exact}
0 \rightarrow C_r^*(\G\setminus \G_X) \rightarrow C_r^*(\G) \rightarrow C_e^*(\G_X) \rightarrow 0.
\end{equation}
where $e$ is a uniquely determined exotic C$^*$-norm on $C_c(\G_X)$. That $e$ is exotic means that $r\leq e\leq u$ where $u$ denotes the universal norm.

\subsection{The Boundary Quotient}
We now give an introduction into the C$^*$-algebras associated to left cancellative monoids. These are related to the inverse hull which we also introduce.\\

Let $S$ be a left cancellative monoid with identity element $e$. We have a left regular representation of $S$ given by
$$
\lambda\colon S\to B(\ell^2(S)),\quad \lambda_s\delta_t=\delta_{st}.
$$
The reduced C$^*$-algebra $C_r^*(S)$ of $S$ is defined as the C$^*$-algebra generated by the operators $\lambda_s$, $s\in S$. Many operators in $C_r^*(S)$ are best understood as coming from an underlying collection of partial bijections on $S$, namely the left inverse hull of $S$.\\

The left inverse hull of $S$, $I_\ell(S)$ is the inverse semigroup of partial bijections on $S$ generated by the left translations $S\rightarrow S$. Whenever convenient we will identify an element $s\in S$ with the left translation $t\mapsto st$. Similarly, we denote by $s^{-1}\in I_\ell(S)$ the bijection $sS\mapsto S$ inverse to $s$. Then any element of $I_\ell(S)$ can be written as a composition
$$s_{2n}^{-1}s_{2n-1}\cdots s_2^{-1}s_1$$
for some elements $s_1, \cdots , s_{2n}\in S$ and $n\in \mathbb{N}$. If the map with empty domain lies in $I_\ell(S)$ we denote it by $0$.  The left regular representation $\lambda$ can be extended to a representation of $I_\ell(S)$ by
$$
\lambda \colon I_\ell(S) \to C_r^*(S), \quad \lambda(s_{2n}^{-1}s_{2n-1}\cdots s_2^{-1}s_1)=\lambda_{s_{2n}}^*\lambda_{s_{2n-1}}\cdots \lambda_{s_{2}}^*\lambda_{s_{1}}.
$$
Now let $E(S)$ be the abelian semigroup of idempotents in $I_\ell(S)$. Every element of~$E(S)$ is the identity map on its domain of definition $X\subseteq S$, which is a right ideal of $S$ of the form
$$
X=s_{2n}^{-1}s_{2n-1}\cdots s_2^{-1}s_1S.
$$
for some $s_1, \ldots , s_{2n}\in S$. Right ideals that arise in this way are 
called constructible. We denote by~$\J(S)$ the collection of all 
constructible right ideals. It is a semigroup under intersection, and we have an 
isomorphism $E(S)\cong \J(S)$. Note that $\emptyset\in \J(S)$ if and only if $0\in I_\ell(S)$.\\

Denote by $p_X\in E(S)$ the idempotent 
corresponding to $X\in \mathcal{J}(S)$. For $X\in \mathcal{J}(S)$ we denote the 
operator $\lambda(p_X)\in C_r^*(S)$ by $\un_X$. Concretely we have
\begin{equation}\label{eq:indicator}
\un_X(\delta_s)=\begin{cases}
\delta_s & \text{ if }s\in X,\\
0 & \text{ if }s\notin X.
\end{cases}
\end{equation}
This formula makes sense more generally and for any subset $X\subseteq S$, $\un_X$ will denote the projection defined by~\eqref{eq:indicator}. In general $\un_X\notin C_r^*(S)$.\\

To define the boundary quotient of $C_r^*(S)$ we need a preliminary definition.
\begin{definition}[cf.~{\cite[Definition~6.5]{LS}, \cite[Definition~11.5]{MR2419901}}]\label{def:foundationset}
Let $X, X_1, \ldots , X_m\in \mathcal{J}(S)$ be constructible right ideals, $m\geq 1$. Then $\{X_1, \ldots , X_m\}$ is a \textbf{foundation set} for $X$ if
\begin{enumerate}[(I)]
  \item  $X_i\subseteq X$ for all $i$;
  \item for any nonempty constructible right ideal $Y\in \mathcal{J}(S)\setminus \{\emptyset\}$ satisfying $Y\subseteq X$, there is an index $i_0$ such that $Y\cap X_{i_0}\neq \emptyset$.
\end{enumerate}
\end{definition}

\begin{proposition}\label{prop:noncovering-of-foundation-differences}
Let $X\in \J(S)\setminus \{\emptyset\}$ be a nonempty constructible right ideal. Then $X$ is not 
contained in a finite union of differences of the form $X'\setminus \cup_{i=1}^m X_i$ where $X', X_1, \cdots , X_m\in\mathcal{J}(S)$ and $\{{X_1}, 
\cdots, {X_m}\}$ is a foundation set for $X'$.\\
\end{proposition}
\begin{proof}
Assume for contradiction that
\begin{equation}\label{eq:differencecovering}
X\subseteq \Big( X^1\setminus \bigcup\limits_{i=1}^{m_1} X_i^1 \Big) \cup \cdots \cup \Big( X^n\setminus \bigcup\limits_{i=1}^{m_n} X_i^n \Big)
\end{equation}
where all sets appearing on the RHS are constructible right ideals and $\{X_1^j, \cdots, X_{m_j}^j\}$ is a foundation set for $X^j$. 
Fix a maximal set $F\subseteq \{1, \dots, n\}$ for which there exists indices $1\leq i_j \leq m_j$, $j\in F$, such that $X\cap\bigcap_{j\in F}X_{i_j}^j \neq \emptyset$. If $F=\{1, \dots, n\}$, then this contradicts~\eqref{eq:differencecovering}. On the other hand, if $F\neq \{1, \dots, n\}$, then~\eqref{eq:differencecovering} implies that $A:=X\cap \bigcap_{j\in F}X_{i_j}^j$ has nonempty intersection with a difference $X^{j'}\setminus \cup_{i=1}^{m_{j'}} X_i^{j'}$ for some $j'\notin F$. Then $A\cap X^{j'}$ is a nonempty constructible ideal contained in $X^{j'}$ not intersecting $X_i^{j'}$ for any $i$. This contradicts $\{X_1^{j'}, \cdots, X_{m_{j'}}^{j'}\}$ being a foundation set for~$X^{j'}$.\\
\end{proof}

We wrap up this section by defining the boundary quotient.
\begin{definition}\label{def:BoundaryQuotient}
Let $S$ be a left cancellative monoid. The \textbf{boundary quotient of $C_r^*(S)$} is the quotient of $C_r^*(S)$ by the ideal
\begin{align}\label{eq:I}
I=\langle \un_{X\setminus \bigcup\limits_{i=1}^m X_i} \mid & \ X, X_1, \ldots ,X_m\in \mathcal{J}(S), \nonumber \\[-1.3ex]
& \{X_1, \ldots , X_m\} \text{ is a foundation set for } X,\ m\geq 1\rangle.
\end{align}
\end{definition}

\begin{remark}
The generators of $I$ lie in $C_r^*(S)$; if $X, X_1, \ldots ,X_m\in \mathcal{J}(S)$ we have\\ $\un_{X\setminus \bigcup\limits_{i=1}^m X_i}=\prod\limits_{i=1}^m (\un_X- \un_{X_i})$, showing that this projection lies in $C_r^*(S)$.
\end{remark}
Later we shall see how Proposition~\ref{prop:noncovering-of-foundation-differences} implies that $I$ is always a proper ideal, so that the boundary quotient $C_r^*(S)/I$ is always non-trivial.\\

\section{Boundary Quotients of Semigroup C$^*$-algebras}\label{sec:boundary models}
In this section we give two main candidates for groupoid models for the boundary quotient $C_r^*(S)/I$. The first candidate is a reduction of the groupoid $\G_P(S)$ introduced in Section $2$ of \cite{NS} which in turn is a reduction of Paterson's universal groupoid for $I_\ell(S)$ (\cite{MR1724106}). The second one is a reduction of Spielberg's groupoid $\G(S)$, see \cite{NS}.

\subsection{The Reduced Paterson Groupoid}
Consider the collection $\widehat{\J(S)}$ of non-zero characters on~$\J(S)$, that is, semigroup homomorphisms $\J(S) \rightarrow \{0, 1\}$ which are not identically zero, where~$\{0, 1\}$ is considered as a semigroup under multiplication. Then $\widehat{\J(S)}$ is compact Hausdorff in the topology of pointwise convergence and can be identified with the space of filters on $\J(S)$ under the correspondence $\J(S)\ni \chi \mapsto \chi^{-1}(1)$. For every $s\in S$, consider the character $\chi_s\in \widehat{\J(S)}$ corresponding to the principal filter on $s$:
$$
\chi_s(X)=\begin{cases}
1 \text{ if }s\in X,\\
0 \text{ if }s\notin X.
\end{cases}
$$

Consider the closure $\Omega(S):=\overline{\{\chi_s \mid s\in S\}}$. It is again compact Hausdorff in the topology of pointwise convergence. Alternatively $\Omega(S)$ can be described as the characters $\chi \in \widehat{\J(S)}$ satisfying the properties

\begin{enumerate}[(I)]
\item if $0\in I_\ell(S)$, or equivalently if $\emptyset\in \J(S)$, then $\chi(\emptyset)=0$;
\item if $\chi(X)=1$ and $X=X_1\cup \cdots \cup X_n$ for $X, X_1, \ldots ,X_n\in \J(S)$, then $\chi(X_i)=1$ for at least one index $i$.
\end{enumerate}
Let
$$
\G_P(S)=\Sigma/\sim, \quad \text{ where }\Sigma=\{(g, \chi)\in I_\ell(S)\times \Omega(S) \mid \chi(\dom g)=1\}
$$
and the equivalence relation $\sim$ is defined by declaring $(g_1, \chi_1)$ and $(g_2, \chi_2)$ in $\Sigma$ to be equivalent if and only if
$$
 \chi_1=\chi_2\ \ \text{and there exists}\ \ A\in \J(S)\ \ \text{such that}\ \ {g_1}_{|A}={g_2}_{|A}\ \ \text{and}\ \ \chi_1(A)=1.
$$
We denote by $[g,\chi]$ the class of $(g,\chi)\in\Sigma$ in $\G_P(S)$. The product is defined by
$$
[g,\chi]\,[h,\psi] = [gh,\psi]\quad\text{if}\quad \chi=\psi(h^{-1}(\cdot)).
$$
In particular, the unit space $\G_P(S)^{(0)}$ can be identified with $\Omega(S)$ via the bijection $\chi \mapsto [p_S,\chi]$. The source and range maps are given by
$$
s([g,\chi])=\chi,\qquad r([g,\chi])=\chi(g^{-1}(\cdot)),
$$
while the inverse is given by $[g,\chi]^{-1}=[g^{-1},\chi(g^{-1}(\cdot))]$.\\

The topology on $\G_P(S)$ is chosen to make the source and range maps local homeomorphisms. More precisely, for each $g\in I_\ell(S)$ and subset $U\subset \Omega(S)$, define
$$
D(g, U):=\{[g, \chi]\in \G_P(S)\mid \chi\in U\}.
$$
Then the topology on $\G_P(S)$ is defined by taking as a basis the sets $D(g, U)$, as $g$ ranges over $I_\ell(S)$ and $U$ ranges over open subsets of the clopen set $\{\chi~\in~\Omega(S)~\mid~\chi(\dom g)=1\}$. Then $\G_P(S)$ is a locally compact étale groupoid.\\

Recall that we have a bijection $S\cong (\G_P(S))_{\chi_e}$ via $s \mapsto [s, \chi_e]$ and that under this identification the left regular representation of $\chi_e$ becomes a surjective $*$-homomorphism \cite[Lemma 2.1]{NS},
$$
\rho_{\chi_e}: C_r^*(\G_P(S))\rightarrow C_r^*(S), \qquad \rho_{\chi_e}(\un_{D(s,\ \Omega(S))})=\lambda_s, \quad s\in S.
$$
The orbit of $\chi_e$ in $\G_P(S)$ is $\{\chi_s \mid s\in S\}$ so a sufficient (but not necessary) condition for $\rho_{\chi_e}$ to be an isomorphism is that Condition \ref{condition2} does not hold for $Y=\{\chi_s \mid s\in S\}$ and any $x\in \Omega(S)\setminus Y$. This has an equivalent formulation in terms of $S$.

\begin{definition}[{\cite[Definition~2.2]{NS}}]\label{def:strong-regularity2}
We say that $S$ is \textbf{strongly C$^*$-regular} if, given elements $h_1,\dots,h_n\in I_\ell(S)$ and constructible ideals $X,X_1,\dots,X_m\in\J(S)$ satisfying
\begin{equation}\label{eq:strong-regularity0}
\emptyset\ne X\setminus\bigcup^m_{i=1} X_i\subseteq \bigcup^n_{k=1}\{s\in S \mid h_ks=s\},
\end{equation}
there are constructible ideals $Y_1,\dots, Y_l\in \J(S)$ and indices $1\le k_j\le n$ ($j=1,\dots,l$) such that
\begin{equation}\label{eq:strong-regularity}
X\setminus\bigcup^m_{i=1} X_i\subseteq\bigcup^l_{j=1}Y_j\qquad\text{and}\qquad h_{k_j}p_{Y_j}=p_{Y_j}\quad\text{for all}\quad 1\le j\le l.
\end{equation}
\end{definition}

\subsection{Paterson's Boundary Groupoid}\label{subsec:3.2}
We consider a suitable reduction of $\G_P(S)$ which will give a candidate for a groupoid model of the boundary quotient $C_r^*(S)/ I$.\\

Recall that a filter on $\J(S)$ is a nonempty collection of nonempty sets $\F\subseteq \J(S)$ which is closed under finite intersection and upwards inclusion. In other words (i) $A, B\in \F \Rightarrow A\cap B\in~\F$, and (ii) $A\in \F,\ B\in \J(S),\ A\subseteq~B \Rightarrow B\in \F$.
We say that a character $\chi\in \widehat{\J(S)}$ is \textbf{maximal} if the underlying filter $\chi^{-1}(1)$ of sets is maximal among the filters on $\J(S)$. If $\emptyset \notin \J(S)$, then $\J(S)$ itself is the unique such filter, otherwise there are more. The collection of maximal characters in $\widehat{\J(S)}$ is denoted $\Omega_{\text{max}}(S)$. One verifies that $\Omega_{\text{max}}(S)\subseteq\Omega(S)$; the same proof used in \cite[Lemma~5.7.7]{CELY} works. In particular $\Omega_{\text{max}}(S)$ is a subset of $\G_P(S)$. We note the following facts.

\begin{lemma}[{\cite[Lemma~2.21]{Li-1}}]\label{lem:maximal-facts}
\begin{enumerate}[(i)]

\item $\Omega_{\text{max}}(S)$ is an invariant subset of $\G_P(S)$.
\item For all $X\in \J(S)$, $X\neq \emptyset$, there exists a $\chi\in \Omega_{\text{max}}(S)$ with $\chi(X)=1$.\label{lem:maximal-fact2}
\item If $I_\ell(S)$ contains $0$, then $\chi\in \Omega_{\text{max}}(S)$ if and only if for all $X\in\J(S)$ with $\chi(X)=0$ there exists $Y\in \J(S)$ such that $\chi(Y)=1$ and $X\cap Y=\emptyset$.\label{lem:maximal-fact3}
\item If $I_\ell(S)$ does not contain $0$, then $\Omega_{\text{max}}(S)$ consists of a single character $\chi$ such that $\chi(X)=1$ for all $X\in\J(S)$.
\end{enumerate}
\end{lemma}
It follows from the lemma that the closure
$$\partial \Omega(S):=\overline{\Omega_{\text{max}}(S)}\subseteq \Omega(S)$$
is an invariant subset of $\G_P(S)$. In \cite{MR2419901} the elements of $\partial \Omega(S)$  are called tight characters. We define the \textbf{boundary groupoid} of $S$ as 
$$\partial\G_P (S):=(\G_P(S))_{\partial\Omega(S)}.$$
For readability, we will sometimes suppress the left cancellative monoid $S$ in our notation. So we write $\G_P$, $\partial\G_P, \Omega$ and $\partial\Omega$ for $\G_P(S)$, $\partial\G_P(S), \Omega(S)$ and $\partial\Omega(S)$ when $S$ is clear from context.\\

Now, applying~\eqref{eq: reduced-short-exact} we obtain an exotic norm $e$ on $C_c(\partial\G_P (S))$ and a short exact sequence
\begin{equation}\label{eq: short-exact-boundary}
0 \rightarrow C_r^*(\G_P\setminus \partial\G_P) \rightarrow C_r^*(\G_P) \xrightarrow[]{\phi} C_e^*(\partial\G_P) \rightarrow 0.
\end{equation}
We argue that $\phi$ always induces the following commutative diagram,
\begin{equation}\label{eq:commutativediagram}
\xymatrix{
C^*_r(\G_P)\ar[d]^{\phi} \ar[r]^{\rho_{\mathcal{X}_e}} & C^*_r(S) \ar[d]^q\\
C^*_e(\partial\G_P) \ar@{.>}[r] & C_r^*(S)/I.
}
\end{equation}
To see why the bottom map exists we begin with the following observation.

\begin{proposition}\label{prop:foundation-maximal}
Let $X, X_1, \ldots ,X_m\in \J(S)$, $m\geq 1$, be constructible right ideals such that $X_i\subseteq X$ for all $i$. Then clopen set
\begin{equation}\label{eq:U-neighbourhood2}
U=\{\eta\in\Omega(S)\mid \eta(X)=1,\ \eta(X_i)=0\ \text{for}\ i=1,\dots,m\}
\end{equation}
is disjoint from $\partial \Omega$ if and only if the $X_i$'s form a foundation set for $X$.
\end{proposition}
\begin{proof}
For the forward direction, suppose that the $X_i$'s do not form a foundation set for~$X$. Then there exists a nonempty constructible ideal $Y\in \mathcal{J}(S)$ contained in $X$ which is disjoint from each~$X_i$. By Lemma~\ref{lem:maximal-facts} (\ref{lem:maximal-fact2}) there is a maximal character $\chi$ with $\chi(Y)=1$. Then $\chi\in U$, in particular~$U$ intersects $\partial \Omega$.\\

For the other direction, suppose $U$ intersects $\partial\Omega$. In particular $U$ contains a maximal character, say $\chi$. This is impossible in the case $0\notin I_\ell(S)$ since then $\chi(Y)=1$ for all $Y\in \J(S)$. Therefore $0\in I_\ell(S)$. By Lemma~\ref{lem:maximal-facts}(\ref{lem:maximal-fact3}) there exists, for each $i$, a $Y_i\in \mathcal{J}(S)$ disjoint from $X_i$ with $\chi(Y_i)=1$. Put $Y=X\cap\bigcap_i Y_i$. Then $\chi(Y)=1$ so $Y$ is a nonempty constructible ideal contained in $X$ which is disjoint from each $X_i$. This means that the $X_i$'s do not form a foundation set for~$X$.
\end{proof}

\begin{proposition}\label{prop:comdiagramexists}
In Diagram~\eqref{eq:commutativediagram} we have $\ker(q\circ \rho_{\chi_e})=\ker\rho_{\chi_e}+C_r^*(\G_P\setminus \partial \G_P)$. In particular the bottom map in~\eqref{eq:commutativediagram} exists.
\end{proposition}

\begin{proof}
Let $I$ be as in Equation~\eqref{eq:I}. We will establish that
\begin{equation}\label{eq:image1}
\rho_{\chi_e}(C_r^*(\G_P\setminus \partial \G_P))=I.
\end{equation}
This will prove the proposition. Indeed, assuming Equality~\eqref{eq:image1}, if $x\in \ker(q\circ \rho_{\chi_e})$ then $\rho_{\chi_e}(x)\in I$, so there exists $y\in C_r^*(\G_P\setminus \partial \G_P)$ with $\rho_{\chi_e}(y)=\rho_{\chi_e}(x)$. Then $x=(x-y)+y\in \ker\rho_{\chi_e}+C_r^*(\G\setminus \partial \G)$, which demonstrates that $\ker(q\circ \rho_{\chi_e})\subseteq\ker\rho_{\chi_e}+C_r^*(\G_P\setminus \partial \G_P)$. The inclusion  $\ker(q\circ \rho_{\chi_e})\supseteq\ker\rho_{\chi_e}+C_r^*(\G_P\setminus \partial \G_P)$ follows similarly.\\

To verify~\eqref{eq:image1}, let us first show that $\rho_{\chi_e}(C_r^*(\G_P\setminus \partial \G_P))\subseteq I$. Since $C_c(\Omega\setminus \partial\Omega)$ generates $C_r^*(\G\setminus \partial \G)$ as an ideal it suffices to show that $\rho_{\chi_e}(C_c(\Omega\setminus \partial\Omega))\subseteq I$. Fix $f\in C_c(\Omega\setminus \partial\Omega)$ and let $K=\text{supp}(f)$. For each $\chi\in K$ there is an open neighbourhood $U_\chi$ of the form
$$U_\chi=\{\eta\in\Omega(S)\mid \eta(X)=1,\ \eta(X_i)=0\ \text{for}\ i=1,\dots,m\},$$
where $X, X_1, \ldots ,X_m\in \J(S)$, $m\geq 1$ which does not intersect $\partial \Omega$. By intersecting the $X_i$'s with~$X$, we may assume that $X_i\subseteq X$ for each $i$. Then Proposition~\ref{prop:foundation-maximal} implies that $\{X_1, \ldots ,X_m\}$ is a foundation set for $X$. The $U_\chi$, $\chi\in K$, form an open cover for $K$ and by compactness there are finitely many $\chi_1, \ldots ,\chi_n \in K$ such that
$$\text{supp}(f)=K\subseteq \bigcup\limits_{j=1}^n U_{\chi_j}.$$

Let $X^j, X_i^j\ (1\leq i \leq m_j)$, be the constructible ideals used to define $U_{\chi_j}$. Then
$$
f(\chi_s) \neq 0\Rightarrow s\in \bigcup\limits_{j=1}^n (X^j\setminus \cup_{i=1}^{m_j} X_i^j).
$$
Since $\rho_{\chi_e}(f)\delta_s=f(\chi_s)\delta_s$ this means that
$$
\rho_{\chi_e}(f)=\rho_{\chi_e}(f)\un_{\bigcup\limits_{j=1}^n (X^j\setminus \cup_{i=1}^{m_j} X_i^j)}\in I.
$$

We now show the other inclusion, that $I\subseteq\rho_{\chi_e}(C_r^*(\G_P\setminus \partial \G_P))$. Since the image $\rho_{\chi_e}(C_r^*(\G_P\setminus \partial \G_P))$ is a closed ideal of $C_r^*(S)$ we have only to see that it contains the generators of $I$. So suppose $X_1, \ldots, X_m \in \J(S)$ forms a foundation set for $X\in \J(S)$. Then the clopen set
$$U=\{\chi\in \Omega(S) \mid \chi(X)=1,\ \chi(X_i)=0,\ i=1,\ldots , m\}$$
is contained in $\Omega\setminus \partial\Omega$ by Proposition~\ref{prop:foundation-maximal}. Therefore
$$\un_{X\setminus \bigcup\limits_{i=1}^m X_i}=\rho_{\chi_e}(\un_U)\in \rho_{\chi_e}(C_r^*(\G_P\setminus \partial \G_P)),$$
which completes the proof.
\end{proof}

\begin{corollary}\label{cor:iso-sufficient-condition-1}
The bottom map in Diagram~\eqref{eq:commutativediagram} is an isomorphism if and only if $\ker{\rho_{\chi_e}}\subseteq C_r^*(\G_P\setminus \partial \G_P)$.
\end{corollary}

\begin{remark}
In particular, if $\rho_{\chi_e}$ is an isomorphism, then the bottom map in Diagram~\eqref{eq:commutativediagram} is an isomorphism. So for example, when $S$ is strongly C$^*$-regular, then the bottom map is an isomorphism.
\end{remark}

\begin{corollary}\label{cor:iso-sufficient-condition-2}
If the norm $e$ from Short Exact Sequence~\eqref{eq: short-exact-boundary} coincides with the reduced norm and $\rho_z\prec \rho_{\chi_e}\sim \bigoplus_{s\in S}\rho_{\chi_s}$ for all $z\in \partial \Omega$, then the bottom map in Diagram~\eqref{eq:commutativediagram} defines an isomorphism $C_r^*(\partial\G_P(S)) \cong C_r^*(S)/I$.
\end{corollary}
\begin{proof}
We always have the commutative diagram
\begin{equation}\label{eq:reducedcommutativediagram}
\xymatrix{
C^*_r(\G_P)\ar[d]_{\pi} \ar[dr]^{\bigoplus\limits_{z\in \partial \Omega}\rho_z}\\
C^*_r(\partial\G_P) \ar[r] & \bigoplus\limits_{z\in \partial \Omega} B(\ell^2 ((\G_P)_z))
}
\end{equation}
where $\pi$ extends restriction of functions and the bottom map is the direct sum of all left regular representations of $C_r^*(\partial \G_P)$. Since the bottom map is faithful we have $\ker{\pi}=\ker(\bigoplus_{z\in \partial \Omega}\rho_z)$. Assuming $C_r^*(\partial \G_P)=C_e^*(\partial \G_P)$, then by Short Exact Sequence~\eqref{eq: short-exact-boundary} these kernels are just \mbox{$C_r^*(\G_P\setminus\partial\G_P)$}. Then, if $\bigoplus_{z\in \partial \Omega}\rho_z$ is weakly contained in $\rho_{\chi_e}\sim \bigoplus_{s\in S}\rho_{\chi_s}$, we have
$$\ker{\rho_{\chi_e}}\subseteq \ker(\bigoplus_{z\in \partial \Omega}\rho_z)=C_r^*(\G_P\setminus \partial  \G_P)$$
and the condition in Corollary~\ref{cor:iso-sufficient-condition-1} is met.
\end{proof}
It seems difficult to find conditions on $S$ which guarantee that the norm $e$ in~\eqref{eq: short-exact-boundary} is the reduced norm. However, negating Condition \ref{condition2} in the groupoid $\G_P(S)$ for $Y=\{\chi_s \mid s\in S\}$ and every $x\in \partial\Omega\setminus Y$, leads to the following condition on $S$ which guarantees that the second part of the hypothesis in Corollary~\ref{cor:iso-sufficient-condition-2} is true.

\begin{definition}\label{def:strong-regularity-on-boundary}
We say that $S$ is \textbf{strongly C$^*$-regular on the boundary} if, given elements $h_1, \ldots, h_n\in I_\ell(S)$ and constructible ideals $X, X_1, \ldots, X_m\in \J(S)$, $X_i\subseteq X$ for all $i$, satisfying
\begin{equation}\label{eq:hypstronggreg}
\emptyset \neq X\setminus \bigcup_{i=1}^m X_i \subseteq \bigcup_{k=1}^n \{s\in S \mid h_ks=s\},
\end{equation}
then there are constructible ideals $Y_1, \ldots ,Y_l\in \J(S)$ and indices $1\leq k_j\leq n$ ($j=1, \ldots, l$) such that $h_{k_j}p_{Y_j}=p_{Y_j}$, and $\{X_1, \ldots , X_m, Y_1, \ldots ,Y_l\}$ is a foundation set for $X$. That is, $Y_j\subseteq X$ for all $j$ and the set
$$X\setminus (\bigcup_{i=1}^m X_i\cup \bigcup_{j=1}^l Y_j)$$
does not contain a nonempty constructible ideal. Here we allow the set $\{Y_1, \ldots ,Y_l\}$ to be empty.
\end{definition}

\begin{remark}
Clearly, strong C$^*$-regularity on the boundary is a weaker property than strong C$^*$-regularity, see Definition~\ref{def:strong-regularity2}.
\end{remark}

\begin{remark}\label{rem:strong-regularity-on-boundary-remark-2}
In checking Definition~\ref{def:strong-regularity-on-boundary} we may assume that $X\subseteq\dom h_k$ for all $k$. Indeed, suppose that inclusion~\eqref{eq:hypstronggreg} is satisfied. For $F\subseteq \{1, \ldots ,n\}$, put
$$X_F:=X\cap \bigcap\limits_{k\in F}\dom h_k.$$
Then~\eqref{eq:hypstronggreg} is satisfied for $X_F$, $\{X_1\cap X_F, \ldots , X_m\cap X_F, \text{ }
\dom h_k \cap X_F \text{ }(k\notin F)\}$ and $\{h_k\text{ }(k\in F)\}$ in place of $X, 
\{X_1, \ldots, X_m\}$ and $\{h_k\text{ }(1\leq k\leq n)\}$. Suppose that for each $F$ we can find finitely many constructible ideals $\{A_i\}_{i\in I_F}$ such that for each $i$ we have $h_{k_i}p_{A_i}=p_{A_i}$ for some $k_i$, and $\{X_1\cap X_F, \ldots , X_m\cap X_F,\text{ } \dom h_k~\cap~X_F \text{ }(k\notin F), \text{ } A_i \text{ }(i\in I_F)\}$ is a foundation set for $X_F$. We claim that then $\{X_1, 
\ldots , X_m,\text{ } A_i\text{ }(i\in I_F,\text{ }F\subseteq \{1, \ldots, n\})\}$ is a foundation set for $X$. Indeed, if it was not, then by definition there is a nonempty constructible ideal $B\in \mathcal{J}(S)\setminus \{\emptyset\}$, satisfying
$$B\subseteq X\setminus (\bigcup\limits_{i=1}^m X_i \cup \bigcup\limits_{F,\text{ }i\in I_F} A_i).$$
Now pick a maximal subset $F$ such that $B\cap X_F\neq \emptyset$. Then
$$B\cap X_F\subseteq X_F\setminus (\bigcup\limits_{i=1}^m X_i \cup \bigcup\limits_{k\notin F}\dom h_k \cup \bigcup\limits_{i\in I_F} A_i),$$
which is a contradiction. Hence it is enough to establish the definition when $X\subseteq\dom h_k$ for all $k$.
\end{remark}

\begin{lemma}\label{lem:boundaryregular}
Condition \ref{condition2} is not satisfied for $\G=\G_P(S)$, $Y=\{\chi_s~\mid~s\in~S\}$ and every $\chi\in \partial\Omega\setminus Y$ if and only if $S$ is strongly C*-regular on the boundary.
\end{lemma}
\begin{proof}
Assume first that $S$ is strongly C$^*$-regular on the boundary. Suppose there is a $\chi\in \partial\Omega\setminus Y$ such that Condition \ref{condition2} is satisfied for $x=\chi$ and let $g_k=[h_k, \chi]$, $U_k$ $(1\leq k\leq n)$ and $U$ be as in that condition. We may assume that $U_k=D(h_k, \Omega)$ and
\begin{equation}\label{eq:Ufirst}
U=\{\eta\in\Omega(S)\mid \eta(X)=1,\ \eta(X_i)=0\ \text{for}\ i=1,\dots,m\}
\end{equation}
for some $X, X_1, \ldots ,X_m\in \mathcal{J}(S)$, $X_i\subseteq X$ for all $i$. Then Condition \ref{condition2} says that for every $s\in X\setminus \cup_{i=1}^m X_i$ there is a $k$ such that $\chi_s\in U_k$, that is, $h_ks=s$. Applying strong C$^*$-regularity on the boundary there exists $Y_1, \ldots, Y_l\in \mathcal{J}(S)$ such that $h_{k_j}p_{Y_j}=p_{Y_j}$, $j=1, \ldots ,l$, and\\ $X_1, \ldots , X_m, Y_1, \ldots ,Y_l$ form a foundation set for $X$. Then, by Proposition~\ref{prop:foundation-maximal}, the set
$$V=\{\eta\in\Omega(S)\mid \eta(X)=1,\ \eta(X_i)=\eta(Y_j)=0\ \text{for}\ i=1,\dots,m,\ j=1, \ldots ,l\}$$
cannot contain $\chi$, so $\chi\in U\setminus V$ and there must exist $j$ with $\chi(Y_j)=1$. This implies that $g_{k_j}=[h_{k_j},\chi]=\chi$ which contradicts that $g_1, \ldots ,g_n$ are nontrivial elements of the isotropy group $\G_{\chi}^\chi$.\\

Assume now that $S$ is not strongly C$^*$-regular on the boundary. Then there are elements $h_1, \ldots ,h_n\in I_\ell(S)$ and constructible ideals $X, X_1, \ldots ,X_m\in \mathcal{J}(S)$ such that~\eqref{eq:hypstronggreg} holds and for any finite set $F\subset \mathcal{J}(S)$ where each member of $F$ is contained in $X$ and fixed by some $h_k$, the set $\{X_1, \ldots ,X_m\}\cup F$ is not a foundation set for $X$. For such $F$ the compact set
\begin{equation}\label{eq:UF}
U_F:=\{\eta\in\Omega(S)\mid \eta(X)=1,\ \eta(X_i)=0\ \text{for}\ i=1,\dots,m,\ \eta(Y)=0\ \text{for}\ Y\in F\}
\end{equation}
intersects $\partial\Omega$ by Proposition~\ref{prop:foundation-maximal}. Letting $\F$ denote the collection of finite sets $F\subset \J(S)$ where each member is contained in $X$ and fixed by some $h_k$, we can pick
$$\chi\in \bigcap_{F\in \F}U_F\cap \partial\Omega.$$

We claim that it is possible to replace $U$ by a smaller neighbourhood of $\chi$ and discard some of the elements $h_k$ in such a way that Condition \ref{condition2} gets satisfied for $x=\chi$, $g_k=[h_k,\chi]$, $U=U_F$ and $U_k=D(h_k,\Omega(S))$ for the remaining $k$. Namely, if $(h_k,\chi)\not\in\Sigma$ for some $k$, then we add $\dom h_k$ to the collection $\{X_1,\dots,X_m\}$ and discard such $h_k$. If $(h_k,\chi)\in\Sigma$ but $\chi(h_k^{-1}(\cdot))\ne\chi$, then $\chi_s(h_k^{-1}(\cdot))\ne\chi_s$ for all $\chi_s$ close $\chi$, so by replacing $X$ by a smaller ideal and adding more constructible ideals to $\{X_1,\dots,X_m\}$ we may assume that $\chi_s(h_k^{-1}(\cdot))\ne\chi_s$ for all $\chi_s\in U$ and again discard such $h_k$. For the remaining elements $h_k$ and the new $U$ we have that for every $s\in S$ such that $\chi_s\in U$ there is an index $k$ satisfying $h_ks=s$. Then, in order to show that Condition~\ref{condition2} is satisfied, it remains to check that the elements $g_k=[h_k,\chi]$ of $\G^\chi_\chi$ are nontrivial. But this is clearly true, since $\chi(p_J)=0$ for every $J\in\J(S)$ such that $h_kp_J=p_J$.
\end{proof}

Applying Lemma~\ref{lem:boundaryregular} to Proposition~\ref{prop:weak-containmenttwo} and Corollary~\ref{cor:iso-sufficient-condition-2} gives:
\begin{theorem}\label{prop:iso-sufficient-condition-final}
Let $S$ be a left cancellative monoid. If the norm $e$ from the Short Exact Sequence \eqref{eq: short-exact-boundary} coincides with the reduced norm and $S$ is strongly C*-regular on the boundary, then the bottom map in Diagram~\eqref{eq:commutativediagram} defines an isomorphism $C_r^*(\partial\G_P(S)) \cong C_r^*(S)/I$
\end{theorem}

In our previous paper \cite[Proposition~2.6]{NS}, we saw that the full groupoid C*-algebra $C^*(\G_P(S))$ had a description in terms of generators and relations. By modifying one of these relations we obtain a description for $C^*(\partial\G_P(S))$ as well.

\begin{proposition}[cf.~{\cite[Theorem~10.10]{MR4151331}, \cite[Theorem~6.13]{LS}}]\label{prop:universal-description}
Assume~$S$ is a left cancellative monoid. Consider the elements
$v_s:=\un_{\{[s,\ \chi]\ \mid\ \chi\in \partial \Omega\}}\in C^*(\partial\G_P(S))$, $s\in S$. Then $C^*(\partial\G_P(S))$ is the universal unital C$^*$-algebra generated by the elements $v_s$, $s\in S$, satisfying the following relations:
\begin{enumerate}
  \item[(R1)] $v_e=1$;
  \item[(R2)] for every $g=s_{2n}^{-1}s_{2n-1}\dots s_2^{-1}s_1\in I_\ell(S)$, $(s_i\in S)$, the element $v_g:=v_{s_{2n}}^*v_{s_{2n-1}}\dots v_{s_2}^*v_{s_1}$ is independent of the presentation of $g$;
  \item[(R3)] if $0\in I_\ell(S)$, then $v_0=0$;
  \item[(R4)] if $X_1, \ldots, X_m \in \J(S)$ forms a foundation set for $X\in \J(S)$, then
$$
\prod^m_{i=1}(v_{p_X}-v_{p_{X_i}})=0.
$$
\end{enumerate}
\end{proposition}
\begin{proof}
From \cite[Proposition~2.6]{NS} we know that $C^*(\G_P)$ is the universal C$^*$-algebra generated by elements $v_s$, $s\in S$, subject to relations $(R1)$-$(R3)$ and a weaker relation $(R4')$ in place of $(R4)$, namely \textit{
\begin{enumerate}
  \item[(R4')] if $X=\bigcup\limits_{i=1}^m X_i$ for $X,X_1,\dots,X_m\in\J(S)$, then $\prod^m_{i=1}(v_{p_X}-v_{p_{X_i}})=0$.
\end{enumerate}
}

Equation~\eqref{eq: universal-short-exact} tells us that $C^*(\partial \G_P)$ is canonically isomorphic to the quotient \mbox{$C^*(\G_P)/C^*(\G_P\setminus\partial\G_P)$} so to prove the proposition it suffices to check that the set
\begin{align}\label{eq:generating-set}
\{\prod^n_{i=1}(v_{p_X}-v_{p_{X_i}}) \mid \ & X, X_1, \ldots, X_m\in \J(S), \nonumber \\
& \{X_1, \ldots, X_m\} \text{ is a foundation set for } X,\ m\geq 1 \}
\end{align}
generates $C^*(\G_P\setminus\partial\G_P)$ as a closed ideal of $C^*(\G_P)$. For $Y\in \J(S)$ one has $v_{p_Y}=\un_{\{\eta\in \Omega(S)\ \mid \ \eta(Y)=1\}}$. Therefore
$$\prod^m_{i=1}(v_{p_X}-v_{p_{X_i}})=\un_{\{\eta\in\Omega(S)\ \mid\ \eta(X)=1,\ \eta(X_i)=0\ \text{for}\ i=1,\dots,m\}}.$$

Proposition~\ref{prop:foundation-maximal} implies that as $X, X_1, \ldots, X_m\in \J(S)$, $m\geq 1$, range over all possible foundation sets $\{X_1, \ldots, X_m\}$ of all constructible ideals~$X$, the span of such functions is dense in $C_0(\Omega\setminus \partial\Omega)$. But by \cite[Proposition~1.1~and~Corollary~1.2]{NS}, $C_0(\Omega\setminus \partial\Omega)$ generates  $C^*(\G_P\setminus\partial\G_P)$ as a closed ideal.
\end{proof}

\begin{remark}
Since $\partial \G_P(S)= \G_{\text{tight}}(I_\ell(S))$, it follows from \cite{MR2419901} that $C^*(\partial\G_P(S))\cong C_{\text{tight}}^*(I_\ell(S))$ where the latter is the C$^*$-algebra generated by $I_\ell(S)$ which is universal for tight representations (see \cite[Definition~11.6]{MR2419901}). Using the description of $C^*(\G_P(S))$ from \cite[Proposition~2.6]{NS} and the description of $C^*(\partial \G_P(S))$ from Proposition~\ref{prop:universal-description}, we then obtain the following diagram
\begin{equation}\label{eq:inversesemigroupdiagram}
\xymatrix{
C^*(I_\ell(S))\ar[d] \ar[r] & C^*(\G_P(S)) \ar[d]\\
C_{\text{tight}}^*(I_\ell(S)) \ar[r]^{\simeq} & C^*(\partial \G_P(S)).
}
\end{equation}
Since $C^*(I_\ell(S))\cong C^*(\G(I_\ell(S)))$ by \cite[Theorem 4.4.1]{MR1724106}, the top map is an isomorphism if $\G(I_\ell(S))=\G_P(S)$ which happens precisely when $S$ satisfies the independence condition,\cite[Definition~6.30]{Li-3}. The bottom map is always an isomorphism.
\end{remark}

\subsection{Spielberg's Groupoid}
We introduce a quotient groupoid of $\G_P(S)$, which was defined by Spielberg in~\cite{MR4151331}. Like $\G_P(S)$ it is a candidate for a groupoid model of $C_r^*(S)$.\\

Denote by $\bar\J(S)$ the collection obtained by appending to $\J(S)$ sets of the form $X\setminus \cup_{i=1}^m X_i$ where $X, X_1, \dots, X_m\in \J(S)$. Then $\bar\J(S)$ is a semigroup under intersection. We can extend Definition~\ref{def:foundationset} to sets in $\bar\J(S)$.

\begin{definition}\label{def:foundationset-generalized}
Let $X, X_1, \dots, X_m\in \bar\J(S)$. Then $\{X_1, \ldots , X_m\}$ is a \textbf{foundation set} for $X$ if
\begin{enumerate}[(I)]
  \item  $X_i\subseteq X$ for all $i$;
  \item for any nonempty constructible right ideal $Y\in \mathcal{J}(S)\setminus \{\emptyset\}$ satisfying $Y\subseteq X$, there is an index $i_ 0$ such that $Y\cap X_{i_0}\neq \emptyset$.
\end{enumerate}
\end{definition}

Every $\chi \in \Omega(S)$ extends naturally to a character on $\bar\J(S)$ by letting $\chi(X\setminus \cup_{i=1}^m X_i)=1$ if $\chi(X)=1$ and $\chi(X_i)=0$ for all $i$, and $0$ otherwise.\\

Define an equivalence relation $\sim_2$ on $\G_P(S)$ by declaring $[g, \chi]\sim_2 [h, \chi]$ iff there exists $X\in \bar\J(S)$ with $\chi(X)=1$ such that $g_{|X}=h_{|X}$. Consider the quotient groupoid
$$
\G(S):=\G_P(S)/\sim_2.
$$
Note that the unit space of $\G(S)$ is the same as that of $\G_P(S)$, namely $\Omega(S)$. To avoid confusion with the previous groupoid, the left regular representation of $C_r ^*(\G(S))$ associated with a character $\chi\in \Omega(S)$, will be denoted $(\rho_\chi)_2$. Again the left regular representation associated to the character~$\chi_e$ gives a surjective $*$-homomorphism
$$(\rho_{\chi_e})_2: C_r^*(\G(S)) \rightarrow C_r^*(S).$$
The negation of Condition \ref{condition2} in the groupoid $\G(S)$ for the subset $Y=\{\chi_s~\mid~s\in~S\}$ and every $x\in \Omega(S)\setminus Y$ gives the following condition:

\begin{definition}[{\cite[Definition~2.11]{NS}}]\label{def:regularity2}
We say that $S$ is \textbf{C$^*$-regular} if, given $h_1,\dots,h_n\in I_\ell(S)$ and $X\in\bar\J(S)$ satisfying
\begin{equation*}
\emptyset\ne X\subseteq \bigcup_{k=1}^n\{s\in S \mid h_ks=s\},
\end{equation*}
there are sets $Y_1,\dots, Y_l\in\bar\J(S)$ and indices $1\le k_j\le n$ ($j=1,\dots,l$) such that
\begin{equation*}
X\subseteq\bigcup_{j=1}^lY_j\qquad\text{and}\qquad h_{k_j}p_{Y_j}=p_{Y_j}\quad\text{for all}\quad 1\le j\le l.
\end{equation*}
\end{definition}

Again $\partial\Omega(S)$ is an invariant subset of $\G(S)$ and we may form the boundary groupoid
$$
\partial\G(S):={\G(S)}_{\partial\Omega(S)}.
$$

\subsection{Spielberg's Boundary Groupoid}
As with $\G_P(S)$, $\partial\Omega(S)$ is an invariant subset of $\G(S)$ and we may form the boundary groupoid
$$
\partial\G(S):={\G(S)}_{\partial\Omega(S)}.
$$
This is another candidate for a groupoid model for the boundary quotient. In this subsection we show that the analogous results of Subsection \ref{subsec:3.2} hold for $\G(S)$.
As before, we occasionally suppress $S$ from our notation and write $\G$ and $\partial \G$ for~$\G(S)$ and $\partial\G(S)$. First and foremost, there is a short exact sequence analogous to Short Exact Sequence~\eqref{eq: short-exact-boundary}:
\begin{equation}\label{eq:short-exact-sequences}
0 \rightarrow C_r^*(\G\setminus \partial \G) \rightarrow C_r^*(\G) \xrightarrow[]{\phi_2} C_{e_2}^*(\partial \G) \rightarrow 0.
\end{equation}
Here $e_2$ is an exotic norm completion of $C_c(\partial \G)$.\\

We have the following $\bar\J(S)$ generalization of Proposition~\ref{prop:foundation-maximal}:

\begin{proposition}\label{prop:foundation-maximal-G_2}
Let $X, X_1, \ldots, X_m\in \bar\J(S)$, $m\geq 1$, $X_i\subseteq X$ for all $i$. Then the clopen set
$$U=\{\eta\in\Omega(S)\mid \eta(X)=1,\ \eta(X_i)=0\ \text{for}\ i=1,\dots,m\}$$
is disjoint from $\partial\Omega$ if and only if the $X_i$'s form a foundation set for $X$.
\end{proposition}

\begin{proof}
For the forward direction, suppose that the $X_i$'s do not form a foundation set for $X$. Then there exists a nonempty $Y\in \J(S)$ contained in $X$ which is disjoint from each $X_i$. There exists a maximal character $\chi\in \Omega$ with $\chi(Y)=1$; we claim that $\chi\in U$. Since $Y\subseteq X$ we clearly have $\chi(X)=1$. Now, fix an $X_i$ and write it as a finite difference of constructible ideals

\begin{equation}\label{eq:differences-of-constructible-ideals}
X_i=X_i^0 \setminus \bigcup\limits_{j=1}^{m_i}X_i^j.
\end{equation}

Suppose that $\chi(X_i^0)=1$. Then $\chi(Y\cap X_i^0)=1$ and since $Y$ is disjoint from $X_i$ one has
$$Y\cap {X_i}^0\subseteq \bigcup\limits_{j=1}^{m_i} X_i^j.$$
Since $\chi\in \Omega$ we must have $\chi({X_i}^{j'})=1$ for some $j'$. This shows that $\chi(X_i)=0$ for all $i$, hence $\chi\in U$. In particular $U$ intersects $\partial \Omega$.\\

For the opposite direction, suppose $U$ intersects $\partial\Omega$. In particular $U$ contains a maximal character, say $\chi$. Write $X_1, \ldots , X_m$ as in~\eqref{eq:differences-of-constructible-ideals} and similarly write
$$X=X^0 \setminus \bigcup\limits_{j=1}^{m_0}X^j,$$
We have $\chi(X^j)=0$ for $1\leq j\leq m_0$ so by maximality there is a constructible ideal, $Y^j\in \mathcal{J}(S)$, disjoint from $X^j$, satisfying $\chi(Y^j)=1$. Now, for each $i$, $1\leq i\leq m$, such that $\chi(X_i^0)=0$, let $A_i\in \mathcal{J}(S)$ be a constructible ideal disjoint from $X_i^0$ satisfying $\chi(A_i)=1$. Amongst the $i$'s for which $\chi(X_i^0)=1$ there must be an index~$j_i$ such that $\chi({X_i}^{j_i})=1$. In this case, set $A_i=X_i^{j_i}$. Let

$$Y=X^0\cap \bigcap\limits\limits_{j=1}^{m_0} Y^j\cap \bigcap\limits_{i=1}^m A_i.$$
By construction $Y$ is contained in $X$ and is disjoint from each $X_i$. Moreover, $\chi(Y)=1$ so $Y$ is nonempty. Therefore the $X_i$'s do not form a foundation set for~$X$.
\end{proof}

With Proposition~\ref{prop:foundation-maximal-G_2} at our disposal it is not hard to check that Proposition~\ref{prop:comdiagramexists} and Corollaries~\ref{cor:iso-sufficient-condition-1} and~\ref{cor:iso-sufficient-condition-2} carry over for $\G$. Namely, we have the following diagram
\begin{equation}\label{eq:commutativediagram-G_2}
\xymatrix{
C^*_r(\G)\ar[d]^{\phi_2} \ar[r]^{(\rho_{\mathcal{X}_e})_2} & C^*_r(S) \ar[d]^q\\
C^*_{e_2}(\partial\G) \ar[r] & C_r^*(S)/I
}
\end{equation}
where $e_2$ is the norm from Short Exact Sequence~\eqref{eq:short-exact-sequences}. Furthermore
\begin{proposition}\label{prop:comdiagramexists-G_2}
In Diagram~\eqref{eq:commutativediagram-G_2} we have $\ker(q\circ (\rho_{\chi_e})_2)=\ker(\rho_{\chi_e})_2+C_r^*(\G\setminus \partial \G)$. In particular the bottom map in Diagram~\eqref{eq:commutativediagram-G_2} exists.
\end{proposition}

\begin{corollary}\label{cor:iso-sufficient-condition-1-G_2}
The bottom map in Diagram~\eqref{eq:commutativediagram-G_2} is an isomorphism if and only if $\ker{(\rho_{\chi_e})_2}\subseteq C_r^*(\G(S)\setminus \partial \G(S))$.
\end{corollary}

\begin{corollary}\label{cor:iso-sufficient-condition-2-G_2}
If the norm $e_2$ from Short Exact Sequence~\eqref{eq:short-exact-sequences} coincides with the reduced norm and $\bigoplus_{z\in \partial \Omega}(\rho_z)_2$ is weakly contained in $(\rho_{\chi_e})_2\sim \bigoplus_{s\in S}(\rho_{\chi_s})_2$, then the bottom map in Diagram ~\eqref{eq:commutativediagram-G_2} defines an isomorphism $C_r^*(\partial\G(S)) \cong C_r^*(S)$.
\end{corollary}
Negating Condition \ref{condition2} for the groupoid $\G(S)$, the subset $Y=\{\chi_s \mid s\in S\}$ and every $x\in \partial\Omega\setminus Y$ leads to the following condition on $S$ which guarantees that the second part of the hypothesis in Corollary~\ref{cor:iso-sufficient-condition-2-G_2} is true:

\begin{definition}\label{def:regularity-on-boundary}
We say that $S$ is \textbf{C*-regular on the boundary} if, for any elements $h_1, \ldots, h_n\in I_\ell(S)$ and constructible ideals $X, X_1, \ldots, X_m\in \J(S)$, $X_i\subseteq X$ for all $i$, satisfying
$$\emptyset \neq X\setminus \bigcup_{i=1}^m X_i \subseteq \bigcup_{k=1}^n \{s\in S \mid h_ks=s\},$$
there are $Y_1, \ldots ,Y_l\in \bar\J(S)$ and indices $1\leq k_j\leq n$, $j=1, \ldots, l$ such that $h_{k_j}p_{Y_j}=p_{Y_j}$, and $\{X_1, \ldots ,X_m,Y_1, \ldots ,Y_l\}$ is a foundation set for $X$, i.e. $Y_j\subseteq X$ for all $j$ and the set
$$X\setminus (\bigcup_{i=1}^m X_i\cup \bigcup_{j=1}^l Y_j)$$
does not contain a nonempty constructible ideal. Here we allow the set $\{Y_1, \ldots ,Y_l\}$ to be empty.
\end{definition}
\begin{remark}
By the same argument as in Remark \ref{rem:strong-regularity-on-boundary-remark-2}, in order to check C$^*$-regularity on the boundary it suffices to consider $h_k$ such that $X\subseteq \dom h_k$ for all $k$. Note also that the only difference between strong C$^*$-regularity on the boundary and C$^*$-regularity on the boundary is that the sets $Y_j$ are required to be in $\J(S)$ in the former and in $\bar\J(S)$ in the latter. In particular, strong C$^*$-regularity on the boundary implies C$^*$-regularity on the boundary.
\end{remark}

The analogue of Theorem~\ref{prop:iso-sufficient-condition-final} holds:
\begin{theorem}
Let $S$ be a left cancellative monoid. If the norm $e_2$ from Short Exact Sequence~\eqref{eq:short-exact-sequences} coincides with the reduced norm and $S$ is C*-regular on the boundary, then the bottom map in Diagram~\eqref{eq:commutativediagram-G_2} defines an isomorphism $C_r^*(\partial\G(S)) \cong C_r^*(S)/I$.
\end{theorem}

We remark that the universal C$^*$-algebra $C^*(\partial\G(S))$ has a presentation in terms of generators and relations from \cite[Theorem~10.10]{MR4151331} similar to that of $C^*(\partial\G_P(S))$. We therefore have two candidates for groupoid models for the boundary quotient $C_r^*(S)/I$. We note the following result.

\begin{proposition}
Let $S$ be a left cancellative monoid such that the bottom map in Diagram~\eqref{eq:commutativediagram} defines an isomorphism $C_e^*(\partial\G_P(S))\cong C_r^*(S)/I$. Then $\partial\G_P(S)=\partial\G(S)$.
\end{proposition}
\begin{proof}
We show the contrapositive so assume that $\partial\G_P(S)\neq \partial\G(S)$. Then there exists $[g, \chi]\in \partial\G_P(S)$ with $[g, \chi]\neq \chi$ and $[g, \chi]\sim_2 \chi$. Let $X, X_1,\dots ,X_m\in \J(S)$ be such that $gs=s$ for $s\in X\setminus \cup_{i=1}^m X_i$, $\chi(X)=1$ and $\chi(X_i)=0$ for all $i$. Consider the clopen set $U$ defined by
$$U=\{\eta\in\Omega\mid \eta(X)=1,\ \eta(X_i)=0\ \text{for}\ i=1,\dots,m\}
$$

and put $f=\un_{D(g, U)}-\un_U\in C_r^*(\G_P(S))$. Then $f$ is nonzero on $\partial\Omega$ so $f\notin C_r^*(\G_P\setminus \partial\G_P)$, but identifying $\ell^2((\G_P(S))_{\chi_e})$ with $\ell^2(S)$ we have
$$
\rho_{\chi_e}(f)=(\lambda_g-1)\un_{X\setminus \cup_{i=1}^m X_i}=0.
$$
Then Corollary~\ref{cor:iso-sufficient-condition-1} implies that the bottom map in Diagram~\eqref{eq:commutativediagram} is not an isomorphism.
\end{proof}

To our knowledge the bottom map in Diagram~\eqref{eq:commutativediagram} being an isomorphism does not imply that the two exotic norms $e$ and $e_2$ coincide or that the bottom map in Diagram~\eqref{eq:commutativediagram-G_2} defines an isomorphism. In other words, the above proposition does not tell us that $\partial \G(S)$ is a better model for the boundary quotient than $\partial \G_P(S)$. This is in contrast to the non-boundary case where we showed that $\G(S)$ is the superior groupoid model for $C_r^*(S)$, see \cite[Proposition~2.13]{NS}.\\

There is a condition on $S$ which guarantees precisely that $\partial\G_P(S)=\partial\G(S)$. It is reminiscent of the condition in \cite[Lemma~2.15]{NS} which characterized when $\G_P(S)=\G(S)$.

\begin{proposition}\label{prop:equality-of-boundary-groupoids}
Let $S$ be a left cancellative monoid. We have $\partial\G_P(S)=\partial\G(S)$ if and only if $S$ satisfies the following property: 
for every $g\in I_\ell(S)$ and $X, X_1, \ldots, X_m\in \J(S)$, $X_i\subseteq X$ for all $i$, such that 
$gs=s$ for all $s\in X\setminus \cup_{i=1}^m X_i$, there are $Y_1, 
\ldots ,Y_l\in \J(S)$ such that $X_1, \ldots , X_m, Y_1, \ldots ,Y_l$ form 
a foundation set for $X$ and $g p_{Y_j}=p_{Y_j}$ for all $j$. Here we allow the set $\{Y_1, \dots Y_l\}$ to be empty.
\end{proposition}
\begin{proof}
Assume first that $S$ satisfies the condition formulated in the lemma. In order 
to prove that $\partial\G_P(S)=\partial\G(S)$ we assume that $[g, \chi]\sim_2 \chi$ for some $[g, \chi]\in \partial\G_P(S)$ and prove that then $[g, \chi]=\chi$. Let $X, X_1, \ldots, X_m\in\J(S)$ be such that $gs=s$ for all $s\in X\setminus \cup_{i=1}^m X_i$, $\chi(X)=1$ and $\chi(X_i)=0$ for all~$i$. By assumption, there 
are $Y_1, \ldots , Y_l\in \J(S)$ such that $X_1, \ldots , X_m, Y_1, \ldots ,Y_l$ 
form a foundation set for $X$ and $g p_{Y_j}=p_{Y_j}$ for all $j$. The reason the collection $\{Y_1, \dots ,Y_l\}$ afforded to us cannot be empty is due to Proposition~\ref{prop:foundation-maximal}.
By the same proposition we must have $\chi(Y_j)=1$ for some $j$, hence $[g, \chi]=\chi$.\\

Assume now that the condition formulated in the lemma is not satisfied. Choose $g\in 
I_\ell(S)$ and $X, X_1, \ldots, X_m \in \J(S)$, $X_i\subseteq X$ for all $i$, which witness the negation of that condition. Then $gp_{X\setminus \cup_{i=1}^m X_i}=p_{X\setminus \cup_{i=1}^m X_i}$ and denoting by $\F$ the set of all finite  collections of constructible ideals such that for each member~$Y$, we have $Y\subseteq X$ and $gp_Y=p_Y$, then for any $F\in \F$, $\{X_1, \dots ,X_m\}\cup F$ is not a foundation set for $X$. In other words, the clopen set
\begin{equation}
U_F:=\{\eta\in\Omega(S)\mid \eta(X)=1,\ \eta(X_i)=0\ \text{for}\ i=1,\dots,m,\ \eta(Y)=0\ \text{for}\ Y\in F\}
\end{equation}
must intersect $\partial\Omega$ and by compactness there exists $\chi \in \bigcap_{F\in \F} U_F\cap \partial\Omega$. Then $\chi(X)=1$, $\chi(X_i)=1$ for all $i$ and $\chi(Y)=0$ for all $Y\in \J(S)$ satisfying $Y\subseteq X$ and $gp_{Y}=p_{Y}$. So if $A\in \J(S)$ satisfies $gp_{A}=p_{A}$ then $\chi(A\cap X)=0$ which implies $\chi(A)=0$. Then $[g, \chi]\sim_2 \chi$, but $[g, \chi]\neq \chi$, hence $\partial\G_P(S)\neq \partial\G(S)$.
\end{proof}
\begin{remark}\label{rem: strong-regularity-implies-equality-of-boundary-groupoids}
The condition in Proposition \ref{prop:equality-of-boundary-groupoids} is the same as the condition in strong C$^*$-regularity on the boundary with $n=1$. In particular, it is weaker than strong C$^*$-regularity on the boundary.
\end{remark}
\begin{remark}
The argument in Remark \ref{rem:strong-regularity-on-boundary-remark-2} shows that in the above proposition we may in addition assume that $X\subseteq \dom g$.
\end{remark}

Later we shall see examples of left cancellative monoids $S$ for which $\partial\G_P(S)\neq \partial \G(S)$. In particular this implies that $\G_P(S)\neq \G(S)$.\\

\subsection{The Diagonal Map}
We finish this section by proving that, on the diagonal, the bottom maps in diagrams~\eqref{eq:commutativediagram} and \eqref{eq:commutativediagram-G_2}, which coincide thereon, restrict to isomorphisms.\\

Recall that for any open bisection $U$ in an étale groupoid $\G$ we have $\|f\|_r=\|f\|_u=\|f\|_\infty$ for $f\in C_c(U)$. Here $\|\cdot\|_u$ denotes the norm on $C^*(\G)$ and $\|\cdot\|_\infty$ the usual supremum norm. It follows that $C_0(\Gu)\subseteq C_e^*(\G)$ for any C$^*$-norm $e$ on $C_c(\G)$. The algebra $C_0(\Gu)$ will be referred to as the \textit{diagonal} of the groupoid C*-algebra. Thus $C(\Omega(S))$ is the diagonal of both $C_r^*(\G_P(S))$ and $C_r^*(\G(S))$.\\

On the semigroup side there is a diagonal subalgebra $D_r^*(S)\subseteq C_r^*(S)$, defined by
$$D_r^*(S):=C^*\{\un_X \mid X\in \J(S)\}\subseteq C_r^*(S).$$
By construction $C(\Omega(S))$ is canonically isomorphic to $D_r^*(S)$. The image of $D_r^*(S)$ under the quotient map $q: C_r^*(S) \rightarrow C_r^*(S)/I$ is $D_r^*(S)/(D_r^*(S)\cap I)$. In the end, the bottom maps in diagrams \eqref{eq:commutativediagram} and \eqref{eq:commutativediagram-G_2} restrict to the canonical map
\begin{equation}\label{eq:diagonal-map}
C(\partial\Omega)\rightarrow D_r^*(S)/(D_r^*(S)\cap I),
\end{equation}
on the diagonal. Under the isomorphism $C(\Omega(S))\cong D_r^*(S)$, the ideal $D_r^*(S)\cap I$ is defined by the vanishment on a closed invariant subset $K\subseteq \Omega(S)$,
$$K=\{\eta\in \Omega(S) \mid f(\eta)=0 \ \text{for all }f\in D_r^*(S)\cap I\}.$$
Then $D_r^*(S)/(D_r^*(S)\cap I)\cong C(K)$ and to show that the map~\eqref{eq:diagonal-map} is an isomorphism we need only show that $K=\partial\Omega$. Merely the existence of~\eqref{eq:diagonal-map} implies that $K\subseteq \partial\Omega$. To establish equality the next proposition will be useful. The argument is due to Xin Li, but we include it since it simplifies in the monoid case.

\begin{proposition}[cf.~{\cite[Proposition~5.21]{Li-1}}]
Let $S$ be a left cancellative monoid. Then $\partial\Omega(S)$ is the smallest nonempty closed invariant subset of the groupoid $\G_P(S)$ or $\G(S)$.
\end{proposition}
\begin{proof}
Let $\chi\in \Omega_{\text{max}}(S)$ and $\eta\in \Omega(S)$. For every $A\in \J(S)$ such that $\chi(A)=1$, take $s_A\in A$ and put $\chi_A=r([s_A, \eta])$. Then $\chi_A(A)=\eta({s_A}^{-1}A)=\eta(S)=1$. We claim that $\chi \in \overline{\{\chi_A \mid \chi(A)=1\}}$ which will complete the proof. For this, let $U$ be a neighboorhood of $\chi$ of the form~\eqref{eq:U-neighbourhood2}. By maximality there exists $Y_i\in \J(S)$ disjoint from $X_i$ with $\chi(Y_i)=1$. Put $A=X\cap \bigcap_{i=1}^m Y_i$. Then $\chi(A)=1$ and~$\chi_A\in U$.
\end{proof}

By the proposition, to establish that $K=\partial\Omega$, it suffices to check that $K$ is nonempty. The next lemma implies this, thus completing the proof that the map~\eqref{eq:diagonal-map} is an isomorphism.

\begin{lemma}\label{lem:I-does-not-contain-diagonal}
Let $S$ be a left cancellative monoid. For any nonempty constructible ideal \mbox{$A\in \J(S)\setminus \{\emptyset\}$} we have $\un_A \notin I$. In particular
$$I\cap D_r^*(S)\subsetneq D_r^*(S).$$
\end{lemma}
\begin{proof}
Let us generalize some previous notation. For a subset $B\subseteq S$, let $p_B$ denote the identity function on $B$, viewed as a partial bijection on $S$. For a partial bijection $f$ on $S$, let $\lambda_f\in B(\ell^2(S))$ denote the operator determined by
\begin{equation}
\lambda_f(\delta_s)=\begin{cases}
\delta_{fs} & \text{if }s\in\dom f,\\
0 & \text{otherwise.}
\end{cases}
\end{equation}
Then $I$ is densely spanned by operators of the form
\begin{equation}\label{eq:operatorsinI}
\lambda_g \un_{X\setminus \bigcup\limits_{i=1}^m X_i} \lambda_h=\lambda_{gp_{X\setminus \bigcup\limits_{i=1}^m X_i}h}\neq 0
\end{equation}
where $g, h\in I_\ell(S)$ and $\{X_1, \ldots, X_m\}$ is a foundation set for $X$.
Now
$$\dom(gp_{X\setminus \bigcup\limits_{i=1}^m X_i}h)\subseteq \dom(p_{X\setminus \bigcup\limits_{i=1}^m X_i}h)=h^{-1}(X\setminus \bigcup\limits_{i=1}^m X_i)=h^{-1}X\setminus \bigcup\limits_{i=1}^m h^{-1}X_i.$$
Therefore
\begin{align}\label{eq:kernelcontains}
\overline{\text{span}}\left\{\delta_s \mid s\notin h^{-1}X \setminus \bigcup\limits_{i=1}^m h^{-1}X_i\right\}
&\subseteq 
\overline{\text{span}}\left\{\delta_s \mid s\notin \dom\left(gp_{X\setminus \bigcup\limits_{i=1}^m X_i}h\right)\right\} \notag \\
&= \ker \left(\lambda_g \un_{X\setminus \bigcup\limits_{i=1}^m X_i} \lambda_h\right)
\end{align}

We assert that $\{h^{-1}X_1, \ldots ,h^{-1}X_m\}$ is a foundation set for $h^{-1}X$. Indeed, suppose that there exists $Y\in \J(S)$ such that
$$\emptyset \neq Y\subseteq h^{-1}X \setminus \bigcup\limits_{i=1}^m h^{-1}X_i.$$
Applying $h$ to both sides yields
$$\emptyset \neq hY \subseteq h(h^{-1}X\setminus \bigcup\limits_{i=1}^m h^{-1}X_i)=h(h^{-1}(X\setminus \bigcup\limits_{i=1}^m X_i))\subseteq X\setminus \bigcup\limits_{i=1}^m X_i$$
which is a contradiction. In conclusion, if $T$ is in the span of operators of the form~\eqref{eq:operatorsinI}, there are finitely many differences $X^j\setminus 
\bigcup_{i=1}^{m_j}X_i^j$, $j=1, \ldots, n$, of constructible ideals 
such that $\{X_1^j, \ldots ,X_{m_j}^j\}$ is a foundation set for $X^j$, and
$$\overline{\text{span}}\left\{\delta_s \mid s\notin X^j\setminus 
\bigcup\limits_{i=1}^{m_j}X_i^j \text{ for }j=1, \ldots, n\right\}\subseteq \ker(T).$$
Then Proposition~\ref{prop:noncovering-of-foundation-differences} implies that for any nonempty $A\in \J(S)$ there is an $s\in A$ such that $T(\delta_s)=0$. Then $\|\un_A-T\|\geq 1$. Since this holds for any $T$ in the linear span of operators of the form~\eqref{eq:operatorsinI}, we have $\un_A\notin I$.
\end{proof}

\section{A non-Regular Monoid which is Strongly Regular on the Boundary}\label{sec:example}

Consider the monoid $R$ with identity $e$ given by the monoid representation
\begin{align}
R=\langle a, b, c, d, f, x_n, y_n \ (n\in\mathbb{Z}) \mid & \ \ abx_n=bx_n,\ aby_n=by_{n+1},\nonumber\\
 &\ \ cbx_n=bx_{n+1},\  cby_n=by_n,\nonumber\\ 
 &\ \ dbx_nx=bx_nx,\  fby_nx=by_nx \nonumber\\  
 &\ \ (n\in \mathbb{Z},\ x\in\{a, b, c, d, f, x_k, y_k \ (k\in\mathbb{Z})\})\rangle.\label{eq:monoid-R}
\end{align}
We note that the last two relations are equivalent to $$dbx_nw'=bx_nw' \quad \text{and} \quad fby_nw'=by_nw'$$ for all $w'\in R\setminus \{e\}$ and $n\in \Z$.\\

The goal of this section is to show the following.

\begin{theorem}\label{thm:Rproperties}
The monoid $R$ defined by~\eqref{eq:monoid-R} is left cancellative, strongly C$^*$-regular on the boundary and the groupoid $\partial\G_P(R)=\partial\G(R)$ is Borel amenable. The bottom maps in diagrams~\eqref{eq:commutativediagram} and~\eqref{eq:commutativediagram-G_2} coincide and define an isomorphism $C_r^*(\partial\G(R)) \cong C_r^*(R)/I$. $R$ is not C*-regular and the map $(\rho_{\chi_e})_2: C_r^*(\G(R))~\rightarrow~C_r^*(R)$ is not an isomorphism.
\end{theorem}
Note that by Proposition~\ref{prop:equality-of-boundary-groupoids}, the strong C$^*$-regulariy of $R$ on the boundary will imply that $\partial\G_P(R)=\partial\G(R)$. 
The proof of the theorem requires several lemmas. Some are proved in this section, others are proved in the appendix. The main challenge is to prove that the boundary groupoid $\partial\G_P(R)=\partial\G(R)$ is Borel amenable, which will imply that the $e$ and $e_2$ norms from the short exact sequences \eqref{eq: short-exact-boundary} and \eqref{eq:short-exact-sequences} are the reduced norm. Then Corollary~\ref{cor:iso-sufficient-condition-2} implies that the map $C_e^*(\partial\G(R))=C_r^*(\partial\G(R))\rightarrow C_r^*(R)/I$ is an isomorphism.\\

\subsection{Cancellation Property}
We begin by proving that $R$ is left cancellative.  Consider the set $\RR$ of finite words in the alphabet $\{a,b,c, d, f, x_n,y_n\ (n\in\Z)\}$. We include the empty word in $\RR$ and denote it by $\epsilon$. For a word $s$ and letter $x$, let $m(x: s)$ denote the number of occurences of $x$ in $s$. We write $\con(s)$ for the set of letters which occur in $s$, and~$\alpha(s)$ and $\beta(s)$ for the first and last letter of $s$. If $s$ begins with the word $t$, $s- t$ denotes the word obtained by removing $t$ from $s$. If $s$ ends with the word $t$, $s-_Rt$ denotes the word obtained by removing $t$ from $s$.\\

We say that two words $s$ and $t$ are equivalent, and write $s\sim t$ or $s=t$, if they represent the same element of $R$. If $s$ and $t$ are equal as words, we write $s\equiv t$. Let $\tau\subset\RR\times\RR$ be the symmetric set of relations defining~$R$,~so
\begin{align}
\tau=\{(abx_n,bx_n),(bx_n,abx_n),(aby_n,by_{n+1}),(by_{n+1},aby_n),\nonumber \\
 (cbx_n,bx_n),(bx_n,cbx_n),(cby_n,by_{n+1}),(by_{n+1},cby_n),\nonumber \\
 (dbx_nx,bx_nx),(bx_nx, dbx_nx),(fby_nx,by_nx),(by_nx,fby_nx)\nonumber \\
 (n\in\Z,\ x\in\{a, b, c, d, f, x_n, y_n \ (n\in\mathbb{Z})\})\}.
\end{align}

We refer to words in the defining relations of $R$, $abx_n, bx_n, aby_n,$ etc. as \emph{$\tau$-words}. By a \emph{~$\tau$-sequence} we mean a finite sequence $s_0,\dots,s_n$ of words such that for every $i=1,\dots, n$ we can write $s_{i-1}\equiv c_ip_id_i$ and $s_i\equiv c_iq_id_i$ with $(p_i,q_i)\in\tau$. Then by definition, $s\sim t$ if and only if there is a $\tau$-sequence $s_0,\dots,s_n$ with $s_0\equiv s$ and $s_n\equiv t$. If there exists such a $\tau$-sequence of length one (i.e. $n=1$) we say that $s$ and $t$ are elementary equivalent.\\

Adopting notation from Paper $1$, we shall write $s \perp_0 t$ if every $\tau$-word in $st$ which begins in $s$, ends in $s$. We write $s \perp t$ if for all words $s'$ and $t'$ such that $s\sim s'$, $t \sim t'$, we have $s'\perp_0 t'$. For example, we have
$$a \perp b^2ac, \quad bx_nb \perp bx_k, \quad f \perp bx_n.$$
Also note that $bx_n \not\perp a$, for instance. We will repeatedly use that if $s\perp t$ and $st\sim w$ for a word $w$, then $w\equiv s't'$ for some $s'\sim s$ and $t' \sim t$.

We have the following:
\begin{equation}\label{eq:equiv-bx_n}
[bx_n]=\{rbx_{n-m(c:r)} \mid \con(r)\subseteq\{a,c\}\},
\end{equation}
\begin{equation}\label{eq:equiv-by_n}
[by_n]=\{rby_{n-m(a:r)} \mid \con(r)\subseteq\{a,c\}\}.
\end{equation}
For $w\neq \epsilon$ we have
\begin{equation}\label{eq:equiv-bx_nw}
[bx_nw]=\{sbx_{n-m(c:s)}w' \mid \con(s)\subseteq\{a,c,d\}, \ w'\sim w\},
\end{equation}
\begin{equation}\label{eq:equiv-by_nw}
[by_nw]=\{sby_{n-m(a:s)}w' \mid \con(s)\subseteq\{a,c, f\}, \ w'\sim w\}.
\end{equation}

We limit ourselves to proving Equation~\eqref{eq:equiv-bx_nw}, by symmetry this also 
proves~\eqref{eq:equiv-by_nw}. Clearly the RHS of~\eqref{eq:equiv-bx_nw} is 
contained in the LHS and we can focus on the reverse inclusion. Since 
$bx_nw\in \text{RHS}$ it suffices, by induction on the length of a $\tau$-sequence, to see that the RHS is closed with respect to elementary equivalence.

Suppose $w' \sim w$ and $sbx_{n-m(c:s)}w'\equiv z'pz''$, where $\con(s)\subseteq \{a, c,d\}$ and $(p,q)\in \tau$. We must show that $z'qz''$ is of the same form. If $p$ 
begins in $w'$ then $z'qz''\equiv sbx_{n-m(c:s)}w''$ where $w''\sim w'\sim w$ as 
desired. If $p$ does not begin in $w'$ then either $p\equiv bx_{n-m(c:s)}$, $p\equiv abx_{n-m(c:s)}$, $p\equiv~cbx_{n-m(c:s)}$, $p\equiv dbx_{n-m(c:s)}\alpha(w)$  or $p \equiv bx_{n-m(c:s)}\alpha(w)$. Then we have the following possibilities:
\begin{align*}
z'qz'' &\equiv sabx_{n-m(c:s)}w',         &\quad z'qz'' &\equiv scbx_{n-m(c:s)-1}w', \\
z'qz'' &\equiv sdbx_{n-m(c:s)}w',         &\quad z'qz'' &\equiv (s-_Ra)bx_{n-m(c:s)}w', \\
z'qz'' &\equiv (s-_Rc)bx_{n-m(c:s)+1}w',  &\quad z'qz'' &\equiv (s-_Rd)bx_{n-m(c:s)}w'.
\end{align*}
In any case $z'qz''$ is of desired form. Hence~\eqref{eq:equiv-bx_nw} holds.

\begin{lemma}\label{lem:R-cancellation}
The monoid $R$ is left cancellative.
\end{lemma}
\bp
It suffices to show that for all letters $x$ and words $w$ and $w'$, the equivalence $xw\sim xw'$ implies that $w\sim w'$.\\

\underline{Case $x=x_n, y_n$:}\\
The only word equivalent to $x$ is $x$ itself, and we have $x\perp w$, so the equivalence $xw\sim xw'$ furnishes the equivalence $w\sim w'$.\\

\underline{Case $x=b$:}\\
If $w=\epsilon$ or $\alpha(w)=a, b, c, d, f$, then $x\perp w$ and we are done similarly to the first case. If $\alpha(w)=x_n, y_n$ we may by symmetry focus on the former case and write $xw\equiv bx_nv$.\\
\indent If $v=\epsilon$, then~\eqref{eq:equiv-bx_n} implies that $xw'\equiv rbx_{n-m(c:r)}$ where $\con(r)\subseteq\{a,c\}$. The only such word beginning in $b$ is $bx_n$ itself. Then $w'=x_n$ and we are done.\\
\indent If $v\neq \epsilon$, then~\eqref{eq:equiv-bx_nw} implies that $xw'\equiv sbx_{n-m(c:s)}v'$ with $\con(s)\subseteq \{a, c, d\}$ and $v'\sim v$. The only such words that begin with $b$ are $bx_nv'$ with $v' \sim v$. Then $w\equiv x_nv\sim x_nv'\equiv w'$.\\

\underline{Case $x=a, c, d, f$:}\\
Write $xw\equiv sv$ where $\con(s)\subseteq \{a, c, d, f\}$ and $\alpha(v)\neq a, c, d ,f$. If $v=\epsilon$ or $\alpha(v)=x_n ,y_n$, then $xw\equiv s\perp v$ and we are done since $s$ is only equivalent to itself. Otherwise we can write $xw\equiv sbu$. If  $u=\epsilon$ or $\alpha(u)=a, b, c, d, f$ we are similarly done and otherwise we can write $xw\equiv sbx_n \omega$ or $xw\equiv sby_n\omega$. By symmetry we may focus on the former case. Then, if the letter $f$ occurs in $s$ we can write $s\equiv s_1fs_2$. Then $xw\equiv s_1f\perp s_2bx_n\omega$ and are done, similar to before. Hence we may assume that $f$ does not occur in $s$, i.e. that $\con(s)\subseteq \{a, c, d\}$.\\

Suppose that $\omega=\epsilon$. If the letter $d$ occurs in $s$ we can write $s\equiv s_1ds_2$ and then $xw\equiv s_1d\perp s_2bx_n$, so we are done. If $d$ does not occur in $s$ then by~\eqref{eq:equiv-bx_n}, $xw'\equiv s'bx_{n+m(c:s)-m(c:s')}$ where $\con(s')\subseteq \{a, c\}$ and $s'\neq \epsilon$. The same equation then shows that $w\equiv (s- x)bx_n\sim (s'-x)bx_{n+m(c:s)-m(c: s')}\equiv ~w'$.\\

If $\omega\neq \epsilon$ then by~\eqref{eq:equiv-bx_nw} we have that $xw'\equiv s'bx_{n+m(c:s)-m(c: s')}\omega'$ where $\con(s')\subseteq \{a, c, d\}$ $\omega'\sim \omega$. The same equation then shows that $w\equiv (s- x)bx_n\omega\sim (s'-x)bx_{n+m(c:s)-m(c: s')}\omega'\equiv w'$.\\
\ep

In the appendix we prove that constructible right ideals of $R$ consists of precisely the sets:
\begin{equation}\label{eq:R-constructible-ideals-first}
wR,\text{ }w(\bigcup\limits_n x_nR\cup \bigcup\limits_n y_nR), \text{ }w(\bigcup\limits_{n, w'\neq e} x_nw'R), \text{ }w(\bigcup\limits_{n, w'\neq e} y_nw'R), \text{ }w(R\setminus\{e\}), \text{ }\emptyset \quad (w\in R).
\end{equation}
To prove that $R$ is not C$^*$-regular we need the technical lemma which is proved in the appendix.

\begin{lemma}\label{lem:multiple x_n elements}
If $r\in R$ and a constructible ideal $A\in \J(R)$ contains two $rx_n$ (resp. $ry_n$) elements, then $A$ contains all $rx_n$ and $ry_n$ elements. If $r, t_1, t_2\in R$ and a constructible ideal of $R$ contains $rx_{n_1}t_1$ and $rx_{n_2}t_2$ (resp. $ry_{n_1}t_1$ and $ry_{n_2}t_2$), $n_1\neq n_2$, then it contains the constructible ideal $r\Big(\bigcup\limits_{n, w'\neq e} x_nw'R\Big)$ (resp. $r\Big(\bigcup\limits_{n, w'\neq e} y_nw'R\Big)$).
\end{lemma}

\begin{proposition}
The canonical map $(\rho_{\chi_e})_2: C_r^*(\G(R)) \rightarrow C_r^*(R)$ is not an isomorphism. In particular, $R$ is not C*-regular.
\end{proposition}
\bp
Let $X=\bigcup\limits_n bx_nR\cup \bigcup\limits_n by_n \in \J(R)$ and let $U=\{\chi\in \Omega(R) \mid \eta(X)=1\}$. Consider the function
$$f=(\un_{D(a, \Omega(R))}-\un_{\Omega(R)})*(\un_{D(c, \Omega(R))}-\un_{\Omega(R)})*\un_U\in C_c(\G(R)).$$
Then $\rho_{\chi_e}(f)=(\lambda_a-1)(\lambda_c-1)\un_X=0$. If we define $\chi=\lim_n \chi_{bx_n}$ then $\chi\in U$ by Lemma~\ref{lem:multiple x_n elements} and $f(\chi)=1$, so $f\neq 0$.
\ep

One can also verify that $R$ is not C*-regular from the definition.

\begin{proposition}
$R$ is not C*-regular.
\end{proposition}
\begin{proof}
Let $X=\bigcup\limits_{n\in \mathbb{Z}} bx_nR\cup \bigcup\limits_{n\in \mathbb{Z}}by_nR\in \J(R)$ and $h_1=a$, $h_2=c$. Then
$$X\subseteq \bigcup\limits_{k=1}^2 \{s \mid h_k s=s\}.$$
Now suppose $X\subseteq \cup_{j=1}^l Y_j$ for some $Y_1, \ldots , Y_l\in \bar{\J}(R)$. Then there is a $Y_{j'}$, containing at least two $bx_n$ elements. Write
$$Y_{j'}=A\setminus \bigcup\limits_{i=1}^m A_k$$
for some $A, A_1, \dots ,A_m \in \J(R)$. By Lemma~\ref{lem:multiple x_n elements}, $A$ contains all~$bx_n$ and $by_n$ elements. By the same lemma, each $A_i$ contains at most one $bx_n$ and $by_n$ element. Then $Y_{j'}$ certainly contains both $bx_n$ and $by_n$ elements; in particular $Y_{j'}$ is not fixed by $h_1$ or $h_2$.
\end{proof}

\subsection{Regularity on the Boundary}
The next post on the program is to prove that $R$ is strongly C$^*$-regular on the boundary. We require a technical result which we prove in the appendix. By a generator (of $R$) we mean a letter from the generating set $\{a, b, c, d, f, x_n, y_n \ (n\in \Z)\}$.

\begin{lemma}\label{lem:fix x_n or x_nx}
If $h\in I_\ell(R)$ fixes an $x_n$ element and $\dom h$ contains multiple $x_n$ 
elements, then $h$ fixes all $x_n$ elements. If $x$ is a generator, $h\in I_\ell(R)$ fixes an 
$x_nx$ element and $\dom h$ contains multiple~$x_nx$ elements, then $h$ 
fixes the constructible ideal $\bigcup\limits_{n, w'\neq e} x_nw'R$.
\end{lemma}

\begin{corollary}\label{lem:shifts x_n or x_nx}
If $h\in I_\ell(R)$ maps $x_{n_0}$ to $x_{n_0+k}$ for some $n_0, k\in\mathbb{Z}$, and $\dom h$ contains multiple~$x_n$ elements, then $h$ maps $x_n$ to $x_{n+k}$ for all $n\in \mathbb{Z}$.\\

If $h\in I_\ell(R)$ maps $x_{n_0}x$ to $x_{n_0+k}x$ for some $n_0, k\in\mathbb{Z}$ and generator $x$, and $\dom h$ contains multiple~$x_nx$ elements, then $h$ maps $x_nw'$ to $x_{n+k}w'$ for all $n\in \mathbb{Z}$ and $w'\neq e$.
\end{corollary}
\bp
Suppose $h$ satisfies the hypothesis in the first part. Then $b^{-1}c^{-k}bh\in I_\ell(R)$ fixes $x_{n_0}$ and its domain contains multiple $x_n$ elements. Therefore it fixes all $x_n$, $n\in \mathbb{Z}$. Now,
$$(b^{-1}c^{-k}bh)x_n=x_n$$
$$\Rightarrow hx_n=b^{-1}c^kbx_n=x_{n+k}$$

Now suppose $h$ satisfies the hypothesis in the second part. Then $b^{-1}c^{-k}bh \in I_\ell(R)$ fixes $x_{n_0}x$ and its domain contains multiple $x_nx$ elements. Therefore it fixes all elements $x_nw'$, $n\in \mathbb{Z}$, $w'\neq e$. Now,
$$(b^{-1}c^{-k}bh)x_nw'=x_nw'$$
$$\Rightarrow hx_n=b^{-1}c^kbx_nw'=x_{n+k}w'.$$
\ep

\begin{proposition}\label{prop:Rstronglyboundaryregular}
$R$ is strongly C*-regular on the boundary.
\end{proposition}
\bp
By Equation \eqref{eq:R-constructible-ideals-first} it suffices to show that $R$ satisfies the condition in Definition~\ref{def:strong-regularity-on-boundary} when $X$ is of the form
\begin{equation}\label{eq:Xsimpleform}
R,\text{ }\bigcup\limits_n x_nR\cup \bigcup\limits_n y_nR, \bigcup\limits_{n, w'\neq e} x_nw'R, \bigcup\limits_{n, w'\neq e} x_nw'R,\text{ }R\setminus{\{e\}}
\end{equation}
and $X\subseteq \dom h_k$ for all $k$. Indeed, suppose that $X\in \J(R)$ has the property that whenever $X$, $h_1, \ldots h_n\in I_\ell(S)$ and $X_1, \ldots ,X_m\in \J(S)$ satisfy~\eqref{eq:hypstronggreg} and $X\subseteq \dom h_k$ for all $k$, then the conclusion in the definition holds. Let us show that then the same is true for $rX$, for any $r\in R$. Let $X_1, \ldots, X_m\in \J(R)$ and $h_1, \ldots h_n\in I_\ell(R)$ be such that
$$(rX)\setminus \bigcup\limits_{i=1}^m X_i \subseteq \bigcup\limits_{k=1}^n\{s \mid h_ks=s\}$$
and $rX\subseteq \dom h_k$ for all $k$. Then
\begin{equation}\label{eq:pconjugate}
X\setminus \bigcup\limits_{i=1}^m r^{-1}X_i\subseteq \bigcup\limits_{k=1}^n\{s \mid (r^{-1}h_kr)s=s\}
\end{equation}
and $X \subseteq \dom (r^{-1}h_kr)$. By assumption we are granted $Y_1, \ldots Y_l \in \mathcal{J}(R)$ such that each $Y_j$ is fixed by some $r^{-1}h_{k_j}r$ and $r^{-1}X_1,\ldots r^{-1}X_m, Y_1, \ldots ,Y_l$ form a foundation set for $X$. Then each~$rY_j$ is fixed by $h_{k_j}$ and to show that the condition in Definition~\ref{def:strong-regularity-on-boundary} is met we need to argue that $X_1, \ldots ,X_m, rY_1, \ldots rY_l$ forms a foundation set for $rX$. But we already wrote down such an argument in the proof of Lemma~\ref{lem:I-does-not-contain-diagonal}.\\

Now suppose we are given $X$ of the form~\eqref{eq:Xsimpleform}, $h_1, \ldots, h_n\in I_\ell(R)$ and $X_1, \ldots ,X_m\in \J(R)$ such that~\eqref{eq:hypstronggreg} is satisfied and $X\subseteq\dom h_k$ for all $k$. Consider the various subcases.\\

\underline{Case $X=R$:}\\
Since $R\setminus \cup_{i=1}^m X_i= X\setminus \cup_{i=1}^m X_i \neq \emptyset$, it contains $e$. Then some $h_k$ fixes $e$, so $h_k=p_R$ and $\{R\}$ is the required foundation set for $X=R$.\\

\underline{Case $X=R\setminus\{e\}$:}\\
We construct a collection $\F$ of constructible ideals as follows.\\

\indent If $X\setminus \cup_{i=1}^m X_i$ contains $a, b, c, d$ or $f$, respectively,  then we add, respectively, $aR$, $bR$, $cR$, $dR$ or $fR$ to $\mathcal{F}$. If $X\setminus \cup_{i=1}^m X_i$ contains an $x_n$ or $y_n$ element, then we add both $\cup_{n, w'\neq e} x_nw'R$ and $\cup_{n, w'\neq e} y_nw'R$ to $\F$.\\

We argue that each member of $\F$ is fixed by some $h_k$: if $xR\in \mathcal{F}$, 
$x=a,b,c,d,f$, then $x\in X\setminus \cup_{i=1}^m X_i$, so there is an $h_k$ which fixes $x$. If $\cup_{n, w'\neq e} 
x_nw'R,\ \cup_{n, w'\neq e} y_nw'R\in \mathcal{F}$, then $
X\setminus \cup_{i=1}^m X_i$ contains either an $x_n$ 
or $y_n$ element, in particular none of the $X_i$'s contain all $x_n$ and $y_n$ elements, so by Lemma~\ref{lem:multiple x_n elements} they contain at most one $x_n$ or $y_n$ element. Thus $
X\setminus \cup_{i=1}^m X_i$ contains cofinitely 
many $x_n$'s and $y_n$'s. In particular there is an $h_k$ which fixes an $x_n$ and an $h_{k'}$ which fixes a $y_n$. By Lemma~\ref{lem:fix x_n or x_nx}, $h_k$ 
fixes all $x_n$'s and $h_{k'}$ fixes all $y_n$'s.\\

By construction, if $xR\notin\mathcal{F}$ for $x\in \{a,b,c,d,f\}$ then $x\in X_i$ for some $i$. In particular $xR\subseteq X_i$. Similarly, if $\cup_{n, w'\neq e} x_nw'R$ and $\cup_{n, w'\neq e} y_nw'R$ do not lie in $\mathcal{F}$, then we have $\bigcup_n x_nR\cup\bigcup_n y_nR\subseteq~X_i$ for some $i$. This shows that
$$aR\cup bR\cup cR\cup dR\cup fR\cup \bigcup\limits_{n, w'\neq e} x_nw'R\cup \bigcup\limits_{n, w'\neq e} y_nw' \subseteq \bigcup\limits_{F\in \F} F \cup \bigcup\limits_{i=1}^m X_i.$$
Hence $\mathcal{F}\cup \{X_1, \ldots , X_m\}$ is a foundation set for $X=R\setminus \{e\}$.\\

\underline{Case $X=\bigcup_{n, w'\neq e} x_nw'R$:}\\
In this case $X\setminus \cup_{i=1}^m X_i$ surely contains some $x_nx$ element, and the $h_k$ element which fixes $x_nx$ fixes~$X$ by Lemma~\ref{lem:fix x_n or x_nx}.\\

\underline{Case $X=\bigcup_n x_nR\cup \bigcup_n y_nR$:}\\
By Lemma~\ref{lem:multiple x_n elements}, $X\setminus \cup_{i=1}^m X_i$ must contain cofinitely many $x_n$ and $y_n$. By Lemma~\ref{lem:fix x_n or x_nx}, the $h_k$ which fixes any one $x_n$ fixes all $x_n$'s and the $h_{k'}$ which fixes any one $y_n$ fixes all $y_n$'s. Then $\{\bigcup_{n, w'\neq e} x_nw'R, \bigcup_{n, w'\neq e} y_nw'R\}$ is the required foundation set.
\ep

As per Remark \ref{rem: strong-regularity-implies-equality-of-boundary-groupoids}, an immediate consequence of the above proposition is that $\partial \G(R)=\partial \G_P(R)$.

\subsection{Amenability of the Boundary Groupoid}
The goal of this subsection is to prove that the boundary groupoid $\partial \G(R)$ is Borel amenable. To achieve this we require a description of the boundary characters~$\partial \Omega(R)$. Once this is in place, proving Borel amenability is similar to the free monoid case.\\

As expected, some of the boundary characters are given by infinite words.
\begin{definition}\label{def: infinite-word}
An infinite word $w$ in $R$ is a sequence $\{w_i\}_{i\in 
\mathbb{N}}$ of generators of $R$, i.e. a sequence of letters from 
the alphabet $\{a, b, c, d, f, x_n, y_n \mid n\in \mathbb{Z}
\}$. We write $w\equiv 
w_1w_2\cdots w_n \cdots$.  The collection of infinite words in $R$ 
is denoted $\RR^\infty$.
\end{definition}
Recall that the collection of finite words in the generators of $R$ was denoted~$\RR$. We have a left action $\RR\curvearrowright \RR^\infty$ given by concatenation: for $u\equiv u_1 \cdots u_m\in \RR$ and $w\in \RR^\infty$, let
$$u\cdot w := u_1 \cdots u_m w_1 w_2 \cdots.$$
To $w\in \RR^\infty$ we associate the character
$$\chi_w := \lim_n \chi_{\text{\footnotesize $w_1 w_2 \cdots w_n$}} \in \Omega(R)$$
On the RHS $w_1 w_2\cdots w_n$ is viewed as an element of $R$. Note that this limit exists since the filters $\chi_{w_1\cdots w_n}^{-1}(1)$ are increasing in $n$. More transparently we have
$$\chi_w(A)=\begin{cases}
1 & \text{ if }w_1\cdots w_n\in A \text{ for sufficiently large }n,\\
0 & \text{ otherwise.}
\end{cases} \quad A\in\J(S)$$

We call a finite or infinite word $w$ \textit{reduced} if in $w$ 
there are no occurences of the strings $abx_n, aby_n$, $cbx_n$, 
$cby_n$,  $dbx_n$ or $fby_n$ for any $n\in \mathbb{Z}$. The 
collection of infinite reduced words will be denoted~$\RR_{\text{red}}^\infty$.\\

\begin{proposition}\label{prop:distinctfinitereducedwords}
Suppose $u, v\in \RR$ are two finite reduced words such that $u\sim v$. Then $u\equiv v$.
\end{proposition}
\bp
Write $u\equiv u_1\cdots u_m$ and $v\equiv v_1 \cdots v_l$ where the $u_i$'s and $v_i$'s are letters. $R$ is left cancellative, so by induction it will suffice to verify that $u_1= v_1$.\\

Assume that $u_1\neq b$. Then we can write $u\equiv ru_k\cdots u_m$ where $\con(r)\subseteq~\{a, c, d, f\}$ and $u_k=b, x_n, y_n$. If $u_k=x_n, y_n$, then $u\equiv ru_k \perp u_{k+1}\cdots u_m$ where $ru_k$ is only equivalent to itself. If $u_k=b$ and $u_{k+1}\notin \{x_n, y_n \mid n\in \Z\}$ then again $u\equiv rb \perp u_{k+1}\cdots u_m$ where $rb$ is only equivalent to itself. If $u_k=b$ and $u_{k+1}=x_n$, then since $u$ is reduced, we must have $u_{k-1}=f$ and $u\equiv r \perp bu_{k+1}\cdots u_n$ where $r$ is only equivalent to itself. Similarly if $u_k=b$ and $u_{k+1}=y_n$.\\

Now assume that $u_1=b$. If $u_2\notin \{x_n, y_n \mid n\in \Z\}$ we have $u\equiv b \perp u_2\cdots u_m$. If $u_2=x_n, y_n$, then equations~\eqref{eq:equiv-bx_n} to~\eqref{eq:equiv-by_nw} show that then, since $v$ is reduced, we also have have $v\equiv bu_2 \cdots$.
\ep

\begin{proposition}\label{prop:distinctinfinitereducedwords}
Suppose $w, v\in \RR_{\text{red}}^\infty$ are two distinct reduced infinite words. Then $\chi_w\neq \chi_v$. 
\end{proposition}
\begin{proof}
There exists a minimal $i_0\in \mathbb{N}$ such that $w_{i_0}\neq v_{i_0}$ and it suffices to consider the case $i_0=1$. Indeed, it will be valid to pass to this case if we can show that for a finite word $u$, $\chi_w \neq \chi_v \Rightarrow \chi_{uw}\neq \chi_{uv}$:\\

We show the contrapositive of this so assume $\chi_{uw}= \chi_{uv}$. Then, if $A\in \mathcal{J}(R)$ contains $w_1 \cdots w_n$ for sufficiently large $n$, $uA$ contains $uw_1 \cdots w_n$ for sufficiently large $n$ and by assumption $uA$ contains $uv_1\cdots v_{n'}$
for sufficiently large $n'$. Then~$A$ contains $v_1\cdots v_{n'}$ for sufficiently large $n'$. Hence $\chi_w(A)=1 \Rightarrow \chi_v(A)=1$. Similarly $\chi_v(A)=1 \Rightarrow \chi_w(A)=1$ so that $\chi_w=\chi_v$.\\

Now we assume that $w_1\neq v_1$ and show that then $\chi_w\neq \chi_v$. If $w_1\in~\{x_n,~y_n~\mid n\in~\mathbb{Z}\}$, then $w_1R$ does not contain $v_1\cdots v_n$ for any $n$ since then a word beginning in $w_1$ is never equivalent to a word not beginning in $w_1$. So $\chi_w(w_1R)=1$, but $\chi_v(w_1R)=0$.\\

Suppose that $w_1\in \{a, c, d, f\}$. We claim that
\begin{equation}\label{eq:vposs1}
w_1\perp w_2\cdots w_k \qquad \text{ for }k\geq 2.
\end{equation}
If $w$ does not contain $b$'s, $x_n$'s or $y_n$'s, then this is clear.
If $w$ contains $b$'s, $x_n$'s or $y_n$'s, choose $m\geq 2$ such that $w_i \in\{a, c, d ,f\}$ for $1\leq i \leq m-1$ and $w_m\in \{b, x_n, y_n \mid n\in \Z\}$. If $w_m=x_n, y_n$, or $w_m=b$ and $w_{m+1}\neq x_n, y_n$, then
\begin{equation}\label{eq:vposs2}
w_1\cdots w_{m-1}\perp w_m \cdots w_{k} \qquad \text{for }k\geq m.
\end{equation}
Otherwise we have $w_mw_{m+1}\equiv bx_n$, or $w_mw_{m+1}\equiv by_n$, and then, since $w$ is reduced, we must have $w_{m-1}=f$, resp. $w_{m-1}=d$, and~\eqref{eq:vposs2} again holds. Since $w_i \in\{a, c, d ,f\}$ for $1\leq i \leq m-1$,~\eqref{eq:vposs2} implies~\eqref{eq:vposs1}.  Then it is clear that $v_1R$ cannot contain $w_1 \cdots w_n$ for any $n$. So $\chi_w(v_1R)=0$, but $\chi_v(v_1R)=1$. This proves that $\chi_w\neq \chi_v$ in the cases $w_1\in \{x_n, y_n \mid n\in \mathbb{Z}\}$ and $w_1\in \{a, c, d, f\}$. By symmetry we are done.
\end{proof}

\begin{definition}
We say that a reduced infinite word $w\equiv w_1w_2\cdots w_n \cdots\in \RR_{\text{red}}^\infty$ is of \textbf{type} $\mathbf{1}$ if either of the following is satisfied:
\begin{itemize}
\item Infinitely many $w_i$'s belong to the set $\{b, x_n, y_n\mid n\in \mathbb{Z}\}$.
\item We have $w_i=d$ for infinitely many $i$ \textit{and} $w_i=f$ for infinitely many $i$. 
\end{itemize}

If $w$ is not of type $1$ we say it is of \textbf{type} $\mathbf{2}$ in which case we can write $w\equiv w'w''$, where $w'$ is a finite word and we have either $\con(w'')\subseteq \{a, c, d\}$ or $\con(w'')\subseteq \{a, c, f\}$.\\
\end{definition}
We need to describe another type of character. Let $S$ be an arbitrary left cancellative monoid. For a constructible ideal $A\in \mathcal{J}(S)$ we write $\langle A \rangle$ for the character corresponding to the principal filter of $A$. In other words,

$$\langle A\rangle (X)=\begin{cases}
1 \text{ if }X\supseteq A,\\
0 \text{ otherwise.}
\end{cases} \quad X\in \J(S).$$
\text{ }\\
Note that such a character need not belong to $\Omega(S)$. The lemma below implies that $\{\langle A \rangle \mid A\in \J(S),\ \langle A \rangle\in \Omega(S)\}$ is an invariant subset of $\G_P(S)$ or $\G(S)$.

\begin{lemma}\label{lem:range}
Let $S$ be a left cancellative monoid. Assume  that $A\in\mathcal{J}(S)$, $h\in I_\ell(S)$ are such that $\langle A \rangle\in \Omega(S)$ and $[h, \langle A \rangle]\in \G_P(S)$. Then the range of $[h, \langle A \rangle]$, either in the groupoid $\G_P(S)$ or $\G(S)$, is $\langle hA\rangle$.
\end{lemma}
\bp
Set $\chi=\langle A \rangle$. For $X\in \J(S)$ we have
$$r([h, \chi])(X)=\chi(h^{-1}X)=\begin{cases}
1 & \text{if } h^{-1}X \supseteq A,\\
0 & \text{otherwise.}
\end{cases}$$
We have then to prove that $h^{-1} X\supseteq A \Leftrightarrow X \supseteq h A$.\\

Assuming $h^{-1} X\supseteq A$ immediately yields $h(h^{-1} X)\supseteq h A$. Clearly then, $X\supseteq h(h^{-1}X)\supseteq h A$, as desired.\\

Conversely, assume that $X\supseteq h A$. Then $h^{-1}X\supseteq h^{-1}(h A)$. Since $\dom h \supseteq A$ by assumption, we have $h^{-1}X\supseteq h^{-1}(h A)=A$, as desired.
\ep

We are ready to state the proposition which describes the boundary characters~$\Omega(R)$.\\

\begin{proposition}\label{prop:boundary-characters-of-R}
We have
$$\Omega_{\text{max}}(R)=\{\chi_w \mid\text{ }w\in \RR^\infty_{\text{red}}\text{ is of type }1\}$$
and
$$\partial\Omega(R) \setminus \Omega_{\text{max}}(R)=\{\langle w\bigcup\limits_{n, w'\neq e}x_nw'R\rangle \mid w\in R\}$$ $$\sqcup \{\langle w\bigcup\limits_{n, w'\neq e}y_nw'R\rangle \mid w\in R\} \sqcup \{\chi_w \mid\ w\in \RR^\infty_{\text{red}}\text{ is of type }2\}.$$

Moreover, the three sets in the union are invariant subsets of $\partial \G(R)$.
\end{proposition}

To prove the proposition we need a series of lemmas. We omit the more technical ones here, but their formulations and proofs can be found in the appendix. We include the lemma below which states that if the characters corresponding to two infinite reduced words are sufficiently close, then they agree on the first $n$ letters, for any prescribed $n\geq 1$. The proof is in the appendix.

\begin{lemma}\label{lem:reducedinfinitewordsneighborhood}
Let $x=\chi_w$, where $w\in\RR_{\text{red}}^\infty$, and let $n\geq 1$. Then
\begin{align*}
U:=\{\eta\in \Omega(R) \mid \ & \eta(w_1\cdots w_nR)=1, \\ & \eta(w_1\cdots w_{n-1}b(\bigcup_{n'} x_{n'}R\cup \bigcup_{n'} y_{n'}R))=0 \text{ if }w_n\in\{a, c, d, f\},\\
& \eta(w_1\cdots w_{n-2}b(\bigcup_{n'} x_{n'}R\cup \bigcup_{n'} y_{n'}R))=0 \text{ if }w_{n-1}\in\{a, c, d, f\}\}
\end{align*}
is a neighbourhood of $x$ such that for any $\chi_v\in U$, where $v\in \RR_{\text{red}}^\infty$, we have $v_1 v_2\cdots v_n\equiv w_1 w_2\cdots w_n$.
\end{lemma}

We are almost ready to prove the proposition, but first we need some notions of "length" in $R$.

\begin{definition}\label{def: length-of-R-elements}
For an element $r\in R$, let
\begin{align*}
m_b(r)& :=\text{the number of }b\text{'s in }r,\\
m_x(r)& :=\text{the number of occurences of letters from the set }\{x_n \mid n\in \Z\} \text{ in }r,\\
m_y(r)&:=\text{the number of occurences of letters from the set }\{y_n \mid n\in \Z\} \text{ in }r, \\
m_a(r)&:=\min \{m(a:w)\text{ }|\text{ } w\text{ is a word representing }r\},\\
m_c(r)&:=\min \{m(c:w)\text{ }|\text{ } w\text{ is a word representing }r\},\\
m_d(r)&:=\min \{m(d:w)\text{ }|\text{ } w\text{ is a word representing }r\},\\
m_f(r)&:=\min \{m(f:w)\text{ }|\text{ } w\text{ is a word representing }r\}.
\end{align*}
\end{definition}
The first three quantities $m_b(\cdot), m_x(\cdot)$ and $m_y(\cdot)$ are well defined since they measure properties of words that are invariant under $\tau$-transitions. Some examples are in order:
$$m_{x}(x_{-10}^{\text{ }3} x_7^{\text{ }2})=5, \quad m_{y}(x_0abcy_1^{\text{ }3})=3, \quad m_a(a^{6}bx_n)=0, \quad m_d(d^{\text{ }9}bx_nd^{\text{ }3})=3.$$

\begin{lemma}\label{lem:uniformly-bounded}
Any nonempty constructible ideal of $R$ can be written as $\bigcup\limits_{i\in I} r_i R$ such that the quantities in Definition~\ref{def: length-of-R-elements} associated to the $r_i$'s are uniformly bounded, i.e. the set
$$\{m_b(r_i),\ m_x(r_i),\ m_y(r_i),\ m_a(r_i),\ m_c(r_i),\ m_d(r_i),\ m_f(r_i)\mid\ i\in I\}$$
is bounded.
\end{lemma}
\begin{proof}
The result is trivial for principal ideals. Turning to the remaining ideals, for $r\in R$ we have

\begin{equation}\label{eq:p_i1}
r\big(\bigcup\limits_n x_nR\cup \bigcup\limits_n y_nR \big)=\bigcup\limits_n rx_nR\cup \bigcup\limits_n ry_nR
\end{equation}

\begin{equation}
r(R\setminus \{e\})=raR\cup rbR\cup rcR\cup rdR \cup rfR\cup \big(\bigcup\limits_n rx_nR\cup \bigcup\limits_n ry_nR \big)
\end{equation}

\begin{align}
r\big(\bigcup\limits_{n, w'\neq e}x_nw' R\big)=& \bigcup\limits_n rx_naR\cup \bigcup\limits_n rx_nbR\cup \bigcup\limits_n rx_ncR \nonumber \\
& \cup \bigcup\limits_n rx_ndR \cup \bigcup\limits_n rx_nfR\cup \big(\bigcup\limits_{n, m} rx_nx_mR\cup \bigcup\limits_{n, m} rx_ny_mR \big)
\end{align}

\begin{align}\label{eq:p_i4}
r\big(\bigcup\limits_{n, w'\neq e}y_nw' R\big)
=& \bigcup\limits_n ry_naR\cup \bigcup\limits_n ry_nbR\cup \bigcup\limits_n ry_ncR \nonumber \\
& \cup \bigcup\limits_n ry_ndR \cup \bigcup\limits_n ry_nfR\cup \big(\bigcup\limits_{n, m} ry_nx_mR\cup \bigcup\limits_{n, m} ry_ny_mR \big).
\end{align}

It is clear that as $\boldsymbol{\cdot}$ runs over all the generators of the principal ideals in the unions in~\eqref{eq:p_i1} to~\eqref{eq:p_i4}, the quantities $m_b(\boldsymbol{\cdot})$, $m_x(\boldsymbol{\cdot})$ and $m_y(\boldsymbol{\cdot})$ are no greater than $m_b(r)+1$, $m_x(r)+2$ and $m_y(r)+2$, respectively. For any word $w$ we have $m_a(rw)\leq m_a(r)+m(a: w)$, and similarly for $m_c(rw)$, $m_d(rw)$ and $m_f(rw)$ so as~$\boldsymbol{\cdot}$ runs over all the generators of the principal ideals in the unions above, $m_a(\boldsymbol{\cdot})$, $m_c(\boldsymbol{\cdot})$, $m_d(\boldsymbol{\cdot})$ and $m_f(\boldsymbol{\cdot})$ are bounded by $m_d(r)+1$.
\end{proof}

We are now ready to tackle Proposition~\ref{prop:boundary-characters-of-R}.\\

\begin{proof}[Proof of proposition ~\ref{prop:boundary-characters-of-R}]
We prove the proposition in four steps:\\

\underline{Step 1:} $\{\chi_w\mid\ w\in \RR^\infty_{\text{red}}\text{ is of type }1\}\subseteq  \Omega_{\text{max}}$.\\

Indeed, let $w\equiv w_1w_2\cdots  w_n\cdots \in \RR^\infty_{\text{red}}$ be of type $1$. It suffices to show that if $\chi_w(A)=0$ for some $A\in\mathcal{J}(R)$, then there is a 
$B\in \mathcal{J}(R)$, disjoint from $A$, such that $\chi_w(B)=1$.\\
We argue that it is enough to verify the following claim:\\

\begin{claim}
If $r\in R$ is such that $w_1\cdots w_n\notin rR$ for all $n\geq 1$, then there is an $m\geq 1$, depending only on the quantities $m_b(r)$, $m_x(r)$, $m_y(r)$ and  $m_d(r)$, such that $w_1\cdots w_mR$ is disjoint from $rR$.\\
\end{claim}

Indeed, assume the claim is established. Then if a constructible ideal $\bigcup_{i\in I} r_i R$, as in the statement of Lemma~\ref{lem:uniformly-bounded}, does not  contain $w_1\cdots w_n $ for any $n\geq 1$, there are $m_i\geq 1$, depending only on $m_b(r_i)$, $m_x(r_i)$, $m_y(r_i)$ and $m_d(r_i)$, such that $w_1\ldots w_{m_i}R$ is disjoint from $r_iR$. The quantities $m_b(r_i)$, $m_x(r_i)$, $m_y(r_i)$ and $m_d(r_i)$ are bounded so the $m_i$'s can be chosen to be bounded as well. Setting $m=\max\limits_{i}m_i$ we see that $w_1\ldots w_mR$ is disjoint from $\bigcup\limits_{i\in I} r_i R$.\\

Let us formulate the contrapositive of the claim: given $r\in R$ there is an $m\geq 1$, which only depends on $m_b(r)$, $m_x(r)$, $m_y(r)$ and $m_d(r)$, such that whenever the equality
\begin{equation}\label{eq:mequality}
w_1w_2 \cdots w_m r_1=rr_2
\end{equation}
holds for some $r_1, r_2\in R$, then $w_1w_2\cdots w_n\in rR$ for some $n\geq 1$. Let us therefore assume Equation~\eqref{eq:mequality} and consider various cases depending on the nature of the reduced infinite word $w$.\\

\underline{Case 1: there are infinitely many strings $bx_n$ or $by_n$, $n\in \mathbb{Z}$, in $w$.}\\

We  may write $w\equiv v_1bz_1v_2bz_2v_3\cdots$ where $z_i\in \{x_n, y_n \mid n\in \mathbb{Z}\}$ and the $v_i$'s are words that do not contain the strings $bx_n$ or $by_n$ for any $n\in \mathbb{Z}$. Let $m$ be such that
$$w_1w_2\cdots w_m\equiv \prod\limits_{i=1}^{m_b(r)+2} v_ibz_i$$
(note that the product on the RHS is ordered according to the index $i$). To avoid cumbersome notation we will assume that the $z_i's$ are $x_n$ elements, say $z_i=x_{n_i}$. Then

$$w_1w_2\cdots w_m\equiv \prod\limits_{i=1}^{m_b(r)+2} v_ibx_{n_i}.$$

For $r_1\neq \epsilon$ we have

$$[w_1w_2\cdots w_mr_1]=[\Big(\prod\limits_{i=1}^{m_b(r)+2} v_ibx_{n_i}\Big) r_1]$$
$$=\{\Big(\prod\limits_{i=1}^{m_b(r)+2} v_i s_ibx_{n_i+m(c: s_i)}\Big) r_1'\mid \text{ }r_i\text{ are words such that }\forall i\text{ }\con(s_i)\subseteq \{a, c,d\}, \text{ }r_1'\sim r_1\}.$$
\text{ }\\

Then the equivalence~\eqref{eq:mequality} means that $r$ is a subword of a word in the above set. Furhermore~$r$ contains $m_b(r)$ many $b$'s so $r$ must be a subword of a word $\prod\limits_{i=1}^{m_b(r)+1} v_is_i'bx_{n_i+m(c:s_i')}$ for some words $s_i'$ such that $\con(s_i')\subseteq \{a, c, d\}$ for all $i$. Then

$$w_1w_2\cdots w_m= \prod\limits_{i=1}^{m_b(r)+2} v_ibx_{n_i}=\Big(\prod\limits_{i=1}^{m_b(r)+1} v_i s_i'bx_{n_i+m(c:s_i')}\Big) v_{m_b(r)+2}bx_{n_{m_b(r)+2}}$$ $$=r\Big(\Big(\prod\limits_{i=1}^{m_b(r)+1} v_i s_i'bx_{n_i+m(c:s_i')}\Big)-r\Big)v_{m_b(r)+2}bx_{n_{m_b(r)+2}}\in rR.$$
\text{ }\\

\underline{Case 2: there are finitely many strings $bx_n$ and $by_n$, $n\in \mathbb{Z}$, in $w$.}\\

\underline{Subcase 1: There are infinitely many $x_n$ or $y_n$, $n\in \mathbb{Z}$, 
in $w$.}\\

Assume without loss of generality that $w$ contains infinitely many $x_n$ elements. Choose an $m\geq 1$ such that $w_{m-1}\neq b$, $w_m=x_k$ for some $k\in \mathbb{Z}$, and $w_1w_2\cdots w_{m-1}w_m$ contains more $x_n$ elements than $r$.  Then $w_1w_2 \cdots w_{m-1}w_m \perp r'$ for any $r'\in R$. Hence~\eqref{eq:mequality} implies that $w_1w_2 \cdots w_{m-1}w_m$ is equivalent to some subword of $rr_2$. Since the former contains more $x_n$ elements than $r$ does, we must have $w_1w_2 \cdots w_{m-1}w_m=rr_3$ for some $r_3\in R$, so that $w_1w_2\cdots w_m\in rR$. Note that here the choice of $m$ only depends on $m_x(r)$.\\

\underline{Subcase 2: There are finitely many $x_n$ and $y_n$, $n\in \mathbb{Z}$, 
in $w$.}\\

Let us first assume that there are infinitely many $b$'s in $w$. Choose an $m\geq 1$ such  that $w_{m-1}=b$, $w_m\neq x_n, y_n$ and $w_1w_2\cdots w_{m-1}$ contains more $b$'s than $r$ does. Then $w_1w_2\cdots w_{m-2}b\perp w_mr'$ for any $r'\in R$. Hence~\eqref{eq:mequality} implies that $w_1w_2\cdots w_{m-2}b$ is equivalent to a subword of $rr_2$. Since the first word contains more $b$'s than $s$ does, we must have $w_1w_2\cdots w_{m-2}b=w_1w_2\cdots w_{m-1}=rr_3$ for some $r_3\in R$, so that $w_1w_2\cdots w_{m-1}\in rR$. Here the choice of $m$ only depends on~$m_b(r)$.\\

Now let us assume that there are finitely many $b$'s in $w$. Then, since it is type $1$, $w$ contains infinitely many $d$'s  and $f$'s and there is a $k_1\geq 1$ such that $w_i\in \{a, c, d, f\}$ for $i\geq k_1$. Choose $k_2>k_1$ such that $w_{k_1}w_{k_1+1}\cdots w_{k_2}$ contains $m_d(r)+1$ many $d$'s and $w_{k_2}=d$. Choose $k_3>k_2$ so that $w_{k_3}$ is the next letter after $w_{k_2}$ which is an $f$. Then choose $m$ so that $w_m$ is the next letter after $w_{k_3}$ which is a $d$. Then Lemma~\ref{lem:dbx_n-separation} implies that
$$w_1\cdots w_{k_3-1}f\perp w_{k_3+1}\cdots w_{m-1}dr' \quad \text{ for all }r'\in R.$$
Hence~\eqref{eq:mequality} implies that $w_1\cdots w_{k_3-1}f$ is equivalent to some subword of $rr_2$. We claim that this subword is actually of the form $rr_3$ for some $r_3\in R$, so that $w_1\cdots w_{k_3-1}f=w_1\cdots w_{k_3}=rr_3\in r R$. To ascertain this we have to rule out the possibility that $w_1\cdots w_{k_3-1}f$ is equivalent to a subword of $r$. We do this by proving that for any $u\in R$ we have $m_d(w_1\cdots w_{k_3-1}fu)\geq m_d(r)+1$. This will imply that an equality $r=w_1\cdots w_{k_3-1}fu$ is impossible.\\

For $u\in R$ we have
$$w_1\cdots w_{k_3-1}fu\equiv w_1\cdots w_{k_2-1}d w_{k_2+1}\cdots w_{k_3-1}fu,$$
and since $w_{k_2+1}\cdots w_{k_3-1}fu$ is \textit{not} of the form $sbx_nu'$ where $n\in \mathbb{Z}$, $\con(s)\subseteq \{a, c, d\}$ and $u'\neq \epsilon$, Lemma~\ref{lem:dbx_n-separation} implies that
$$[w_1\cdots w_{k_3-1}fu]=[w_1\cdots w_{k_2-1}d w_{k_2+1}\cdots w_{k_3-1}fu]$$ $$=\{vv'\mid v\sim w_1\cdots w_{k_2-1}d, \text{ }v'\sim w_{k_2+1}\cdots w_{k_3-1}fu\}.$$
Now, $w_1\cdots w_{k_2-1}d\equiv w_1\cdots w_{k_1-1}w_{k_1}\cdots w_{k_2-1}d$, where $w_i\in \{a, c, d ,f\}$ for $k_1\leq i~\leq~k_2$, so Lemma~\ref{lem:vrequivalent} implies that the above set is contained in
$$\{v''w_{k_1}\cdots w_{k_2-1}dv'\mid \text{ }v''\text{ is a word},\text{ }v'\sim w_{k_2+1}\cdots w_{k_3-1}fu\}.$$
This shows that any word equivalent to $w_1\cdots w_{k_3-1}f u$ contains the string $w_{k_1}\cdots w_{k_2-1}d$ which in particular contains $m_d(r)+1$ many $d$'s. Hence $m_d(w_1w_2\cdots w_{k_3-1}fu)\geq m_d(r)+1$.\\

\underline{Step 2:}  
\begin{align*}
\partial \Omega\subseteq\{\langle w\bigcup\limits_{n, w'\neq e}x_nw'R\rangle\mid w\in R\}
\sqcup \{\langle w\bigcup\limits_{n, w'\neq e}y_nw'R\rangle \mid w\in R\}
\sqcup\{\chi_w \mid w\in \RR^\infty_{\text{red}}\}.
\end{align*}

Denote
$$X=\bigcup\limits_{n, w'\neq e}x_nw'R, \quad Y=\bigcup\limits_{n, w'\neq e}y_nw'R, \quad Z=\bigcup\limits_n x_nR\cup \bigcup\limits_n y_nR.$$
For characters $\chi_1, \chi_2$ we write $\chi_1\leq\chi_2$ (resp. $<$) whenever we have the (resp. proper) inclusion $\chi_1^{-1}(1)\subseteq \chi_2^{-1}(1)$ of filters. \\

Fix $\chi\in \partial\Omega$. We give an inductive procedure which either produces a reduced infinite word $w$ such that $\chi=\chi_w$, or terminates at some finite step and produces an element $w\in R$ such that $\chi=\langle wX\rangle$ or $\langle wY\rangle$.\\

Since $\chi$ is non-zero we have $\chi(R)=1$. Inductively, suppose that $\chi(wR)=1$ for some finite word $w\equiv w_1w_2\cdots w_m$. 
Since $waR, wbR, wcR, wdR, wfR, wX$ and $wY$ form a foundation set for $wR$, $
\chi$ recognizes at least one of these ideals. Consider two cases.\\

Suppose that $\chi(wbZ)=0$. If $\chi$ recognizes a $wxR$ ideal for some \mbox{$x \in \{a,b,c,d,f\}$}, put $w_{m+1}=x$ so that $\chi(w_1 \cdots w_m w_{m+1}R)=1$. If instead $\chi(wX)=1$, consider the two possibilities: either there is a $k\in\mathbb{Z}$ such that $\chi$ recognizes the constructible ideal $wx_k(R\setminus \{e\})$ or there is not. In the former case $\chi$ must also recognize $wx_kR$ and we put $w_{m+1}=x_k$. In the latter case we claim that $\chi=\langle wX \rangle$. Indeed, to establish this it is enough to ascertain that any constructible ideal which is properly contained in $wX$, is contained in some $wx_n(R\setminus\{e\})$ ideal, $n\in \mathbb{Z}$, but this follows from Lemma~\ref{lem:multiple x_n elements}. The case where $\chi(wY)=1$ is handled analogously.\\

Now suppose that $\chi(wbZ)=1$. Since the ideals $wbX$ and $wbY$ form a foundation set for~$wbZ$, $\chi$ must recognize one of them. Similar to the previous paragraph, either $\chi=\langle wbX \rangle$, resp. $\chi=\langle wbY \rangle$, or $\chi$ recognizes a $wbx_kR$, resp. $wby_kR$, ideal, in which case we put $w_{m+1}w_{m+2}\equiv bx_k$, resp. $w_{m+1}w_{m+2}\equiv by_k$.\\

If the above process does not terminate at any step it produces an infinite word $w\equiv w_1w_2\cdots$ such that $\chi(w_1\cdots w_nR)=1$ for all $n\geq 1$. We claim that $w$ is a reduced and that $\chi=\chi_w$.\\

To see that $w$ is reduced, suppose it contains, say, a string of the form $xbx_n$ where $x=a, c, d$. Then we can write $w\equiv w'xbx_n\cdots$. We have $\chi(w'xbx_nR)=1$. Since $w'xbx_nR\subseteq w'bZ$ we also have $\chi(w'bZ)=1$. By how we constructed $w$ we see that the immediate letters in $w$ after $w'$ should have been $bx_k$ or $by_k$ for some $k$, not $x$. Hence $w$ does not contain strings of the form $xbx_n$ where $x=a, c, d$. Similarly $w$ does not contain strings $yby_n$ where $y=a, c, f$, so $w$ is reduced.\\

By construction, $\chi(w_1\cdots w_nR)=1$ for all $n\geq 1$, in other words, $\chi_w\leq \chi$. When $w$ is type $1$ then $\chi_w$ is maximal by step 1, so $\chi=\chi_w$ follows.\\

Now suppose that $w$ is of type $2$. Then eventually all $w_i$'s are contained in one of the sets $\{a, c ,d\}$ or $\{a, c ,f\}$. Assume without loss of generality that $w_i\in \{a, c ,d\}$ for $i\geq m$. Suppose that $\chi_w(A)=0$ for some $A\in \mathcal{J}(R)$. Then $A$ does not contain $w_1w_2\cdots w_n$ for any $n$ and Lemma~\ref{lem:infinitewordtype2help} gives a $k\geq m$ such that
$$A\cap w_1w_2\cdots w_k R\subseteq w_1w_2\cdots w_{k-1} bZ.$$
By construction of $w$ we have $\chi(w_1w_2\cdots w_{k-1}bZ)=0$. Therefore $\chi(A\cap w_1w_2\cdots w_k R)=0$. Since $\chi(w_1w_2\cdots w_k R)=1$ this implies that $\chi(A)=0$. Hence $\chi=\chi_w$.\\

\underline{Step 3:} 
\begin{align*}
\{\langle vX\rangle \mid v\in R\}\sqcup \{\langle vY\rangle \mid v\in R\}
\sqcup\{\chi_w \mid w\in \RR^\infty_{\text{red}}\text{ is of type }2\} \subseteq \partial \Omega\setminus \Omega_{\text{max}}.
\end{align*}

The characters $\langle vX\rangle$ and $\langle vY\rangle$ are not maximal for $v\in R$ since $\langle vX\rangle< \langle vx_kw'R\rangle$ and $\langle vY\rangle<\langle vy_kw'R\rangle$  for any $k\in \Z$ and $w'\neq \epsilon$. To see that $\langle vX\rangle\in \partial \Omega$, assume $A, A_1, \dots , A_m\in \J(R)$, $A\supseteq vX$ and $A_1, \dots , A_m\subsetneq vX$. Then it suffices to show that the $A_i$'s do not form a foundation set for $A$.  But this is clear since, as already noted, Lemma~\ref{lem:multiple x_n elements} implies that each $A_i$ is contained in some $vx_{n_i}(R\setminus \{e\})$ ideal. Similarly $\langle vY \rangle \in \partial\Omega$.\\

Now fix a reduced infinite word $w$ of type 2. Write $w\equiv w'w''$ where $w'$ is finite and $\con(w'')\subseteq \{a, c, d\}$ or $\con(w'')\subseteq \{a, c, f\}$. Without loss of generality, assume that  $\con(w'')\subseteq \{a, c, d\}$.\\

Then $\chi_w$ is not maximal since $\chi_w<\langle w'bX\rangle$. To show that $\chi_w\in \partial \Omega$, fix $A, A_1, \dots , A_m\in \J(R)$ with $\chi_w(A)=1$ and $\chi_w(A_i)=0$ for $i=1, \dots, m$. We will find a nonempty constructible ideal contained in $A$ which does not intersect any of the $A_i$'s. For $p\in R$ define $$m_{a+c+d}(p):=m_a(p)+m_c(p)+m_d(p).$$ By Lemma~\ref{lem:uniformly-bounded} we can write $A_i=\bigcup_j p_j^iR$ such that $m_{a+c+d}(p_j^i)$ is uniformly bounded in $i$ and $j$. Pick $n\geq 1$ so that $w_1w_2\cdots w_n \in A$, $w_n\in \{a, c, d\}$ \textit{and} $m_{a+c+d}(w_1w_2\cdots w_n)$ exceeds $m_{a+c+d}(p_j^i)$ for all $i$ and $j$. Consider $B:=w_1w_2\cdots w_nx_0 R\subseteq A$.\\

We claim that $B$ does not intersect an $A_i$ ideal. Indeed, suppose for contradiction that $B$ intersects~$A_{i'}$. Then $w_1w_2\cdots w_nx_0q_1=p_{j'}^{i'}q_2$ for some $q_1, 
q_2\in R$ and $j'$. We have $w_1w_2\cdots w_n x_0  \perp q_1$ so $w_1w_2\cdots w_nx_0$ is either a subword of $p_{j'}^{i'}$ or we have 
\begin{equation}\label{eq:w_1w_n}
w_1w_2\cdots w_nx_0=p_{j'}^{i'}q
\end{equation}
for some $q\in R\setminus \{e\}$. However, the former is impossible since
$m_{a+c+d}(w_1w_2\cdots w_nx_0u)$ exceeds $m_{a+c+d}(p_{j'}^{i'})$ for any $u\in R$. Therefore Equation \eqref{eq:w_1w_n} holds. Since $w_1\cdots w_n \perp x_0$,~\eqref{eq:w_1w_n} implies that $p_{j'}^{i'}q\equiv w'x_0$ for some $w'\sim w_1w_2\cdots w_n$. Since $q\neq e$ this implies that $w_1w_2\cdots w_n=p_{j'}^{i'}q'$ for some $q'\in R$. Then $w_1w_2\cdots w_n\in p_{j'}^{i'}R$ which contradicts that $\chi_w(A_{i'})=0$.\\

Step 1-3 establishes both equalities in the proposition.\\

\underline{Step 4:} The sets $\{\langle w\bigcup\limits_{n, w'\neq e}x_nw'R\rangle \mid w\in R\}$, $\{\langle w\bigcup\limits_{n, w'\neq e}y_nw'R\rangle \mid w\in R\}
$\\ and $\{\chi_w \mid w\in \RR^\infty_{\text{red}}\text{ is of type }2\}$ are invariant subsets of $\partial \G(R)$.\\

To see that $\{\langle vX\rangle \mid v\in R\}$ is invariant it is by Lemma~\ref{lem:range} enough to see that for any $v\in R$ and $h\in I_\ell(R)$ such that $\dom h\supseteq vX$, the ideal
$$h(vX)$$
is again of the form $v' X$ for some $v'\in R$. But this will follow from our computations in classifying the constructible ideals of $R$, which is done in the Appendix. Similarly $\{\langle vY\rangle \mid v\in R\}$ is invariant.\\

Now, the set $\partial\Omega\setminus \Omega_{\text{max}}$ is clearly invariant. It follows that the set $\{\chi_w\mid w\in \RR^\infty_{\text{red}}\text{ is of type }2\}$ is invariant, being the relative complement of the two invariant sets $\{\langle vX\rangle \mid v\in R\}$ and $\{\langle vY\rangle \mid v\in R\}
$ in $\partial\Omega\setminus \Omega_{\text{max}}$.
\end{proof}

Let us recall the definition of Borel amenability for topological groupoids.

\begin{definition}[{\cite[Definition~2.1]{R-2}}]\label{def:Borelamenable}
A topological groupoid $\G$ is said to be Borel amenable if there exists a Borel approximate invariant mean, i.e. a sequence $\{\mu_n\}_{n=1}^\infty$, where each $\mu_n$ is a family of probability measures $\{\mu_n^x\}_{x\in \G^{(0)}}$ where $\mu_n^x$ is supported on $\G_x$, such that
\begin{enumerate}[(i)]
\item For all $n\in \mathbb{N}$, $\mu_n$ is Borel in the sense that for all bounded Borel functions $f$ on $\G$, the map
\begin{equation}\label{eq:measurability}
x \mapsto \int f \ d\mu_n^x
\end{equation}
is Borel;\\
\item 
\begin{equation}\label{eq:preapproximateinvariance}
\|\mu_n^{s(g)}-\mu_n^{r(g)}(\cdot \ g^{-1})\|_1 \rightarrow 0
\end{equation}
for all $g\in \G$.
\end{enumerate}
\end{definition}
We are ready to prove the Borel amenability of $\partial\G(R)$.
\newpage

\begin{theorem}
The groupoid $\partial \G_P(R)=\partial \G(R)$ is Borel amenable.
\end{theorem}

\bp
Put $\G:=\partial \G_P(R)=\partial \G(R)$. We will prove that $\G$ satisfies Definition~\ref{def:Borelamenable} above. Specifically we define a Borel approximate invariant mean $\{\mu_n\}_n$ where each $\mu_n^x$ is finitely supported. Since $\G$ is étale, showing~\eqref{eq:preapproximateinvariance} translates to showing that for any $y\in \partial \Omega$ and $g\in \G_y^x$ we have
\begin{equation}\label{eq:approximateinvariance}
\|\mu_n^y-\mu_n^x(\cdot g^{-1})\|_1=\sum\limits_{g'\in \mathcal{G}_y} |\mu_n^y(g')-\mu_n^x(g'g^{-1})|\xrightarrow[n \rightarrow \infty]{} 0.
\end{equation}
In Proposition~\ref{prop:boundary-characters-of-R} we showed that
\begin{equation}\label{eq:the3sets}
\partial\Omega= \{\langle w\bigcup\limits_{n, w'\neq e}x_nw'R \rangle \mid w\in R\}\sqcup \{\langle w\bigcup\limits_{n, w'\neq e}y_nw'R \rangle \mid w\in R\} \sqcup \{\chi_w \mid \text{ }w\in \RR_{\text{red}}^\infty\}
\end{equation}
and that the three sets in this union are invariant. Accordingly, we will define $\mu_n ^x$ in three separate cases.\\

Suppose $x=\chi_w$ for some $w\in \RR_{\text{red}}^\infty$. Then we define, for $n \geq 1$,
\begin{equation}\label{eq:munx1}
\mu_n^x=\sum\limits_{m=1}^n \frac{1}{n}\delta_{\text{\footnotesize $[w_m^{-1}\cdots w_1^{-1}, x]$}}.
\end{equation}

Suppose $x=\langle w \bigcup\limits_{n, w'\neq e}x_nw'R\rangle$ for some $w\in R$. Fixing \textit{one} such $w$,\footnote{\label{note1}Actually, there is a canonical choice for $w$ and using these while modifying~\eqref{eq:munx1}-\eqref{eq:munx3} slightly, guarantees that~\eqref{eq:measurability} is continuous for continuous $\phi$. However, since we only require that the map~\eqref{eq:measurability} is Borel we need not bother with this.} we define for $n \geq 1$,
\begin{equation}\label{eq:munx2}
\mu_n^x=\sum\limits_{m=1}^n \frac{1}{n}\delta_{\text{\footnotesize $[c^{-m}bw^{-1}, x]$}}.
\end{equation}

Suppose $x=\langle w \bigcup\limits_{n, w'\neq e}y_nw'R\rangle$ for some $w\in R$. Fixing \textit{one} such $w$,\footref{note1} we define for $n\geq 1$,
\begin{equation}\label{eq:munx3}
\mu_n^x=\sum\limits_{m=1}^n \frac{1}{n}\delta_{\text{\footnotesize $[a^{-m}bw^{-1}, x]$}}.
\end{equation}

Note that $\mu_n^x$ in Equation~\eqref{eq:munx1} is well defined since whenever $x=\chi_w$ for some $w\in \RR_{\text{red}}^\infty$, then by Proposition~\ref{prop:distinctinfinitereducedwords} there is only one such $w$. We stress that in~\eqref{eq:munx2} and~\eqref{eq:munx3}, for a fixed $x$, the same $w$ is used to define $\mu_n^x$ for \textit{all} $n\geq 1$.\\

We argue that the Dirac measures in the sums~\eqref{eq:munx1}-\eqref{eq:munx3} are distinct. In the~\eqref{eq:munx2} and~\eqref{eq:munx3} case this is clear. For~\eqref{eq:munx1}, let $x=\chi_w$, $w\in \RR_\text{red}^\infty$, and suppose for contradiction that there are $m_1<m_2$ such that $[w_{m_1}^{-1}w_{m_1-1}^{-1}\cdots w_1^{-1}, x]=[w_{m_2}^{-1}w_{m_2-1}^{-1}\cdots w_1^{-1}, x]$. Then the partial bijections $w_{m_1}^{-1}w_{m_1-1}^{-1}\cdots w_1^{-1}$ and $w_{m_2}^{-1}w_{m_2-1}^{-1}\cdots w_1^{-1}$ agree on $w_1\cdots w_k$ for sufficienty large $k\geq m_2$. In other words,
\begin{equation}\label{eq:m_1m_2equality}
w_{m_1+1}w_{m_1+2}\cdots w_k=w_{m_2+1}w_{m_2+2}\cdots w_k.
\end{equation}
But this contradicts Proposition~\ref{prop:distinctfinitereducedwords}, therefore the elements $[w_m^{-1}w_{m-1}^{-1}\cdots w_1^{-1}, x]$, $m\geq 1$, are distinct.\\

Let us establish~\eqref{eq:approximateinvariance}. Since the three sets in~\eqref{eq:the3sets} are invariant we need to consider the cases where $x$ and $y$ both lie in the first set, second set or third set. We begin with the first.\\

Fix $x, y\in \partial \Omega$ and $g\in\G_y^x$, where $x=\langle v\bigcup\limits_{n, w'\neq e}x_nw'R \rangle$, $y=\langle w\bigcup\limits_{n, w'\neq e}x_nw'R \rangle$ and $v,w\in R$ are the monoid elements we used to define the measures $\mu_n^x$ and $\mu_n^y$, $n\geq 1$. Write $g=[h,y]$.  For $n\geq 1$, consider the subset
$$A_n:=\{m\in \{1,\ldots ,n\} \mid \text{for the element }[h', y]:=[c^{-m}bw^{-1}, y]$$ $$\text{ we have }[h'h^{-1}, x]=[c^{-m'}bv^{-1}, x] \text{ for some }m'\in \{1, \ldots ,n\}\}.$$

Since the map $\G_y \rightarrow \G_x$ given by right multiplication by $g^{-1}$ is a bijection, to establish~\eqref{eq:approximateinvariance}, it is enough to establish that $\lim\limits_{n \rightarrow \infty}\frac{|A_n|}{n}= 1$.  By Lemma \ref{lem:range} we have
\begin{equation}\label{eq:range2}
hw \Big( \bigcup\limits_{n, w'\neq e}x_nw'R \Big)=v\Big( \bigcup\limits_{n, w'\neq e}x_nw'R\Big).
\end{equation}
\begin{claim}\label{claim:v^-1hw}
There is a $k\in \Z$ such that $v^{-1}hw$ maps $x_nw'$ to $x_{n+k}w'$ for all $n\in \Z$ and $w'\neq e$.\\
\end{claim}

By Corollary~\ref{lem:shifts x_n or x_nx} it is enough to establish that for some letter $z$ we have $(v^{-1}hw)(x_0z)=x_kz$. Now,~\eqref{eq:range2} implies that for any letter $z$ we have,
\begin{equation}\label{eq:first}
hwx_0z=vx_kw'
\end{equation}
for some $k\in \mathbb{Z}$ and $w'\neq \epsilon$. Write $w'\equiv \alpha(w')w''$. Applying~\eqref{eq:range2}  yields
$$hw x_0z=vx_k\alpha(w')w''=hw x_m w''' w''$$
for some $m\in \mathbb{Z}$ and $w'''\neq \epsilon$. This equation implies that $z=w'''w''$ and since $w'''\neq \epsilon$ we must have $w''=\epsilon$. Hence $w'$ is just a letter. Equation~\eqref{eq:first} then implies that $w'=z$. Rearranging~\eqref{eq:first} we have
$$(v^{-1}hw)(x_0z)=x_kz,$$
which proves the claim.\\

Now let $[h', y]=[c^{-m}bw^{-1}, y]$. We assert that 
\begin{equation}\label{eq:requiredresult}
[h'h^{-1}, x]=[c^{-m-k}bv^{-1}, x].
\end{equation}
By assumption we have $[h'h^{-1}, x]=[c^{-m}bw^{-1}h^{-1}, x]=[c^{-m}b(hw)^{-1}, x]$. Then for any $n\in\mathbb{Z}$ and $w'\neq \epsilon$ we have
$$c^{-m}b(hw)^{-1}(hwx_nw')=c^{-m}bx_nw'=bx_{n-m}w'.$$
On the other hand, the claim gives
$$c^{-m-k}bv^{-1}(hwx_nw')=c^{-m-k}b(v^{-1}hw)x_nw'=c^{-m-k}bx_{n+k}w'=bx_{n-m}w'.$$
This establishes that the partial bijections $c^{-m}b(hw)^{-1}$ and $c^{-m-k}bv^{-1}$ agree on Constructible Ideal~\eqref{eq:range2}, which proves Equality~\eqref{eq:requiredresult}. We have then shown that 
$\{1, \ldots ,n-k\}\subseteq A_n$, hence $\frac{|A_n|}{n} \rightarrow 1$.\\

Now let us establish~\eqref{eq:approximateinvariance} for $x=\chi_v$ and $y=\chi_w$ where $v,w \in \RR_{\text{red}}^\infty$. Suppose $[h, y]\in \mathcal{G}_y^x$.
We can write $h=r_{2m}^{-1}r_{2m-1}\cdots r_2^{-1}r_1$, $r_i\in R$. Put $x_0=y$ and 
$$x_i=r([r_i^{(-1)^{i+1}}\cdots r_2^{-1}r_1, x]), \quad i=1, \ldots, 2m.$$ 
Then $x_{2m}=x$ and we have the telescoping sum
$$\mu_n^y-\mu_n^x(\cdot [h, y]^{-1})=\sum\limits_{i=1}^{2m} 
\mu_n^{x_{i-1}}-\mu_n^{x_i}(\cdot [r_i^{(-1)^{i+1}}, x_i]^{-1}).$$
This shows that, in verifying~\eqref{eq:approximateinvariance}, we may assume that $g=[h, y]$, where $h$ is a generator or the inverse of a generator. In fact, by symmetry, we may assume that $h$ is a generator. Then
\begin{equation}\label{eq:xishw}
	\begin{split}
x=r([h,y])=r([h, \lim\limits_n \chi_{\text{\footnotesize $w_1\cdots w_n$}}])=r(\lim\limits_n[h,\chi_{\text{\footnotesize $w_1\cdots w_n$}}])\\
=\lim\limits_n r([h,\chi_{\text{\footnotesize $w_1\cdots w_n$}}])=\lim\limits_n \chi_{\text{\footnotesize $hw_1\cdots w_n$}}=\chi_{\text{\footnotesize $hw$}}.
	\end{split}
\end{equation}
For $n\geq 1$, consider the subset
\begin{align*}
A_n:=\{m\in \{1,\ldots ,n\} \mid \ & hw_1\cdots w_m=v_1\cdots v_{m'}\text{ (as elements of }R\text{) } \\
& \text{for some }m'\in \{1, \ldots ,n\}\}.
\end{align*}
To establish~\eqref{eq:approximateinvariance} it is enough to prove that $\lim\limits_{n \rightarrow \infty}\frac{|A_n|}{n}= 1$. Indeed, if $[h',y]=[w_m^{-1}\cdots w_1^{-1}, y]$ with $m\in A_n$, then taking inverses we find
$$[{h'}^{-1}, h'y]=[w_1\ldots w_m, h'y].$$
Multiply both sides on the left by the groupoid element $[h, y]$. Then
$$[h,y][{h'}^{-1}, h'y]=[h,y][w_1\ldots w_m, h'y]$$
$$\Rightarrow [h{h'}^{-1}, h'y]=[hw_1\ldots w_m, h'y]=[v_1 \ldots v_{m'}, h'y]$$
for some $m'\in \{1,\ldots, n\}$. Here we use that $m\in A_n$. Taking inverses yields
$$[h'h^{-1}, x]=[{v_{m'}}^{-1}\cdots {v_1}^{-1}, x].$$
Now, if $w_1\in \{a, c,d, f, x_l, y_l\mid l\in \mathbb{Z}\}$ then $hw\equiv hw_1w_2 \cdots$ is 
a reduced word for any letter $h$. If $w_1=b$ and $w_1\notin \{x_n, y_n \mid n\in \Z\}$ then $hw\equiv hw_1w_2 \cdots$ is also reduced for any letter $h$. Otherwise we have $w\equiv bx_l w_3 \cdots$ or $w\equiv by_l w_3 \cdots$ for some $l\in \Z$. The two final cases are symmetric and we divide our analysis as follows.\\

\underline{Case 1: $hw$ is reduced}\\
By Equation~\eqref{eq:xishw} we have $v\equiv hw$ so for any $m\in \{1, \ldots, n-1\}$ the element $hw_1\cdots w_m$ is the product of the first $m+1$ letters of $v$.\\

\underline{Case 2: $w\equiv bx_l w_3 \cdots$}\\
If $h=a, d$ we assert that $v\equiv w$. Indeed, for $m\geq 3$ we have
\begin{equation}\label{eq:reducedwordofxisw}
hw_1\cdots w_m=(hbx_lw_3)w_4\cdots w_m=bx_lw_3w_4\cdots w_m=w_1\cdots w_m
\end{equation}
This shows that $\chi_{hw}=\chi_w$. By Equation~\eqref{eq:xishw} we then have $\chi_v=\chi_w$ and it follows that $v\equiv w$. Then~\eqref{eq:reducedwordofxisw} also shows that for $m\geq 3$, $hw_1\cdots w_m$ is the product of the first $m$ letters of $v$.\\

If $h=c$ we assert that $v\equiv bx_{l+1}w_3\cdots$. Indeed, for $m\geq 2$ we have
\begin{equation}\label{eq:reducedwordofxisshift}
hw_1\cdots w_m=(cbx_l)w_3\cdots w_n=bx_{l+1}w_3\cdots w_n
\end{equation}
This shows that $\chi_v=\chi_{\text{\footnotesize $hw$}}=\chi_{\text{\footnotesize $bx_{l+1}w_3\cdots$}}$. Since the infinite word $bx_{l+1}w_3\cdots $ is reduced we conclude that $v\equiv bx_{l+1}w_3\cdots $. Then~\eqref{eq:reducedwordofxisshift} also shows that for $m\geq 2$, $hw_1\cdots w_m$ is the product of the first~$m$ letters of $v$.\\

If $h=f, x_n, y_n$, $n\in \Z$,  we are in Case 1.\\

In both cases we have $\frac{|A_n|}{n}\rightarrow 1$ as $n\rightarrow \infty$ and~\eqref{eq:approximateinvariance} is established.\\

To complete the proof we need to show that the map~\eqref{eq:measurability} is Borel. Clearly Borel measurability is preserved under taking bounded pointwise limits of $\phi$. Moreover $\G$ is the countable union of compact open second-countable Hausdorff subspaces so it suffices to verify Borel measurability for Borel functions $\phi$ supported on such subspaces. Since Borel functions on metrizable spaces are generated by continuous functions under taking bounded pointwise limits, it suffices to consider continuous $\phi$. By Stone-Weierstrass it further suffices to consider $\phi=\un_{D(h, U)}$ where $h\in I_\ell(R)$ and $U$ is a clopen subset of $\{\eta\in \partial \Omega \mid \eta(\dom h)=~1\}$. Fixing $n\geq 1$, we put
$$F(x):=\int \un_{D(h, U)}\ d\mu_n^x.$$
We argue that the restriction of $F$ to $A:=\{\chi_w \mid w\in \RR_{\text{red}}^\infty\}$ is continuous. Since $\partial \Omega \setminus A$ is countable, this will prove that $F$ is Borel. For $w\in \RR_{\text{red}}^\infty$ we have
$$
F(\chi_w)=\begin{cases}

1/n & \text{if }h(w_1\cdots w_k)=w_{m+1}w_{m+2}\cdots w_k\text{ for some }1\leq m\leq n\leq k, \\
0 & \text{if }\chi_w(\dom h)=0 \text{ or }[h, \chi_w]\neq [w_m^{-1}\cdots w_1^{-1}, \chi_w] \text{ for all }1\leq m\leq n.
\end{cases}
$$
Lemma~\ref{lem:reducedinfinitewordsneighborhood} yields at once that $F_{|A}$ is continuous at $\chi_w$ if $F(\chi_w)=1/n$. Clearly $F_{|A}$ is also continuous at $\chi_w$ if $\chi_w(\dom h)=0$ since this is an open condition.\\

It seems obvious that $[h, \chi_w]\neq [w_m^{-1}\cdots w_1^{-1}, \chi_w]$ is also an open condition, however, since $\G$ is non-Hausdorff it is not so simple. To treat this case, fix $w\in \RR_{\text{red}}^\infty$ such that $\chi_w(\dom h)=~1$ and $F(\chi_w)=~0$. Pick $N\geq n$ such that $w_1\cdots w_N\in \dom h$. Then fix $v\in\RR_{\text{red}}^\infty$ such that $v_1\cdots v_{N+4}\equiv~w_1 \cdots w_{N+4}$. We claim that $F(\chi_v)=0$; by Lemma~\ref{lem:reducedinfinitewordsneighborhood}, this will prove that $F_{|A}$ is continuous at~$\chi_w$. Suppose to the contrary that $[h, \chi_v]=[v_{m_0}^{-1}\cdots v_1^{-1}, \chi_v]$ for some $1\leq m_0\leq n$ and pick $k\geq N+5$ such that
\begin{equation}\label{eq:v_1v_kinitial}
h(v_1\cdots v_k)=v_{m_0+1}\cdots v_k.
\end{equation}
We assert that
\begin{equation}\label{eq:v_{N+4}}
h(v_1 \cdots v_{N+4})=v_{m_0+1}\cdots v_{N+4}.
\end{equation}
This will imply that $h(w_1\cdots w_{N+4})=w_{m_0+1}\cdots w_{N+4}$, contradicting that $F(\chi_w)=0$.\\

To establish the assertion, we rewrite \eqref{eq:v_1v_kinitial}:
\begin{equation}\label{eq:v_1v_k}
h(v_1 \cdots v_N)v_{N+1}\cdots v_k=v_{m_0+1}\cdots v_k.
\end{equation}

If $v_{N+1}\in \{a, c, d, f\}$, or $v_{N+1}=b$ and $v_{N+2}\not\equiv x_l, y_l$ for any $l\in \Z$, we have $h(v_1\cdots v_N)v_{N+1} \perp v_{N+2} \cdots v_k$. Equation \eqref{eq:v_1v_k} then gives a $k'$, $m_0+1\leq k' \leq k$, such that  $h(v_1 \cdots v_N)v_{N+1}=v_{m_0+1}\cdots v_{k'}$ and $v_{N+2}\cdots v_k=v_{k'+1}\cdots v_k$. By Proposition~\ref{prop:distinctfinitereducedwords} we must have $k'=N+1$. Then assertion \eqref{eq:v_{N+4}} holds.\\

If $v_{N+1}=x_{l_1}, y_{l_1}$ for some $l_1\in \Z$, then either $v_{N+2}v_{N+3}\not\equiv bx_l, by_l$ for any $l\in \Z$ or, $v_{N+2}v_{N+3}\equiv bx_{l_2}, by_{l_2}$ for some $l_2\in \Z$. In the former case one works out that $h(v_1 \cdots v_N)v_{N+1}v_{N+2} \perp v_{N+3}\cdots v_k$ (here we are using that $v$ is reduced) and we are done. In the latter case let us assume without loss of generality that $v_{N+1}=x_{l_1}$ and $v_{N+2}v_{N+3}\equiv bx_{l_2}$. Then

$$h(v_1 \cdots v_N)v_{N+1}\cdots v_k=[h(v_1 \cdots v_N)x_{l_1}]bx_{l_2}[v_{N+4}\cdots v_k].$$
Since $h(v_1 \cdots v_N)x_{l_1}$ does not end in $a, c, d$ or $f$ we can apply Lemma~\ref{lem:ubx_nw} to the word above with $u=h(v_1 \cdots v_N)x_{l_1}$. Then Equation \eqref{eq:v_1v_k} tells us that 
$$v_{m_0+1}\cdots v_k\equiv u'sbx_{l_2-m(c:s)}w'$$
for some words $u'$, $s$ and $w'$ such that $u'bx_{l_2}\sim h(v_1 \cdots v_N)x_{l_1}bx_{l_2}$, $\con(s)\subseteq \{a, c, d\}$ and $w' \sim v_{N+4}\cdots v_k$. By Proposition~\ref{prop:distinctfinitereducedwords} we must have $w'\equiv v_{N+4}\cdots v_k$ and therefore
$$u'sbx_{l_2-m(c: s)}\equiv v_{m_0+1}\cdots v_{N+3}.$$
Since the RHS is reduced we must have $s=\epsilon$ so that $u'bx_{l_2}\equiv v_{m_0+1}\cdots v_{N+3}$. The LHS is equivalent to $h(v_1\cdots v_N)x_{l_1}bx_{l_2}$ so in summary $h(v_1\cdots v_N)v_{N+1}v_{N+2}v_{N+3}=v_{m_0+1}\cdots v_{N+3}$ and assertion \eqref{eq:v_{N+4}} holds in this case.\\

The remaining case is when $v_{N+1}v_{N+4}\equiv bx_l, by_l$ for some $l\in \Z$. This is handled similarly. In conclusion, $F_{|A}$ is continuous, hence $F$ is Borel.
\ep

To complete the proof of Theorem \ref{thm:Rproperties} it remains to show that the Borel amenability of $\G:=\partial\G_P(R)=\partial\G(R)$ implies that the universal and reduced norm on $C_c(\G)$ coincide. Since $\G$ is $\sigma$-compact (because $I_\ell(R)$ is countable) this follows from a recent paper by Brix, Gonzales, Hume and Li, \cite[Corollary~6.10]{BGHL}. Let us mention that it also follows from work of Delaroche and Renault. In particular the proofs of \cite[Proposition~6.1.8]{AR} and the implications (\romannumeral 5) $\Rightarrow$ (\romannumeral 1) $\Rightarrow$ (\romannumeral 2) in \cite[Proposition~3.4]{R-3}, also work for non-Hausdorff étale groupoids. Moreover, Borel-amenability implies condition (\romannumeral 5) in \cite[Proposition~3.4]{R-3} since any Borel approximate invariant mean $(m_n^x)_{x\in \G^{(0)}}$, in the sense of \cite[Definition~2.1]{R-2}, gives rise to a sequence of sections $\xi_n(x):=(m_n^x)^{1/2}$ with the right properties.\\

In conclusion the $e$ norm on $C_c(\G)$ must be the reduced norm. Since $R$ is also strongly C*-regular on the boundary, this implies that the bottom maps in diagrams~\eqref{eq:commutativediagram} and ~\eqref{eq:commutativediagram-G_2} are isomorphisms and the proof of Theorem \ref{thm:Rproperties} is complete.\\

\section{Relations Between Regularity Notions}\label{sec:relations}

Consider six properties of left cancellative monoids $S$ we have touched upon in the paper: strong C$^*$-regularity, C$^*$-regularity,  strong C$^*$-regularity on the boundary, C$^*$-regularity on the boundary, the equality $\G_P(S)=\G(S)$ and the equality $\partial \G_P(S)=\partial \G(S)$. In this section we give a complete set of relations between these properties and construct examples of left cancellative monoids of each combination of properties that the relations permit.
We shall use \cite[Lemma~2.15]{NS} and Proposition \ref{prop:equality-of-boundary-groupoids} to characterize when $\G_P(S)=\G(S)$ and $\partial \G_P(S)=\partial \G(S)$.
We know that (1) strong C$^*$-regularity implies C$^*$-regularity, strong C$^*$-regularity on the boundary and $\G_P(S)=\G(S)$, (2) C$^*$-regularity implies C$^*$-regularity on the boundary, (3) strong C$^*$-regularity on the boundary implies C$^*$-regularity on the boundary and $\partial \G_P(S)=\partial \G(S)$, and (4) $\G_P(S)=\G(S)$ implies that $\partial \G_P(S)=\partial \G(S)$. We also have the following:

\begin{proposition}\label{prop:obstruction1}
Suppose a left cancellative monoid $S$ is C$^*$-regular and $\G_P(S)=\G(S)$. Then $S$ is strongly C$^*$-regular.
\end{proposition}

The proof of the above proposition is just a simpler version of the proof of the following:
\begin{proposition}
Suppose a left cancellative monoid $S$ is C$^*$-regular on the boundary and $\partial\G_P(S)=\partial\G(S)$. Then $S$ is strongly C$^*$-regular on the boundary.
\end{proposition}
\bp
Assume $S$ is a left cancellative monoid which is C*-regular on the boundary such that $\partial\G_P(S)=\partial\G(S)$. Suppose $X, X_1 ,\dots , X_m \in \J(S)$, $X_i\subseteq X$, and $h_1, \dots , h_n \in I_\ell(S)$ satisfy
$$X\setminus \bigcup\limits_{i=1}^m X_i \subseteq \bigcup\limits_{k=1}^n \{s \mid h_ks=s\}.$$
To show that $S$ is strongly C$^*$-regular on the boundary we must find a finite collection $\mathcal{A}\subset \J(S)$ such that each member is fixed by some $h_k$ and $\mathcal{A}\cup \{X_1, \dots , X_m\}$ is a foundation set for $X$. Using that $S$ is C$^*$-regular on the boundary we find $Y_1, \dots ,Y_l \in \bar{\J}(S)$ and indices $1\leq k_j \leq n$ ($j=1, \dots ,l$), such that $h_{k_j}p_{Y_j}=p_{Y_j}$ and $\{X_1, \dots, X_m, Y_1, \dots, Y_l\}$ is a foundation set for $X$. For each $1\leq j\leq l$ write
$$Y_j=Y_j^0\setminus \bigcup\limits_{r=1}^{m_j} Y_j^r$$
for some $Y_j^0, Y_j^1, \dots Y_j^{m_j}\in \J(S)$ such that $Y_j^r\subseteq Y_j^0\subseteq X$. By assumption, $\partial\G_P(S)=\partial\G(S)$ and employing Proposition~\ref{prop:equality-of-boundary-groupoids} gives a finite collection $\mathcal{A}_j \subset \J(S)$ such that $h_{k_j}p_A=p_A$ for each $A\in \mathcal{A}_j$ and $\mathcal{A}_j\cup \{Y_j^1, \dots, Y_j^{m_j}\}$ is a foundation set for $Y_j^0$. Put $\mathcal{A}:=\cup_{j=1}^l \mathcal{A}_j$. We claim that $\mathcal{A}\cup \{X_1, \dots , X_m\}$ is a foundation set for $X$.\\

Indeed, assume $B\subseteq X\setminus \cup_{i=1}^m X_i$ is a nonempty constructible ideal. Let $F\subseteq \{1, \ldots, l\}$ be a maximal set for which there exists indices $r_j$, $j\in F$, such that $B\cap \bigcap_{j\in F} Y_j^{r_j}\neq \emptyset$. Since $\{X_1, \dots, X_m, Y_1, \dots, Y_l\}$ is a foundation set for $X$ we must have $F\neq \{1, \ldots ,l\}$. For the same reason $C:=B\cap \bigcap_{j\in F} Y_j^{r_j}$ must intersect some $Y_{j'}$, $j'\notin F$. Then $C$ must intersect some member of $\mathcal{A}_{j'}\cup \{Y_{j'}^1, \dots, Y_{j'}^{m_{j'}}\}$. But $C\cap Y_{j'}^r=\emptyset$ for all $1\leq r \leq m_{j'}$ so $C$ must intersect some member of~$\mathcal{A}_{j'}$.
\ep

It would be satisfying to know that there are no other relations between the six properties than the ones above. Taking all of the obstructions above into account leaves the following $9$ combinations of properties that $S$ may possess; for  each point we shall exhibit a left cancellative monoid with the corresponding properties. 

\begin{enumerate}[(I)]
\item Strongly Regular.
\item Strongly Regular on the Boundary, non-Regular, $\G_P(S)=\G(S)$.
\item Strongly Regular on the Boundary, non-Regular, $\G_P(S)\neq \G(S)$.
\item Non-Strongly Regular, Strongly Regular on the Boundary, Regular. \label{examplecombo}
\item Regular, $\partial \G_P(S)\neq \partial \G(S)$.
\item Non-Regular, Regular on the Boundary, $\partial \G_P(S)\neq \partial \G(S)$.
\item Non-Regular on the Boundary, $\G_P(S)=\G(S)$.
\item Non-Regular on the Boundary, $\G_P(S)\neq \G(S)$, $\partial \G_P(S)=\partial \G(S)$.
\item Non-Regular on the Boundary, $\partial \G_P(S)\neq \partial \G(S)$.\\
\end{enumerate}

Note that the properties in each point completely determine which of the six properties $S$ has. For example, if $S$ satisfies the properties in \ref{examplecombo}, $S$ being strongly regular on the boundary implies that it is regular on the boundary and that $\partial \G_P(S)= \partial \G(S)$. By Proposition~\ref{prop:obstruction1} we have $\G_P(S)\neq \G(S)$ since $\G_P(S)=\G(S)$ and $S$ being regular would imply that  $S$ is strongly regular which it is not.\\

In constructing monoids for each point, the following proposition will be useful.

\begin{proposition}\label{prop:monoid-product}
Let $S$ and $T$ be left cancellative monoids. Then the product $S\times T$ is (strongly) C*-regular (on the boundary) if and only if $S$ and $T$ are (strongly) C*-regular (on the boundary). Moreover, we have $\G_P(S\times T)=\G(S\times T)$ if and only if $\G_P(S)=\G(S)$ and $\G_P(T)=\G(T)$, finally, $\partial\G_P(S\times T)=\partial\G(S\times T)$ if and only if $\partial\G_P(S)=\partial\G(S)$ and $\partial\G_P(T)=\partial\G(T)$.
\end{proposition}
\bp
One checks that $$I_\ell(S\times T)=\{f\times g \mid f\in I_\ell(S),\ g\in I_\ell(T)\}$$
where $(f\times g)(s, t)=fs\times gt$, and
$$\J(S\times T)=\{A\times B \mid A\in \J(S),\ B\in \J(T)\}.$$
Also, the cartesian product of members of $\bar{\J}(S)$  and $\bar{\J}(T)$ lie in $\bar{\J}(S\times T)$. Indeed, if $A, A_1, \dots, A_{m_1} \in \J(S)$ and $B, B_1, \dots, B_{m_2}\in \J(T)$, then
$$\big( A\setminus \bigcup\limits_{i} A_i\big)\times \big( B\setminus 
\bigcup\limits_{j} B_j\big)=(A\times B)\setminus \big( \bigcup\limits_i 
(A_i\times B) \cup \bigcup\limits_j (A\times B_j)\big).$$

Let $p_1: S\times T \rightarrow S$ and $p_2: S\times T \rightarrow T$ denote the projections onto the first and second coordinate. Then $p_1$ maps sets in $\bar{\J}(S\times T)$ to sets in $\bar{\J}(S)$ and $p_2$ maps sets in $\bar{\J}(S\times T)$ to sets in $\bar{\J}(T)$. Indeed, take $Y=(A\times B) \setminus \bigcup_{i=1}^m (A_i\times B_i) \in \bar{\J}(S\times T)$ where $A, A_i \in \J(S)$ and $B, B_i\in \J(T)$. Let $\mathcal{I}$ denote the collection of (minimal) subsets $F$ of $\{1, \dots, m\}$ such that $\bigcup_{i\in F}B_i=T$. Then
$$p_1(Y)=A\setminus \big( \bigcup\limits_{F\in \mathcal{I}} \bigcap\limits_{i\in F} A_i\big)\in \bar{\J}(S).$$

Now let us see why the properties of $S\times T$ imply those of $S$, say. Suppose $X, X_1, \dots ,X_m \in \J(S)$ and $h_1, \dots ,h_n\in I_\ell(S)$ satisfy~\eqref{eq:hypstronggreg}. Then we also have
$$(X\times T)\setminus \bigcup\limits_{i=1}^m (X_i\times T)\subseteq \bigcup\limits_{k=1}^n \{(s,t)\in S \times T \mid(h_k \times p_T)(s,t)=(s,t)\}.$$
In other words, the sets $X\times T, X_i \times T$ and partial bijections $h_k\times p_T$ satisfy~\eqref{eq:hypstronggreg} in place of $X,X_i$ and $h_k$. Then, assuming $S\times T$ is (strongly) C*-regular (on the boundary), let $Y_j$ be the sets in $\bar{\J}(S\times T)$ (resp. $\J(S\times T)$) afforded to us 
by that definition. Put $A_j=p_1(Y_j)\in \bar{\J}(S)$ (resp. $\J(S)$). Each $Y_j$ is fixed by some $h_k\times p_T$ so $A_j$ is fixed by the same $h_k$. If $X\times T\subseteq \bigcup_i (X_i\times T)\cup \bigcup_j Y_j$, then $X\subseteq \bigcup_i X_i\cup \bigcup_j A_j$. Also, if the $X_i\times T$ and $Y_j$ sets form a foundation set for $X\times T$, then clearly the $X_i$ and $A_j$ sets form a foundation set for $X$.\\

To complete the proof we show that the properties of $S$ and $T$ imply those of $S\times T$. For any $A, A_1,\dots, A_m\in \J(S)$ and $B, B_1, \dots, B_m\in \J(T)$ we have
$$(A\times B)\setminus \bigcup\limits_{i=1}^m (A_i\times B_i)=\bigcup\limits_{I\subseteq \{1, \dots, m\}} (A\times B)\setminus \big[ \bigcup_{i\in I} (A_i\times T)\cup \bigcup_{i\notin I} (S\times B_i)\big].$$
Therefore, by Proposition~\ref{prop:noncovering-of-foundation-differences}, in checking the (strong) C*-regularity (on the boundary) of $S\times T$, it suffices to consider the case when~\eqref{eq:hypstronggreg} is satisfied for $X=A\times B$ and $X_i=A_i\times T$ for $i\in I$, $X_i=S\times B_i$ for $i\notin I$ and $h_k=f_k\times g_k$ where $f_k\in I_\ell(S)$, $g_k\in I_\ell(T)$ for each $k=1, \ldots, n$. Then

$$X\setminus \bigcup\limits_{i=1}^m X_i=(A\setminus \cup_{i\in I} A_i)\times (B\setminus \cup_{i\notin I} B_i).$$ 

Denote by $\K(S)$ the collection of subsets $K\subseteq \{1, \ldots, n\}$ such that $\cup_{k\in K} \{s\in S \mid f_ks=s\}\supseteq A\setminus \cup_{i\in I} A_i$. Similarly, let $\K(T)$ denote the collection of subsets $K\subseteq \{1, \ldots, n\}$ such that $\cup_{k\in K} \{t\in T \mid g_kt=t\}\supseteq B\setminus \cup_{i\notin I} B_i$.\\

Consider first the case where $S$ and $T$ are (strongly) C*-regular. For $K\in \K(S)$, pick sets $A_{k, r}(K)\in \bar{\J}(S)$ (resp. $\J(S)$), where $(k, r)$ range over some finite index set such that $k\in K$ for all~$(k,r)$, with the properties

\begin{equation}\label{eq:ffix}
f_k p_{A_{k,r}(K)}=p_{A_{k,r}(K)} \quad \text{ for all }(k, r),
\end{equation}
and
$$A\setminus \cup_{i\in I} A_i\subseteq \bigcup\limits_{k, r}A_{k,r}(K).$$
For $K\in \mathcal{K}(T)$ pick analogous $B_{k, r}(K)$ sets such that

\begin{equation}\label{eq:gfix}
g_k p_{B_{k,r}(K)}=p_{B_{k,r}(K)} \quad \text{ for all }(k, r),
\end{equation}
and
$$B\setminus \cup_{i\notin I} B_i\subseteq \bigcup\limits_{k, r}B_{k,r}(K).$$
Consider the collection
\begin{equation}\label{eq:rectangles}
\mathcal{C}:=\{A_{k, r}(K)\times B_{k, r'}(K') \mid K\in \K(S),\ K'\in \K(T),\ k\in K\cap K',\ r \text{ and }r' \ \text{free}\}
\end{equation}
Then each member of $\mathcal{X}$ is fixed by some $f_k\times g_k$. We claim that $\mathcal{C}$ covers $X\setminus \bigcup\limits_{i=1}^m X_i$. Indeed, fix $(s, t)\in X\setminus \bigcup\limits_{i=1}^m X_i$. Put
$$K_s:=\{k\in \{1, \dots ,n\} \mid s\text{ belongs to some }A_{k, r}(K)\text{ set}\}.$$
Note that for each $t'\in B\setminus \cup_{i\notin I} B_i$ we have $\{k \mid g_kt'=t'\}\in \K(S)$. Therefore $K_s\in \K(T)$. Pick $k_0\in K_s$ such that $t$ lies in a $B_{k_0, r_0}(K_s)$ set. By definition of $K_s$, $s$ lies in some $A_{k_0, r}(K')$ set. Then $(s,t)\in A_{k_0, r}(K')\times B_{k_0, r_0}(K_s)\in \mathcal{C}$.\\

Now consider the case where $S$ and $T$ are (strongly) C$^*$-regular on the boundary. For each $K\in \K(S)$, pick sets $A_{k, r}(K)\in \bar{\J}(S)$ (resp. $\J(S)$) such that~\eqref{eq:ffix} holds and the $A_{k, r}(K)$ sets together with $A_i\ (i\in I)$ form a foundation set for $A$. Similarly pick $B_{k, r}(K)$ sets so that~\eqref{eq:gfix} holds and the $B_{k, r}(K)$ sets together with $B_i\ (i\notin I)$ form a foundation set for $B$. We claim that the collection $\mathcal{C}$ defined by~\eqref{eq:rectangles} together with $X_1, \dots, X_m$ forms a foundation set for $X$.\\

Indeed, suppose that $C\times D\subseteq X\setminus \cup_{i=1}^m X_i$ for some $C\in \J(S)\setminus \{\emptyset\}$ and $D\in \J(T)\setminus \{\emptyset\}$. Put
$$K_C=\{k\in \{1, \dots ,n\} \mid C\text{ intersects some }A_{k, r}(K)\text{ set}\}.$$
For each $t\in B\setminus \cup_{i\notin I} B_i$ we have $\{k \mid g_kt=t\}\in \K(S)$, so we must have $K_C\in \K(T)$. Pick a $k_0\in K_C$ such that $D$ intersects a $B_{k_0, r_0}(K_C)$ set. By definition of $K_C$, $C$ intersects some $A_{k_0, r}(K')$ set. Then $C\times D$ intersects $A_{k_0, r}(K')\times B_{k_0, r_0}(K_C)\in \mathcal{C}$.
\ep

Now we present nine left cancellative monoids $S_i$ ($1\leq i\leq 9$), meeting each combination of properties from the nine points. For five of the points modifications of $R$ will serve as examples. Applying Proposition~\ref{prop:monoid-product}, products of these will then serve as examples for the remaining points.	\\

\begin{enumerate}[(I)]
\item  $S_1:=S$ from \cite{NS} is a non-Hausdorff example, but any monoid such that $\G_P(\cdot)$ is Hausdorff, for example group embeddable monoids, work.\\
\item $S_2:=R$.\\
\item $S_3=S_2\times S_4$.\\
\item $$S_4=\langle a, b, c, d, x_n, y_n \text{ }(n\in\mathbb{Z})\mid abx_n=bx_n, \text{ }dbx_nw'=bx_nw',$$ $$aby_n=by_{n+1}, \text{ } cby_n=by_n \text{ }(n\in \mathbb{Z}, \text{ }w'\neq e)\rangle.$$\\
The constructible ideals are of the form
$$\emptyset,\quad wS_4,\quad w\bigcup\limits_{n}y_nS_4, \quad w\bigcup\limits_{n, w'\neq e}x_nw'S_4 \quad w(\bigcup\limits_{n}x_nS_4\cup\bigcup\limits_{n}y_nS_4) \quad (w\in S_4).$$
This example is similar to $S_5$ (which the reader is advised to read first) except that we have introduced the generator $d$ to ensure that  $\bigcup_{n, w'\neq e}x_nw'S_4 \in \J(S_4)$. Then $\{\bigcup_{n, w'\neq e}x_nw'S_4, \bigcup_n y_nS_4\}$ is a foundation set for $\bigcup_n x_nS_4\cup\bigcup_n y_nS_4$ which is the key point in showing that $S_4$ is strongly C$^*$-regular on the boundary. Identically to $S_5$, we have $\bigcup_n x_nS_4 \not \in \J(S_4)$ and $\bigcup_n x_nS_4 \in \bar{\J}(S_4)$. The former means that the argument we used to show that $T$ is not strongly C$^*$-regular can be used for $S_4$, whereas the latter ensures that $S_4$ is C$^*$-regular.\\

\item $$S_5:=\langle a, b, c, x_n, y_n\ (n\in\mathbb{Z}) \mid abx_n=bx_n, aby_n=by_{n+1}, cby_n=by_n\ (n\in \mathbb{Z})\rangle.$$\\
The constructible ideals are of the form
$$\emptyset,\qquad wS_5,\qquad w\bigcup\limits_{n}y_nS_5
, \qquad w(\bigcup_{n}x_nS_5\cup\bigcup_{n}y_n S_5) \qquad (w\in S_5).$$
We omit the proof here, but let us mention that the key point is that there is a generator, $c$, which only occurs in relations involving the $y_n$'s and not the $x_n$'s, while the same is not true the other way around: $a$ occurs in relations involving $x_n$'s \textit{and} $y_n$'s. This means that $\bigcup_n x_nS_5 \not \in \J(S_5)$, but $\bigcup_n y_nS_5 \in \J(S_5)$, therefore $\bigcup_{n}x_nS_5=\bigcup_{n}x_nS_5\cup\bigcup_{n}y_n S_5 \setminus \bigcup_{n}y_nS_5 \in \bar{\J}(S_5)$. The first fact implies that the criterion in Proposition~\ref{prop:equality-of-boundary-groupoids} is violated with $g=a$, $X=\bigcup_n x_nS_5\cup\bigcup_{n}y_n S_5$ and $X_1=\bigcup_{n}y_nS_5$, hence $\partial \G_P(S_5)\neq \partial\G(S_5)$. By the final fact, essentially the same argument we used for showing that $S$ is strongly C$^*$-regular can be used to show that $S_5$ is C$^*$-regular.\\

\item $S_6= S_2\times S_5$.\\
\item $S_7:=T$ from \cite{NS}.\\
\item $S_8=S_4\times S_7$.\\
\item $S_9=S_5\times S_7$.\\
\end{enumerate}

We finish this section with a result on free products.\\

\begin{proposition}
Let $S$ and $T$ be nontrivial left cancellative monoids. Consider the free monoid product $S*T$. Then the following are equivalent:\\
\begin{enumerate}[(i)]
\item $S*T$ is (strongly) C*-regular on the boundary.\\

\item $S*T$ is (strongly) C*-regular.\\

\item $S$ and $T$ are (strongly) C*-regular.\\
\end{enumerate}
Moreover,  $\partial\G_P(S*T)=\partial\G(S*T)$ if and only if $\G_P(S*T)=\G(S*T)$ if and only if $\G_P(S)=\G(S)$ and $\G_P(T)=\G(T)$.
\end{proposition}
\bp
We have embeddings $S\hookrightarrow S*T$ and $T\hookrightarrow S*T$ which give rise to embeddings $I_\ell(S)\hookrightarrow I_\ell(S*T)$ and $I_\ell(T)\hookrightarrow I_\ell(S*T)$. Accordingly, we will identify~$S$ and $T$ with their images in $S*T$ and $I_\ell(S)$ and $I_\ell(T)$ with their images in $I_\ell(S*T)$. Then $S\cap T=\{e\}$, where $e$ denotes the identity element of $S*T$, and $I_\ell(S)\cap I_\ell(T)=\{p_{S*T}\}$.\\

\begin{claim}
$$I_\ell(S*T)=\{xfy^{-1} \mid x, y\in S*T,\ f\in I_\ell(S)\cup I_\ell(T)\}\bigcup \{0\}.$$
\end{claim}

Note that if $s'\in S$ and $t'\in T$ are nontrivial elements, then $0={s'}^{-1}t'\in I_\ell(S*T)$. By symmetry, to prove the claim it is then enough to fix $s\in S$, $t\in T$, $x,y\in S*T$ and $f\in I_\ell(S)$, and show that $s^{-1}xfy^{-1}$ and $t^{-1}xfy^{-1}$ belong to the RHS.\\

\underline{$s^{-1}xfy^{-1}$:}\\

If $x\in S$ then  $s^{-1}xf\in I_\ell(S)$, so $s^{-1}xfy^{-1}$ is of desired form.\\

If $x\notin S$, we can write $x=s_1t_1\cdots s_nt_n$, where $s_i\in S$, $t_i\in T$ and $t_1\neq e$. Then
\begin{align*}
s^{-1} x f y^{-1} &= s^{-1} s_1 t_1 \cdots s_n t_n f y^{-1} \\
&= 
\begin{cases}
u t_1 \cdots s_n t_n f y^{-1} & \text{if } s_1 = s u \text{ for some } u \in S, \\
0 & \text{otherwise.}
\end{cases}
\end{align*}

\underline{$t^{-1}xfy^{-1}$:}\\

Suppose $x\in T$. Then, if $f=p_{S*T}$, we are done. If $f\neq p_{S*T}$, we have
\begin{align*}
t^{-1} x f y^{-1} &=
\begin{cases}
v f y^{-1} & \text{if } x = t v \text{ for some } v \in T, \\
0 & \text{otherwise.}
\end{cases}
\end{align*}

Now suppose that $x\notin T$. Write $x=t_1s_1\cdots t_ns_n$ where $s_i\in S$, $t_i\in T$ and $s_1\neq e$. Then
\begin{align*}
t^{-1} x f y^{-1} &= t^{-1} t_1 s_1 \cdots t_n s_n f y^{-1} \\
&=
\begin{cases}
v s_1 \cdots t_n s_n f y^{-1} & \text{if } t_1 = t v \text{ for some } v \in T, \\
0 & \text{otherwise.}
\end{cases}
\end{align*}
This proves the claim. Since the constructible right ideals of $S*T$ are exactly the domains of elements of $I_\ell(S*T)$, we have
$$\J(S*T)=\{xA(S*T) \mid x\in S*T,\ A \in \J(S)\cup\J(T)\}\bigcup \{\emptyset\}.$$
\smallskip
\begin{claim}\label{cl:claim2}
If $h\in I_\ell(S*T)$ maps an element of $S$ to $S$, then $h\in I_\ell(S)$.\\
\end{claim}

Indeed, we can write $h=xh'y^{-1}$ for some $x, y\in S*T$ and $h'\in I_\ell(S)\cup I_\ell(T)$. From the assumption, the domain of $h$ intersects $S$ so it follows that $y\in S$. Then~$h$ is of the form $xf$ or $xgs^{-1}$ for some $x\in S*T$, $f\in I_\ell(S)$, $g\in I_\ell(T)\setminus \{p_{S*T}\}$ and $s\in S$. However the domain of the latter does not intersect $S$ so $h=xf$. Then $hs'=xfs'$ for any $s'\in S$ and the assumption implies that~$x\in S$. Hence $h\in I_\ell(S)$ which proves the claim.\\

We assert that $S*T$ is strongly C$^*$-regular (on the boundary) if and only if Definition~\ref{def:strong-regularity2} (resp. Definition~\ref{def:strong-regularity-on-boundary}) is satisfied with $X, X_1, \dots ,X_m, Y_1, \dots ,Y_l$ required to be of the form
\begin{equation}\label{eq:nonshiftformJ}
A(S*T), \qquad A\in \quad \J(S) \text{(resp. }A\in \J(T) \text{)}
\end{equation}
and $h_1, \dots ,h_n\in I_\ell(S)$ (resp. $I_\ell(T)$). Similarly, we assert that $S*T$ is C$^*$-regular (on the boundary) if and only if Definition~\ref{def:regularity2} (resp. Definition~\ref{def:regularity-on-boundary}) is satisfied with $X, X_1, \dots ,X_m$ of the form~\eqref{eq:nonshiftformJ}, $h_1, \dots ,h_n\in I_\ell(S)$ (resp. $I_\ell(T)$) and $Y_1, \dots ,Y_l$ required to be of the form
\begin{equation}\label{eq:nonshiftformJbar}
A(S*T), \qquad A\in \bar{\J}(S) \text{ (resp. }A\in \bar{\J}(T) \text{)}.
\end{equation}

Let us first see why the new notions imply the standard regularity notions. Take arbitrary $X, X_1, \dots , X_m\in \J(S*T)$, $X_i\subseteq X$, and $h_1, \dots ,h_n\in I_\ell(S*T)$ such that~\eqref{eq:hypstronggreg} is satisfied. In verifying the usual regularity notions we may, by conjugating the $h_k$'s if necessary, assume that $X=A(S*T)$ for some $A\in \J(S)\cup \J(T)$. Without loss of generality we may assume that $A\in \J(S)$. Now let us discard all~$X_i$'s that are not of the form~\eqref{eq:nonshiftformJ} and all $h_k$'s that do not fix something in $A$. Then the $X_i$'s we discard are either of the form $x B(S*T)$ for some $x\in (S*T)\setminus S$ and $B\in \J(S)\cup \J(T)$, or $sC(S*T)$ for some $s\in S$ and $C\in \J(T)\setminus \{T\}$. For the remaining $X_i$'s we have $X\setminus \cup_i X_i=(A\setminus \cup_i A_i)(S*T)$ for some $A_i\in \J(S)$. By the second claim, the remaining~$h_k$'s lie in $I_\ell(S)$. Then~\eqref{eq:hypstronggreg} still holds for the new $X_i$'s and $h_k$'s: indeed, any $s\in A\setminus \cup_i A_i$ cannot belong to an $X_i$ set that we removed, thus it is fixed by some $h_k$ which patently is among the new ones. Then the $Y_j$ sets we obtain by the new notion also verify the standard regularity notions.\\

To see why the standard regularity notions imply the new notions, assume we are given $X=A(S*T)$ with $A\in \J(S)$, $Y\in \bar{\J}(S*T)$ and $h\in I_\ell(S)$ such that $\emptyset \neq Y\subseteq X$, $X\subseteq \dom h$ and $hp_Y=p_Y$. It is sufficient to find a $Y'\in \bar{\J}(S*T)$ of the form~\eqref{eq:nonshiftformJ} (resp.~\eqref{eq:nonshiftformJbar}) such that $Y\subseteq Y'\subseteq X$ and $hp_{Y'}=p_{Y'}$. Also, we need to ascertain that if $Y\in \J(S*T)$, then $Y'$ can be chosen to lie in $\J(S*T)$.\\

Write $Y=Y^0\setminus \cup_{r=1}^R Y^r$ for some $Y^0, Y^1,
\dots, Y^R \in \J(S*T)$ such that $Y^r\subseteq Y^0\subseteq 
X$ for all $r$.\\

If $Y^0$ is of form~\eqref{eq:nonshiftformJ}, then a similar argument to the one above shows that after discarding the~$Y^r$'s which are not of the form~\eqref{eq:nonshiftformJ}, and setting $Y'=Y^0\setminus \cup_r Y^r$ for the remaining $Y^r$'s, we still have $hp_{Y'}=p_{Y'}$.\\

Otherwise we either have $Y^0=xB(S*T)$ for some $x\in(S*T)\setminus S$ and $B\in \J(S)\cup \J(T)$, or $Y^0=sC(S*T)$ for some $s\in S$ and $C\in \J(T)\setminus \{T\}$. In either case there must be an element $s_1t_1\dots s_nt_n$, where $s_i\in S$, $t_i\in T$, $s_1\in A$ and $t_1\neq e$, that lies in $Y$. Then $s_1\in \dom h$, hence $h(s_1t_1\dots s_nt_n)=h(s_1)t_1\dots s_nt_n$. On the other hand we have $h(s_1t_1\cdots s_nt_n)=s_1t_1\cdots s_nt_n$. Together these equations imply that $h(s_1)=s_1$ so we can set $Y'=s_1(S*T)$.\\

We have shown that the new regularity notions are equivalent to the usual ones. To complete the proof of the proposition it is now enough to prove the following: if
$$X=A(S*T), \qquad Y_j=\Big( A_j^0\setminus \bigcup\limits_{i=1}^{m_j} A_j^i \Big)(S*T), \quad j=1,\dots ,l,$$
where $A, A_j^i\in \J(S)$ and $A_j^i\subseteq A_j^0\subseteq A$, then $\{Y_j\}_{j=1}^l$ is a foundation set for $X$ if 
and only if $X\subseteq \cup_{j=1}^l Y_j$ if and only if $A
\subseteq \cup_{j=1}^l (A_j^0\setminus \cup_{i=1}^{m_j} A_j^i)$. We show that the first statement implies the third, since the other implications are obvious. To that end, assume that $A\setminus\cup_{j=1}^l (A_j^0\setminus \cup_{i=1}^{m_j} A_j^i)$ is nonempty and pick an element $s_1$ of that set. Then, if $t_1$ is an arbitrary nontrivial element of $T$, we see that the constructible ideal $Y:=s_1t_1(S*T)$ is contained in $X$ and is disjoint from the $Y_j$'s.
\ep

\section{Appendix}

\subsection{The Constructible Right Ideals of $R$} We show that the constructible right ideals of $R$ are exactly the sets
\begin{equation}\label{eq:R-constructible-ideals}
wR,\ w(\bigcup\limits_n x_nR\cup \bigcup\limits_n y_nR), \ w(\bigcup\limits_{n, w'\neq e} x_nw'R), \ w(\bigcup\limits_{n, w'\neq e} y_nw'R),\ w(R\setminus\{e\}),\ \emptyset \qquad (w\in R).
\end{equation}
For readability, denote
$$Z=\bigcup\limits_n x_nR\cup \bigcup\limits_n y_nR,\quad X=\bigcup\limits_{n, w'\neq e} x_nw'R, \quad Y=\bigcup\limits_{n, w'\neq e} y_nw'R.$$ To establish the classification of constructible ideals  it suffices to check that for a letter $x$ and $w\in R$, the pullbacks
\begin{equation}\label{eq:R-pullbacks} 
x^{-1}wR,\quad x^{-1}wZ,\quad x^{-1}wX,\quad x^{-1}wY,\quad  x^{-1}w(R\setminus \{e\})
\end{equation}
are again of the form~\eqref{eq:R-constructible-ideals}. If $\alpha(w)=x$ then the pullbacks in~\eqref{eq:R-pullbacks} become  those in~\eqref{eq:R-constructible-ideals} with $w-\alpha(w)$ in place of $w$. From now on it will therefore be a standing assumption that $x\neq \alpha(w)$.  Then, if $x\in \{x_n,\ y_n\ (n\in \mathbb Z)\}$ we have $xR\cap wR=\emptyset$ and all ideals in~\eqref{eq:R-pullbacks} are $\emptyset$. Hence, by symmetry, we only need to check the cases $x=a, b, d$.\\

In general we can write $w\equiv sv$ where $\con(s)\subseteq \{a, c, d, f\}$ and $v=\epsilon$ or $\alpha(v)\in \ \{b, x_n, y_n \ (n\in \Z)\}$. If $\alpha(v)=x_n$ then $wp\equiv sx_n \perp vp$ for any $p\in R$ so $xR\cap wR=\emptyset$ and all pullbacks in~\eqref{eq:R-pullbacks} are $\emptyset$. Similarly when $\alpha(v)=y_n$. When $\alpha(v)=b$ write $w\equiv sbu$. If $\alpha(u)\in \{a, b, c, d, f \}$ then  $w\equiv sb\perp u$ and similarly all pullbacks in~\eqref{eq:R-pullbacks} are $\emptyset$. When $\alpha(u)=x_n, y_n$ we can write $w\equiv sbx_n\omega$ or $w\equiv sby_n\omega$. Suppose~$\omega\neq \epsilon$. If $\con(s)\subseteq \{a, c\}$ then
$$a^{-1}sbx_n\omega=bx_{n+m(c:s)}\omega, \quad b^{-1}sbx_n\omega=x_{n+m(c:s)}\omega, \quad d^{-1}sbx_n\omega=bx_{n+m(c:s)}\omega,$$
$$a^{-1}sby_n\omega=by_{n+m(a:s)-1}\omega, \quad b^{-1}sby_n\omega=y_{n+m(a:s)}\omega, \quad dR \cap sby_n\omega R=\emptyset.$$
If $s$ contains a $d$ but not an $f$,
$$a^{-1}sbx_n\omega=bx_{n+m(c:s)}\omega, \quad b^{-1}sbx_n\omega=x_{n+m(c:s)}\omega, \quad d^{-1}sbx_n\omega=bx_{n+m(c:s)}\omega,$$
$$xR \cap sby_n\omega R=\emptyset \text{ for }x=a,b,d.$$
If $s$ contains an $f$ but not a $d$,
$$xR \cap sbx_n\omega R=\emptyset \text{ for }x=a,b,d,$$
$$a^{-1}sby_n\omega=by_{n+m(a:s)-1}\omega, \quad b^{-1}sby_n\omega=y_{n+m(a:s)}\omega, \quad dP \cap sby_n\omega P=\emptyset.$$
If $s$ contains a $d$ and an $f$,
$$xR \cap sR=\emptyset \text{ for }x=a,b,d.$$

The computations above show that when $w\equiv sbx_n\omega$ or $w\equiv sby_n\omega$ for some $\omega\neq \epsilon$, then either $x^{-1}w\in R$, or the pullbacks in~\eqref{eq:R-pullbacks} are empty. In summary, we need only compute~\eqref{eq:R-pullbacks} for $x=a, b, d$ and $w\neq \epsilon$ such that $w\equiv s$, $w\equiv sb$, $w\equiv sbx_n$ and $w\equiv sby_n$, where $\con(s)\subseteq \{a, c, d, f\}$ and $s$ does not contain both $d$ and $f$, and $x\neq \alpha(w)$. We divide our argument into subcases depending on the nature of $s$, so let $s_0, s_d$ and $s_f$ be words such that
\begin{itemize}
\item $\con(s_0)\subseteq \{a, c\}$;
\item $\con(s_d)\subseteq \{a, c, d\}$ and $d$ occurs in $s_d$;
\item $\con(s_f)\subseteq\{a, c, f\}$ and $f$ occurs in $s_f$.
\end{itemize}

\underline{Case $w\equiv sb$:}\\
\underline{$x=a, b$}\\
We have,
$$aR\cap sbR=bR\cap sbR=aR\cap sb(R\setminus\{e\})=bR\cap sb(R\setminus\{e\})=aR\cap sbZ=bR\cap sbZ.$$
When $s=s_0$ this intersection equals $bZ$, when $s=s_d$ it equals $bX$ and when $s=s_f$ it equals $bY$. Furthermore,
$$aR\cap s_0bX=aR\cap s_dbX=bR\cap s_0bX=bR\cap s_dbX =bX$$
and
$$aR\cap s_fbX=bR\cap s_fbX=\emptyset.$$
We also have
$$aR\cap s_0bY=aR\cap s_fbY=bR\cap s_0bY=bR\cap s_f bY =bY$$
and
$$aR\cap s_fbY=bR\cap s_dbY=\emptyset.$$
\underline{$x=d$}\\
We have
$$dR\cap s_0bR=dR\cap s_0b(R\setminus\{e\})=dR\cap s_0bZ=dR\cap s_0bX$$$$=dR\cap s_dbR=dR\cap s_db(R\setminus\{e\})=dR\cap s_dbZ=dR\cap s_dbX$$$$=bX.$$
On the other hand
$$dR\cap s_0bY=dR\cap s_dbY=\emptyset.$$
Also
$$dP\cap s_fbR=\emptyset,$$
so the intersections corresponding to $s_fb$ are $\emptyset$.\\
\underline{Case $w\equiv s$:}\\
We have
$$aR\cap sR=aP\cap sbR, \quad aR\cap s(R\setminus\{e\})=aR\cap sb(R\setminus\{e\}),$$
$$bR\cap sR=bR\cap sbR, \quad bR\cap s(R\setminus\{e\})=bR\cap sb(R\setminus\{e\}),$$
$$dR\cap sP=dP\cap sbR, \quad dR\cap s(R\setminus\{e\})=dR\cap sb(R\setminus\{e\}),$$
and we computed the intersections on the RHS in the previous case. We also have
$$xR\cap sZ=\emptyset \text{ for }x=a,b, d,$$
hence $xR\cap sX=xR\cap sY=\emptyset$ for $x=a,b, d$ as well.\\

\underline{Case $w\equiv sbx_n$ or $w\equiv sby_n$:}\\
\underline{$x=a, b$:}\\
We have
$$a^{-1}s_0bx_n=bx_{n+m(c:s_0)}, \quad a^{-1}s_0by_n=by_{n+m(a:s_0)-1},$$
$$b^{-1}s_0bx_n=x_{n+m(c:s_0)}, \quad b^{-1}s_0by_n=x_{n+m(a:s_0)}.$$

which takes care of the $w\equiv s_0bx_n, s_0by_n$ cases. Turning to the $w\equiv s_dbx_n$ case,
$$aR\cap s_dbx_nR=aR\cap s_dbx_n(R\setminus\{e\})=bR\cap s_dbx_nR=bR\cap s_dbx_n(R\setminus\{e\})$$ $$=bx_{n+m(c:s_d)}(R\setminus\{e\}),$$
$$aR\cap s_dbx_nZ=bR\cap s_dbx_nZ=bx_{n+m(c: s_d)}Z,$$
$$aR\cap s_dbx_nX=bR\cap s_dbx_nX=bx_{n+m(c: s_d)}X,$$
$$aR\cap s_dbx_nY=bR\cap s_dbx_nY=bx_{n+m(c: s_d)}Y.$$

On the other hand, $aR\cap s_dby_nR=bR\cap s_dby_nR=\emptyset$ so the intersections corresponding to $w\equiv s_dby_n$ are  $\emptyset$.  Similarly, $aR\cap s_fbx_nR=bR\cap s_fbx_nR=\emptyset$ so the intersections corresponding to $w\equiv s_fbx_n$ are  $\emptyset$. Finally,
$$aR\cap s_fby_nR=aR\cap s_fby_n(R\setminus\{e\})=bR\cap s_fby_nP=bR\cap s_fby_n(R\setminus\{e\})$$ $$=by_{n+m(a:s_f)}(R\setminus\{e\}),$$
$$aR\cap s_fby_nZ=bR\cap s_fby_nZ=by_{n+m(a: s_f)}Z,$$
$$aR\cap s_fby_nX=bR\cap s_fby_nX=by_{n+m(a: s_f)}X,$$
$$aR\cap s_fby_nY=bR\cap s_fby_nY=by_{n+m(a: s_f)}Y.$$

\underline{$x=d$:}\\
We have already observed that $dR\cap s_fR=\emptyset$ so the intersections corresponding to $w\equiv s_fbx_n$ are~$\emptyset$. Similarly $dR\cap sby_nR=\emptyset$  so the intersections corresponding to $w\equiv sby_n$ are  $\emptyset$. For $s=s_0, s_d$ we have
$$dR\cap sbx_nR=sbx_nR\setminus \{sbx_n\}=dR\cap sbx_n(R\setminus\{e\})=sbx_n(R\setminus\{e\}),$$
$$dR\cap sbx_nZ=dR\cap dsbx_nZ=bx_{n+m(c: s)}Z,$$
$$dR\cap sbx_nX=dR\cap dsbx_nX=bx_{n+m(c: s)}X,$$
$$dR\cap sbx_nY=dR\cap dsbx_nY=bx_{n+m(c: s)}Y.$$
This completes our proof of the description of $\J(R)$.\\ 

\subsection{Miscellaneous Results} The following technical results about the monoid $R$ are needed for the main text, but are unenlightening enough that they belong here. 
\begin{replemma}{lem:multiple x_n elements}
If $r\in R$ and a constructible ideal $A\in \J(R)$ contains two $rx_n$ (resp. $ry_n$) elements, then $A$ contains all $rx_n$ and $ry_n$ elements. If $r, t_1, t_2\in R$ and a constructible ideal of $R$ contains $rx_{n_1}t_1$ and $rx_{n_2}t_2$ (resp. $ry_{n_1}t_1$ and $ry_{n_2}t_2$), $n_1\neq n_2$, then it contains the constructible ideal $r\Big(\bigcup\limits_{n, w'\neq e} x_nw'R\Big)$ (resp. $r\Big(\bigcup\limits_{n, w'\neq e} y_nw'R\Big)$). 
\end{replemma}

\bp
We begin by proving the first part of the lemma in the special case $r=e$;  let $A\in \J(R)$ be such that $x_{n_1}, x_{n_2}\in A$, $n_1\neq n_2$. We consider the various cases depending on what form in~\eqref{eq:R-constructible-ideals} $A$ has. If $A=wR$ for some $w\in R$ then $x_{n_1}=wp$ for some $p\in R$. We must have $w=e$ or $p=e$. If $w=e$, then $A=R$ and clearly contains all $x_n$ and $y_n$ elements. If $p=e$, then $w=x_{n_1}$, but then~$x_{n_2}$ cannot lie in $A=wR$. If $A$ is of the second type in~\eqref{eq:R-constructible-ideals} we again have $w=e$ and the conclusion holds. $A$ cannot be of the third or fourth type in~\eqref{eq:R-constructible-ideals}. Finally, assume that~$A=w(R\setminus\{e\})$ for some $w\in R$. Then $x_{n_1}=wp$ for some $p\neq e$. This forces $p=x_{n_1}$ so that $w=e$. Then $A$ clearly contains all $x_n$ and $y_n$ elements.\\

Now assume that $Y\in \J(R)$ contains $rx_{n_1}$ and $rx_{n_2}$, $n_1\neq n_2$. Then $x_{n_1}, x_{n_2}\in r^{-1}Y$, so by what we just proved, $r^{-1}Y$ contains all $x_n$ and $y_n$ elements. Then $r(r^{-1}Y)\subseteq Y$ contains all $rx_n$ and $ry_n$ elements.\\

Let us turn to the second part of the lemma. Again we start by proving it for $r=e$ so assume that $A\in \mathcal{J}(R)$ contains $x_{n_1}t_1$ and $x_{n_2}t_2$, $n_1\neq n_2$, $t_1, t_2\in R$. If $A=wR$, then $x_{n_1}t_1=wp$ for some $p\in R$. Then either $w=e$ or $w$ begins with $x_{n_1}$. But $w$ cannot begin with $x_{n_1}$ as this contradicts that $x_{n_2}x\in A$. Hence $w=e$ and the conclusion holds. Similarly, if $A$ is of the second, third or fifth type in~\eqref{eq:R-constructible-ideals} we must again have $w=e$ and the conclusion holds. Clearly $A$ cannot be of the fourth type. This completes the proof when $r=e$.\\

In the general case, suppose that $B\in \mathcal{J}(R)$ contains $rx_{n_1}t_1$ and $rx_{n_2}t_2$. Then $x_{n_1}t_1, x_{n_2}t_2\in r^{-1}B$ so by the previous part
$\bigcup\limits_{n, w'\neq e} x_nw'R\subseteq r^{-1}B$. Then

$$r(\bigcup\limits_{n, w'\neq e} x_nw'R)\subseteq r(r^{-1}B)\subseteq B.$$

This completes the proof of the lemma.
\ep

\begin{replemma}{lem:fix x_n or x_nx}
If $h\in I_\ell(R)$ fixes an $x_n$ element and $\dom h$ contains multiple $x_n$ 
elements, then $h$ fixes all $x_n$ elements. If $x$ is a letter, $h\in I_\ell(R)$ fixes an 
$x_nx$ element and $\dom h$ contains multiple $x_nx$ elements, then $h$ 
fixes the constructible ideal $\bigcup\limits_{n, w'\neq e} x_nw'R$.
\end{replemma}

\bp
For readability we will throughout this proof write $\un_A$ for the idempotent~$p_A$ corresponding to $A\in \J(S)$. It is sufficient to prove the following claim: if $X=\cup_{n, w'\neq e} x_nw'R$ and $g\in I_\ell(R)$ satisfies $\dom g \supseteq X$, then $g$ must have the form
\begin{equation}\label{eq:gform}
w\un_A, \quad wa^{-k_1}c^{-k_2}b\un_A \text{   or   }wb^{-1}a^{l_1}c^{l_2}b\un_A
\end{equation}
where $w\in R$,  $A\in \J(R)$, $k_1, k_2\geq 0$ and $l_1, l_2\in \mathbb Z$.\\

Indeed, if $h$ satisfies the hypothesis in the first part of the lemma then $\dom h \supseteq \bigcup_nx_nR\cup \bigcup_n y_nR\supseteq X$ by Lemma~\ref{lem:multiple x_n elements}. The claim implies that either $hx_n=wx_n$ for all $n \in \mathbb Z$, or $hx_n=wbx_{n-k_2}$  for all $n \in \mathbb Z$, or $hx_n=wx_{n+l_2}$  for all $n \in \mathbb Z$. By assumption, $h$ fixes some~$x_n$; this rules out the second case, in the first case it implies that $w=e$ and in the third case it implies that $w=e$ and $l_2=0$. In either case, $h$ must fix all $x_n$ elements.\\

If $h$ satisfies the hypothesis of the second part of the lemma then $\dom h \supseteq X$ by Lemma~\ref{lem:multiple x_n elements}. The claim implies that either $h(x_nx)=wx_nx$ for all $n \in \mathbb Z$, or $h(x_nx)=wbx_{n-k_2}x$ for all $n\in \mathbb Z$, or $h(x_nx)=wx_{n+l_2}x$ for all $n \in \mathbb Z$.  By assumption $h$ fixes some $x_nx$; this rules out the second case, in the first case it implies that $w=\epsilon$ and in the third case it implies that $w=\epsilon$ and $l_2=0$. In either case, $h$ must fix $X$.\\

Now, to prove the claim it is enough to show that for a letter $x$, a word $w$, $k_1, k_2\geq 0$ and $l_1, l_2\in \mathbb Z$,  the partial bijections $x^{-1}w$, $x^{-1}wa^{-k_1}c^{-k_2}b$ and $x^{-1}wb^{-1}a^{l_1}c^{l_2}b$  are either of the form~\eqref{eq:gform}
or their domains do not contain $X$.  If $w\neq \epsilon$ and $x=\alpha(w)$ then these partial bijections 
are clearly of form form~\eqref{eq:gform}. If $x \in  \{x_n, y_n \ (n\in \mathbb Z)\}$ and  $w=\epsilon$, then
$$\dom(x^{-1}w)= \dom(x^{-1}wb^{-1}a^{l_1}c^{l_2}b)=x_nR \text{ or }y_nR, \qquad \dom(x^{-1}wa^{-k_1}c^{-k_2}b)=\emptyset.$$
If $w\neq \epsilon$, $x\neq \alpha(w)$, and either $x\in  \{x_n, y_n \ (n\in \mathbb Z)\}$ or $\alpha(w)\in  \{x_n, y_n \ (n\in \mathbb Z)\}$, then
$$\dom(x^{-1}w)=\dom(x^{-1}wa^{-k_1}c^{-k_2}b)=\dom (x^{-1}wb^{-1}a^{l_1}c^{l_2}b)=\emptyset.$$
or $\alpha(w)\in \{x_n, y_n  (n\in \mathbb{Z})\}$. Throughout the argument below we will therefore assume that $x\in \{a, b, c, d, f\}$ and that if $w\neq \epsilon$, then $\alpha(w)\in \{a, b, c, d, f\}$ and $x\neq \alpha(w)$.\\

\underline{$x^{-1}wa^{-k_1}c^{-k_2}b$:}\\
Denote $Y=\bigcup_{n, w'\neq e} y_nw'R$ and $Z=\bigcup_n x_nR\cup \bigcup_n y_nR$.  If $w=\epsilon$ we have the following possibilities for $x^{-1}wa^{-k_1}c^{-k_2}b$:
$$a^{-1}a^{-k_1}c^{-k_2}b=a^{-k_1-1}c^{-k_2}b, \quad b^{-1}a^{-k_1}c^{-k_2}b,$$
$$c^{-1}a^{-k_1}c^{-k_2}b=c^{-1}a^{-k_1}c^{-k_2}b\un_Z=a^{-k_1}c^{-k_2-1}b,$$
$$d^{-1}a^{-k_1}c^{-k_2}b=a^{-k_1}c^{-k_2}b\un_X, \quad f^{-1}a^{-k_1}c^{-k_2}b=a^{-k_1}c^{-k_2}b\un_Y.$$
These are all of desired form~\eqref{eq:gform}. We may therefore assume that $w\neq\epsilon$. Then we can write $w\equiv sv$ with $\con(s)\subseteq \{a,c,d,f\}$, $\alpha(v)\neq a,c,d,f$. If $v=\epsilon$ we have the following possibilities for $x^{-1}wa^{-k_1}c^{-k_2}b=x^{-1}sa^{-k_1}c^{-k_2}b$.\\

\begin{itemize}
	\item If $\con(s)\subseteq\{a, c\}$:
$$a^{-1}sa^{-k_1}c^{-k_2}b=a^{m(a:s)-1-k_1}c^{m(c:s)-k_2}b \un_Z,$$
$$b^{-1}sa^{-k_1}c^{-k_2}b=b^{-1}a^{m(a:s)-k_1}c^{m(c:s)-k_2}b,$$
$$c^{-1}sa^{-k_1}c^{-k_2}b=a^{m(a:s)-k_1}c^{m(c:s)-1-k_2}b \un_Z,$$
$$ d^{-1}sa^{-k_1}c^{-k_2}b=c^{m(c:s)-k_2}b\un_X,$$
$$f^{-1}a^{-k_1}c^{-k_2}b=a^{m(a:s)-k_1}b\un_Y.$$

	\item If $s$ contains a $d$ but not an $f$:
$$a^{-1}sa^{-k_1}c^{-k_2}b=c^{m(c:s)-k_2}b \un_X, \quad b^{-1}sa^{-k_1}c^{-k_2}b=b^{-1}c^{m(c:s)-k_2}b \un_X, $$ $$c^{-1}sa^{-k_1}c^{-k_2}b=c^{m(c:s)-1-k_2}b \un_X, \quad d^{-1}sa^{-k_1}c^{-k_2}b=c^{m(c:s)-k_2}b\un_X,$$
$$f^{-1}sa^{-k_1}c^{-k_2}b=\un_\emptyset.$$

	\item If $s$ contains an $f$ but not a $d$:
$$a^{-1}sa^{-k_1}c^{-k_2}b=a^{m(a:s)-k_1-1}b \un_Y, \quad b^{-1}sa^{-k_1}c^{-k_2}b=b^{-1}a^{m(a:s)-k_1}b \un_Y, $$ $$c^{-1}sa^{-k_1}c^{-k_2}b=a^{m(a:s)-k_1}b \un_Y, \quad d^{-1}sa^{-k_1}c^{-k_2}b=\un_\emptyset,$$
$$f^{-1}sa^{-k_1}c^{-k_2}b=a^{m(a:s)-k_1}b\un_Y.$$

	\item If $s$ contains a $d$ and an $f$:
$$x^{-1}sa^{-k_1}c^{-k_2}b=\un_\emptyset \text{ for }x=a, b, c, d, f.$$
\end{itemize}
These are all of desired form~\eqref{eq:gform} and we can focus on the case $v\neq \epsilon$.  If $\alpha(v)=x_n, y_n$, then $s\alpha(v)\perp p$ for all $p\in R$ so $\dom(x^{-1}wa^{-k_1}c^{-k_2}b)=\emptyset$ in this case.  Otherwise $\alpha(v)=b$  and we can write $w\equiv sbu$. If $u=\epsilon$ or $\alpha(u)=a, b, c, d, f$ then $\ran(wa^{-k_1}c^{-k_2}b)=\ran(sbua^{-k_1}c^{-k_2}b)\subseteq sbubR$ which is disjoint from $xR$. Hence $\dom(x^{-1}wa^{-k_1}c^{-k_2}b)=\emptyset$.  Otherwise $\alpha(u)=x_n, y_n$ and we can write $w\equiv sbx_n\omega$ or $w\equiv sby_n\omega$. By symmetry we may focus on the former case.  Then we have the following possibilites for $x^{-1}w$ (*):\\

\begin{itemize}
	\item If $\con(s)\subseteq\{a, c\}$:
$$a^{-1}sbx_n\omega=sbx_n\omega, \quad b^{-1}sbx_n\omega=x_{n+m(c: s)}\omega, \quad c^{-1}sbx_n\omega=sbx_{n-1}\omega,$$
$$d^{-1}sbx_n\omega=sbx_n \un_{R\setminus \{e\}}\omega,\quad f^{-1}sbx_n\omega=\un_\emptyset.$$

	\item If $s$ contains a $d$ but not an $f$:
$$x^{-1}sbx_n\omega=sbx_n\un_{R\setminus \{e\}}\omega\text{ for }x=a, d, \quad b^{-1}sbx_n\omega=x_{n+m(c:s)}\un_{R\setminus \{e\}}\omega,$$
$$c^{-1}sbx_n\omega=sbx_{n-1}\un_{R\setminus \{e\}}\omega,\quad f^{-1}sbx_n\omega=\un_\emptyset.$$

	\item If $s$ contains an $f$:
$$x^{-1}sbx_n\omega=\un_\emptyset \text{ for }x=a, b, c, d, f.$$
\end{itemize}

In general we have $g_1\un_Ag_2=g_1g_2\un_{g_2^{-1}(A)}$ for any $g_1, g_2\in I_\ell(R),\ A\in \J(R)$, therefore one sees that when right multiplying any of the partial bijections above by $a^{-k_1}c^{-k_2}b$  we get something of desired form~\eqref{eq:gform}.\\

\underline{$x^{-1}w$:}\\
If $w=\epsilon$, then the possible domains of $x^{-1}w=x^{-1}$ are $aR, bR, cR, dR, fR$, none of which contain~$X$. We may therefore assume that $w\neq\epsilon$. Write $w\equiv sv$ with $\con(s)\subseteq \{a,c,d,f\}$, $\alpha(v)\neq a, c, d, f$. If $v=\epsilon$, then $s\neq \epsilon$ and $\dom(x^{-1}v)$ is one of
$$bZ, \quad bX, \quad bY, \quad \emptyset$$
None of these contain $X$, so we may assume that $v\neq \epsilon$. If $\alpha(v)=x_n, y_n$ then $s\alpha(v)\perp (v-\alpha(v))p$ for any $p\in R$ so that $\dom(x^{-1}w)=\emptyset$.\\

If $\alpha(v)=b$, write $w\equiv sbu$. If $u=\epsilon$ then we are in the previous \underline{$x^{-1}wa^{-k_1}c^{-k_2}b$} case with $k_1=k_2=0$. Assume therefore that $u\neq \epsilon$.\\

If $\alpha(u)=a,b,c,d,f$, then $sb\perp up$ for any $p\in R$ so  $\dom(x^{-1}w)=\emptyset$. If $\alpha(u)=x_n, y_n$ we can write $x^{-1}w\equiv x^{-1}sbx_n\omega$ or $x^{-1}w\equiv x^{-1}sby_n\omega$. This case was already dealt with in (*).\\

\underline{$x^{-1}wb^{-1}a^{l_1}c^{l_2}b$:}\\
Note that at least one of $l_1$ and $l_2$ is nonzero, otherwise $x^{-1}wb^{-1}a^{l_1}c^{l_2}b=x^{-1}w$ and we are in the first case. Therefore we have $\ran(b^{-1}a^{l_1}c^{l_2}b)=\bigcup_n x_nR \cup \bigcup_n y_nR$ and if $w=\epsilon$ we obtain $\dom(x^{-1}b^{-1}a^{l_1}c^{l_2}b)=\emptyset$. In the remainder we may therefore assume that $w\neq \epsilon$.\\

Write $w\equiv sv$ where $\con(s)\subseteq \{a, c, d, f\}$, $\alpha(v)\neq a, c, d, f$.  If $v=\epsilon$ or $\alpha(v)=x_n, y_n$ then $xR\cap wR=\emptyset$ so $\dom(x^{-1}wb^{-1}a^{l_1}c^{l_2}b)=\emptyset$. Otherwise we can write $w\equiv sbu$.\\

If $u=\epsilon$ then $x^{-1}wb^{-1}a^{l_1}c^{l_2}b=x^{-1}sbb^{-1}a^{l_1}c^{l_2}b=x^{-1}sa^{l_1}c^{l_2}b \un_{g^{-1}(bR)}$ where $g=a^{l_1}c^{l_2}b$ so we can apply the \underline{$x^{-1}wa^{-k_1}c^{-k_2}b$} case. Otherwise $u\neq \epsilon$. If $\alpha(u)=a, b,c,d,f$ then $sb\perp up$ for all $p\in R$ so $\dom(x^{-1}wb^{-1}a^{l_1}c^{l_2}b)=\emptyset$. Otherwise $ \alpha(u)=x_n, y_n$ and we can write $w\equiv sbx_n\omega$ or $w\equiv sby_n\omega$. Right multiplying the partial bijections in (*) by $b^{-1}a^{l_1}c^{l_2}b$ and applying the rule $g_1\un_Ag_2=g_1g_2\un_{g_2^{-1}(A)}$  we see that $x^{-1}wb^{-1}a^{l_1}c^{l_2}b$ is of desired form also in this case.
\ep

\begin{lemma}\label{lem:infinitewordtype2help}
Let $v$ be a word and $w$ be an infinite word such that $w_i\in \{a, c, d\}$ for all $i\in\mathbb{N}$. Then, if a constructible ideal $A\in \mathcal{J}(R)$ does not contain $vw_1w_2\cdots w_n$ for any $n\geq 1$, there exists a $k\geq 1$ such that
$$A\cap vw_1w_2\cdots w_k R\subseteq vw_1w_2\cdots w_{k-1}\big(\bigcup\limits_n bx_nR \cup \bigcup\limits_n by_nR\big).$$
\end{lemma}
\bp
Let us first assume that $v=\epsilon$. Suppose $A=pR$ or $A=p(R\setminus \{e\})$, $p\in R$, and that either ideal satisfies the hypothesis in the lemma. Then $p\neq e$ and we can write $p=p_1p_2\cdots p_n$ where the $p_i$'s are letters. There must be an index $i$ such that $p_i\neq w_i$. Let $k$ be the smallest such index. Then

\begin{equation}\label{eq:Aintersectioninclusion}
A\cap w_1w_2\cdots w_k R\subseteq p_1\cdots p_k R\cap w_1w_2\cdots w_k R=w_1\cdots w_{k-1}(p_k R\cap w_k R)
\end{equation}
$$\subseteq w_1\cdots w_{k-1}\big( \bigcup\limits_n bx_nR \cup \bigcup\limits_n by_nR \big).$$
Here we have used the fact that for any two distinct generators $x$ and $y$ we have $xR\cap yR \subseteq \bigcup\limits_n bx_nR \cup \bigcup\limits_n by_nR$.\\

Now suppose that ideals $p\big(\bigcup\limits_n x_nR \cup \bigcup\limits_n y_nR \big)$, $p\big(\bigcup\limits_{n, w'\neq e} x_nw'R\big)$ or $p\big(\bigcup\limits_{n, w'\neq e} y_nw'R\big)$ satisfy the hypothesis for $A$ in the lemma. If $p=w_1\cdots w_m$ for some $m\geq 0$, then

$$p\big(\bigcup\limits_n x_nR \cup \bigcup\limits_n y_nR \big)\cap w_1\cdots w_{m+1}R=w_1\cdots w_m\big( \big(\bigcup\limits_n x_nR \cup \bigcup\limits_n y_nR\big)\cap w_{m+1}R \big)=\emptyset$$
and similarly

$$p\big(\bigcup\limits_{n, w'\neq e} x_nw'R\big)\cap w_1\cdots w_{m+1}R=\emptyset, \qquad p\big(\bigcup\limits_{n, w'\neq e} y_nw'R\big)\cap w_1\cdots w_{m+1}R=\emptyset.$$
Setting $k=m+1$ yields the desired inclusion. If $p\neq w_1\cdots w_m$ for any $m\geq 0$, then writing $p=p_1p_2\cdots p_n$ where the $p_i$'s are letters, there must be a smallest $k$ such that $p_k\neq w_{k}$. Then~\eqref{eq:Aintersectioninclusion} holds. We have shown that the lemma holds for~$v=\epsilon$.\\

Let us turn to the general case. Suppose $A\in \mathcal{J}(R)$ does not contain $vw_1w_2\cdots w_n$ for any $n\geq 1$. Then the ideal $v^{-1}A$ does not contain $w_1w_2\cdots w_n$ for any $n\geq 1$. By the previous case we have

$$v^{-1}A\cap w_1w_2\cdots w_k R\subseteq w_1w_2\cdots w_{k-1}\big(\bigcup\limits_n bx_nR \cup \bigcup\limits_n by_nR\big)$$
for some $k\geq 1$. Then
$$v(v^{-1}A)\cap vw_1w_2\cdots w_k R=v(v^{-1}A\cap w_1w_2\cdots w_k R)$$
$$\subseteq vw_1w_2\cdots w_{k-1}\big(\bigcup\limits_n bx_nR \cup \bigcup\limits_n by_nR\big).$$
Now, the LHS equals $A\cap vw_1w_2\cdots w_k R$ because $v(v^{-1}A)=A\cap vS$. Then
$$A\cap vw_1w_2\cdots w_k R\subseteq vw_1w_2\cdots w_{k-1} \big(\bigcup\limits_n bx_nR \cup \bigcup\limits_n by_nR \big).$$
\ep

\begin{lemma}\label{lem:vrequivalent}
Assume $v$ and $r$ are words such that $\con(r)\subseteq \{a, c, d, f\}$. Then any word equivalent to $vr$ ends in $r$, i.e.

$$[vr]\subseteq \{v'r \mid v'\text{ is a word}\}.$$
\end{lemma}
\bp
It is enough to show that any word which is elementary equivalent to $vr$ also ends in $r$, so suppose $(p, q)\in\tau$ and $vr\equiv z_1pz_2$. If $p$ begins and ends in $v$ then clearly $z_1qz_2$ ends in $r$. If $p$ begins in $v$ and ends in $r$ then we must have $p\equiv dbx_n\alpha(r)$, $p\equiv bx_n\alpha(r)$, $p \equiv fby_n\alpha(r)$ or $p \equiv by_n\alpha(r)$ for some $n\in \mathbb{Z}$, and $z_2\equiv r-\alpha(r)$. In any case $q$ ends in $\alpha(r)$ so $z_1qz_2$ ends in $r$. Finally, $r$ does not contain any $\tau$-words so $p$ cannot begin and end in $r$.
\ep

\begin{lemma}\label{lem:dbx_n-separation}
Suppose $v$ is a word which is \textit{not} of the form $rbx_nv'$ for an $n\in \mathbb{Z}$ and words $r$ and $v'$ satisfying $\con(r)\subseteq \{a, c, d\}$ and $v'\neq \epsilon$. Then, for any word $u$,
$$ud\perp v$$
Symmetrically, if $v$ is \textit{not} of the form $rby_nv'$ for an $n\in \mathbb{Z}$ and words $r$ and $v'$ satisfying $\con(r)\subseteq \{a, c, f\}$ and $v'\neq \epsilon$, then for any word $u$,
$$uf\perp v.$$
\end{lemma}
\bp
By Lemma~\ref{lem:vrequivalent}, any word equivalent to $ud$ will again end in $d$. By Equation~\eqref{eq:equiv-bx_nw} any word equivalent to something which is not of the form $rbx_nv'$ for some $n\in \mathbb{Z}$ and words $r$ and $v'$ such that $\con(r)\subseteq \{a, c, d\}$ and $v'\neq \epsilon$, is again not of that form. Therefore, it is enough to prove that
\begin{equation}\label{eq: vdperp}
ud\perp_0 v.
\end{equation}
If $p$ is a $\tau$-word in $udv$ which begins in $ud$ and contains the $d$, then because of the assumption on~$v$, $p$ cannot begin in the $d$. We must have $p\equiv bx_nd$, $ dbx_nd$, $p\equiv by_nd$ or $p\equiv fby_nd$ for some~$n\in \mathbb{Z}$. In all four cases, $p$ ends in $ud$ so~\eqref{eq: vdperp} holds.
\ep

\begin{lemma}\label{lem:ubx_nw}
Fix $n\in \Z$  and let $w\neq \epsilon$ be a nonempty word.
If $u=\epsilon$ or $\beta(u)\notin \{a, c, d\}$, then
\begin{equation}\label{eq:ubx_nw}
[ubx_nw]=\{u'sbx_{n-m(c:s)} w' \mid \con(s)\subseteq \{a, c, d\},\ u'bx_n\sim ubx_n,\ w'\sim w\}
\end{equation}
Symmetrically, if $u=\epsilon$ or $\beta(u)\notin \{a, c, f\}$, then
\begin{equation}\label{eq:uby_nw}
[uby_nw]=\{u'sby_{n-m(a:s)} w' \mid \con(s)\subseteq \{a, c, f\},\ u'by_n\sim uby_n,\ w'\sim w\}
\end{equation}
\end{lemma}
\bp
We prove~\eqref{eq:ubx_nw}. Let $u'sbx_{n-m(c:s)} w'\in \text{RHS}$ where $\con(s)\subseteq \{a, c, d\},\ u'bx_n\sim ubx_n$ and $w'\sim w$. Then
$$u'sbx_{n-m(c:s)} w' \sim u'b x_n w \sim ubx_nw \in \text{LHS}.$$
Therefore $\text{RHS}\subseteq \text{LHS}$. To establish the reverse inclusion we need to prove that the RHS is closed w.r.t elementary $\tau$-equivalence. Let $u'sbx_{n-m(c:s)} w'$ be the same element as above. Then $u'bx_n \in [ubx_n]$ and by the assumption on $u$, any word in~$[ubx_n]$ must end in $bx_n$ and does not end in $dbx_n$. Therefore $\beta(u')\notin \{a, c, d\}$. Now let $(p, q)\in \tau$ be such that
$$u'sbx_{n-m(c:s)} w'\equiv z_1pz_2$$
for some words $z_1$ and $z_2$. We must show that 
$z_1qz_2$ belongs to the RHS. If $p$ begins in the $w'$, 
or begins in the $sb$, or begins and ends in the $u'$, then this is clear. Since $\beta(u')\notin \{a, c, d\}$ the remaining case is that $p$ begins in $u'$ and $p=bx_{n'}x,\ dbx_{n'}x,\ by_{n'}x,\ fby_{n'}x$ where $x$ is the first letter of $sb$, i.e. if $s\neq \epsilon$ then $x=\alpha(s)$, if $s=\epsilon$ then $x=b$.\\

\underline{Case: $p=bx_{n'}x$}: Write $u'\equiv u_0'bx_{n'}$. Then $$z_1qz_2\equiv u_0'dbx_{n'}sbx_{n-m(c:s)}w'.$$
Since we have $u_0'dbx_{n'}bx_n\sim u_0'bx_{n'}bx_n\equiv u'bx_n\sim ubx_n$ we see that $z_1qz_2\in \text{RHS}$.\\

\underline{Case: $p=dbx_{n'}x$}: Write $u'\equiv u_0'dbx_{n'}$. Then $$z_1qz_2\equiv u_0'bx_{n'}sbx_{n-m(c:s)}w'.$$
Since we have $u_0'bx_{n'}bx_n\sim u_0'dbx_{n'}bx_n\equiv u'bx_n\sim ubx_n$ we see that $z_1qz_2\in \text{RHS}$.\\

The cases $p=by_{n'}x$ and $p=fby_{n'}x$ are similar.
\ep

\begin{replemma}{lem:reducedinfinitewordsneighborhood}
Let $x=\chi_w$, where $w\in\RR_{\text{red}}^\infty$, and fix $n\geq 1$. Then
\begin{align*}
U:=\{\eta\in \Omega(R) \mid \ & \eta(w_1\cdots w_nR)=1, \\ & \eta(w_1\cdots w_{n-1}b(\bigcup_{n'} x_{n'}R\cup \bigcup_{n'} y_{n'}R))=0 \text{ if }w_n\in\{a, c, d, f\},\\
& \eta(w_1\cdots w_{n-2}b(\bigcup_{n'} x_{n'}R\cup \bigcup_{n'} y_{n'}R))=0 \text{ if }w_{n-1}\in\{a, c, d, f\}\}
\end{align*}
is a neighbourhood of $x$ such that for any $\chi_v\in U$, where $v\in\RR_{\text{red}}^\infty$, we have $v_1 v_2\cdots v_n\equiv w_1 w_2\cdots w_n$.
\end{replemma}
\bp
Put $Z=\bigcup_{n'} x_{n'}R\cup \bigcup_{n'} y_{n'}R$. We first argue that $\chi_w\in U$. It is clear that $\chi_w(w_1\cdots w_nR)=~1$. Now assume that $w_n\in~\{a, c, d, f\}$. We must show that $\chi_w(w_1\cdots w_{n-1}bZ)=0$, i.e. that $w_1\cdots w_m \notin w_1\cdots w_{n-1}bZ$ for any $m\geq 1$. If not, then
$$w_1\cdots w_m=w_1\cdots w_{n-1}bx_k p$$
for some $m\geq 1$, $k\in \Z$ and $p\in R$. Left cancellativity yields $$w_n\cdots w_m =bx_kp.$$

Write $w_n\cdots w_m\equiv rw_{n_1}\cdots w_m$ where $\con(r)\subseteq \{a, c, d, f\}$ and $w_{n_1}\notin \{a, c, d, f\}$. If $w_{n_1}\in~\{x_n, y_n \mid n\in \Z\}$ then $rw_{n_1} \perp w_{n_1+1}\cdots w_m$ and we have a contradiction. Similarly, if $w_{n_1}=b$ and $w_{n_1+1}\notin \{x_n, y_n \mid n\in \Z\}$. Therefore we must have $w_{n_1}w_{n_1+1}\equiv bx_{n'}$ or $w_{n_1}w_{n_1+1}\equiv by_{n'}$ for some $n'\in \Z$, but this contradicts $w$ being reduced. Similarly it follows that if $w_{n-1}\in~\{a, c, d, f\}$ then $\chi_w(w_1\cdots w_{n-2}bZ)=0$. This shows that $\chi_w\in U$.\\

Now assume that $\chi_v \in U$ where $v\in \RR_{\text{red}}^\infty$. We have $\chi_v(w_1\cdots w_nR)=1$, hence
$$v_1\cdots v_m=w_1 \cdots w_n p$$
for some $m\geq 1$ and $p\in R$. Since $v$ is reduced, Lemma~\ref{lem:dbx_n-separation} implies that $p$ does not end in $dbx_{n'}$ or $fby_{n'}$ for an $n'\in \Z$. Therefore we can write $p\sim p_1\cdots p_l$ where the $p_i$'s are letters and $p_1\cdots p_l$ is a reduced word. We argue that the whole word $w_1\cdots w_n p_1\cdots p_l$ is reduced. If it weren't then one of the following must hold for some $k\in \Z$:
$w_n=a, c, d$ and $p_1p_2\equiv bx_k$, $w_n=a, c, f$ and $p_1p_2\equiv by_k$, $w_{n-1}w_n\equiv ab, cb, db$ and $p_1=x_k$, or $w_{n-1}w_n\equiv ab, cb, fb$ and $p_1=x_k$. If one of the first two are true this contradicts the fact that $\chi_v(w_1\cdots w_{n-1}bZ)=0$ if $w_n\in \{a, c, d, f\}$. If one of the last two are true this contradicts the fact that $\chi_v(w_1\cdots w_{n-2}bZ)=0$ if $w_{n-1}\in \{a, c, d, f\}$.\\

Having established that $w_1\cdots w_n p_1\cdots p_l$ is reduced, Proposition~\ref{prop:distinctfinitereducedwords} implies that $v_1\cdots v_m\equiv w_1 \cdots w_np_1\cdots p_l$, in particular $v_1\cdots v_n\equiv w_1\cdots w_n$.\\
\ep

\bigskip


\bigskip

\begin{bibdiv}
\begin{biblist}

\bib{BGHL}{misc}{
author={Aguyar Brix, Kevin},
	author={Gonzales, Julian},
	author={B. Hume, Jeremy},
	author={Li, Xin},
    title={On Hausdorff covers for non-Hausdorff groupoids},
    how={preprint},
    date={2025},
    eprint={\href{https://arxiv.org/pdf/2503.23203}{\texttt{arXiv:2503.23203v1 [math.OA]}}},
}

\bib{AR}{book}{
   author={Anantharaman-Delaroche, Claire},
   author={Renault, Jean},
   title={Amenable Groupoids},
   series={Monographie de l'Enseignement math{\'e}matique},
   publisher={L'Enseignement Math{\'e}matique},
   date={2000},
   isbn={9782940264018},
   url={https://books.google.no/books?id=HbQrAAAAYAAJ},
}

\bib{CN-2}{article}{
	  author={Christensen, Johannes}
	  author={Neshveyev, Sergey}
      title={Isotropy fibers of ideals in groupoid C$^*$-algebras},
      date={2024}, 
      journal={Advances in Mathematics},
      volume={447},
      pages={109696},
      issn={0001-8708},
      doi={https://doi.org/10.1016/j.aim.2024.109696},
      url={https://www.sciencedirect.com/science/article/pii/S0001870824002111},
}

\bib{CELY}{book}{
   author={Cuntz, Joachim},
   author={Echteroff, Siegfried},
   author={Li, Xin},
   author={Yu, Guoliang},
   title={K-Theory for Group C$^*$-Algebras and Semigroup C$^*$-Algebras},
   series={Oberwolfach Seminars},
   volume={47},
   publisher={Birkhäuser Cham},
   date={2017},
   pages={ix+319},
   isbn={2296-5041},
   doi={https://doi.org/10.1007/978-3-319-59915-1},
   review={\MR{3618901}},
}

\bib{MR2419901}{article}{
   author={Exel, Ruy},
   title={Inverse semigroups and combinatorial $C^\ast$-algebras},
   journal={Bulletin of the Brazilian Mathematical Society, New Series},
   volume={39},
   date={2008},
   number={2},
   pages={191--313},
   issn={1678-7544},
   review={\MR{2419901}},
   doi={10.1007/s00574-008-0080-7},
}

\bib{KKLL}{article}{
	author={Kakariadis, Evgenios},
	author={Katsoulis, Elias},
	author={Laca, Marcelo},
	author={Li, Xin},
	title={Boundary quotient C*‐algebras of semigroups},
	journal={Journal of the London Mathematical Society},
	volume={105},
	number={4},
	date={2022},
	month={02},
	pages={2136-2166},
	doi={10.1112/jlms.12557}
}

\bib{KhSk}{article}{
   author={Khoshkam, Mahmood},
   author={Skandalis, Georges},
   title={Regular representation of groupoid $C^*$-algebras and applications to inverse semigroups},
   journal={Journal Fur Die Reine Und Angewandte Mathematik},
   volume={546},
   date={2002},
   pages={47--72},
   doi={10.1515/crll.2002.045},
}

\bib{LS}{article}{
	author={Laca, Marcelo},
	author={Sehnem, Camila},
	title={Toeplitz algebras of semigroups},
	journal={Transactions of the American Mathematical 	Society},
	volume={375},
	date={2021},
	number={10},
	pages={7443--7507},
	issn={0002-9947},
	review={\MR{4491431}},
	doi={10.1090/tran/8743},
}

\bib{MR2900468}{article}{
   author={Li, Xin},
   title={Semigroup ${\rm C}^*$-algebras and amenability of semigroups},
   journal={Journal of Functional Analysis},
   volume={262},
   date={2012},
   number={10},
   pages={4302--4340},
   issn={0022-1236},
   review={\MR{2900468}},
   doi={10.1016/j.jfa.2012.02.020},
}

\bib{Li-3}{misc}{
      author={Li, Xin},
       title={Semigroup C$^*$-algebras},
       how={preprint},
       date={2017},
       eprint={\href{https://arxiv.org/pdf/1707.05940}{\texttt{arXiv:2110.04501 [math.OA]}}},
}

\bib{Li-1}{article}{
	author={Li, Xin}
	title={Left regular representations of Garside categories I. $C^*$-algebras and groupoids},
	journal={Glasgow Mathematical Journal},
	volume={65},
	date={2023}
	number={S1}
	pages={53--86},
	doi={10.1017/S0017089522000106}
}

\bib{NS}{article}{
	author={Neshveyev, Sergey},
	author={Schwartz, Gaute},
	title={Non-Hausdorff étale groupoids and $C^*$-algebras of left cancellative monoids},
	journal={Münster Journal of Mathematics},
	volume={16},
	date={2023},
	pages={147--175},
	doi={10.17879/51009604279}
}

\bib{MR3200323}{article}{
   author={Norling, Magnus Dahler},
   title={Inverse semigroup $C^*$-algebras associated with left cancellative semigroups},
   journal={Proceedings of the Edinburgh Mathematical Society},
   volume={57},
   date={2014},
   number={2},
   pages={533--564},
   issn={0013-0915},
   review={\MR{3200323}},
   doi={10.1017/S0013091513000540},
}

\bib{MR1724106}{book}{
   author={Paterson, Alan L. T.},
   title={Groupoids, inverse semigroups, and their operator algebras},
   series={Progress in Mathematics},
   volume={170},
   publisher={Birkh\"{a}user Boston, Inc., Boston, MA},
   date={1999},
   pages={xvi+274},
   isbn={0-8176-4051-7},
   review={\MR{1724106}},
   doi={10.1007/978-1-4612-1774-9},
}

\bib{R-3}{article}{
	author={Renault, Jean},
	title={The Fourier Algebra of a Measured Groupoid and 	Its Multipliers},
	journal={Journal of Functional Analysis},
	volume={145},
	date={1997},
	number={2},
	pages={455-490},
	issn={0022-1236},
	doi={https://doi.org/10.1006/jfan.1996.3039},
}

\bib{R-2}{article}{
	author={Renault, Jean},
	title={Topological Amenability is a Borel Property},
	journal={Mathematica Scandinavica},
	volume={117},
	date={2015},
	number={1},
	pages={5--30},
	doi={10.7146/math.scand.a-22235}
}

\bib{SY}{article}{
	author={Sims, Aidan},
	author={Yeend, Trent},
	title={C$^*$-algebras associated to product systems of	Hilbert bimodules},
	journal={Journal of Operator Theory},
	volume={64},
	publisher={Romanian Academy Inst. Math.},
	date={2010},
	issn={0379-4024},
	number={2},
	pages={349--376},
}

\bib{MR4151331}{article}{
   author={Spielberg, Jack},
   title={Groupoids and $C^*$-algebras for left cancellative small categories},
   journal={Indiana University Mathematics Journal},
   volume={69},
   date={2020},
   number={5},
   pages={1579--1626},
   issn={0022-2518},
   review={\MR{4151331}},
   doi={10.1512/iumj.2020.69.7969},
}

\bib{Star}{article}{
	author={Starling, Charles},
	title={Boundary quotients of $C^*$-algebras of right LCM 	semigroups},
	journal={Journal of Functional Analysis},
	volume={11},
	date={2015},
	number={11},
	pages={3326--3356},
	doi={10.1016/j.jfa.2015.01.001}
}

\end{biblist}
\end{bibdiv}
\end{document}